\documentclass{amsart}

\usepackage{latexsym,amsfonts,amssymb,enumerate,amscd}
\usepackage{amsmath}
\usepackage{amsthm}
\usepackage{pstricks}

\psset{linewidth=0.3pt,dimen=middle}
\psset{xunit=.70cm,yunit=0.70cm}
\psset{arrowsize=1pt 5,arrowlength=0.6,arrowinset=0.7}

\usepackage[all]{xy}
\SelectTips{cm}{}

\usepackage{graphicx}

\newcommand{\onen}{{\mathbf 1}_{n}}
\newcommand{\onenn}[1]{{\mathbf 1}_{#1}}

\newcommand{\onel}{{\mathbf 1}_{\lambda}}

\def\k{\Bbbk}

\newcommand\E{{\sf{E}}}
\newcommand\sE{{\cal{E}}}
\newcommand\F{{\sf{F}}}
\newcommand\sF{{\cal{F}}}

\newcommand{\U}{\dot{{\rm U}}}
\newcommand{\Ucat}{\cal{U}}

\newcommand{\UcatD}{\dot{\cal{U}}}

\newcommand{\xsum}[2]{
  \xy
  (0,.4)*{\sum};
  (0,3.7)*{\scs #2};
  (0,-2.9)*{\scs #1};
  \endxy
}

\newcommand{\refequal}[1]{\xy {\ar@{=}^{#1}
(-1,0)*{};(1,0)*{}};
\endxy}

\newcommand{\Hom}{{\rm Hom}}

\renewcommand{\to}{\rightarrow}
\newcommand{\maps}{\colon}

\newcommand{\End}{{\rm End}}

\newcommand{\rk}{{\rm rk}}

\newcommand{\Proj}{{\rm Proj}}

\newcommand{\scs}{\scriptstyle}

\def\l{\lambda}
\newcommand{\g}{\mathfrak{g}}

\def\Id{\mathrm{Id}}

\def\mf{\mathfrak}

\theoremstyle{plain}
\newtheorem{thm}{Theorem}[section]
\newtheorem{lem}[thm]{Lemma}
\newtheorem{prop}[thm]{Proposition}
\newtheorem{cor}[thm]{Corollary}

\theoremstyle{definition}
\newtheorem{defn}{Definition}[section]
\newtheorem{conj}{Conjecture}[section]

\theoremstyle{remark}
\newtheorem*{rem}{Remark}
\newtheorem{remark}[thm]{Remark}

\numberwithin{equation}{section}

\def\emph#1{{\sl #1\/}}

\let\phi=\varphi
\let\theta=\vartheta
\let\epsilon=\varepsilon

\usepackage{bbm}

\def\Z{{\mathbbm Z}}

\def\cal#1{\mathcal{#1}}%
\def\1{\mathbbm{1}}%
\def\nn{\notag}

\def\la{\langle}
\def\ra{\rangle}

\newcommand{\Uup}{
    \xy {\ar (0,-3)*{};(0,3)*{} };(1.5,0)*{};(-1.5,0)*{};\endxy}

\newcommand{\Udown}{
    \xy {\ar (0,3)*{};(0,-3)*{} };(1.5,0)*{};(-1.5,0)*{};\endxy}

\newcommand{\Uupdot}{
   \xy {\ar (0,-3)*{};(0,3)*{} };(0,0)*{\bullet};(1.5,0)*{};(-1.5,0)*{};\endxy}

\newcommand{\Ucupr}{\;\;
    \vcenter{\xy (-2,2)*{}; (2,2)*{} **\crv{(-2,-2) & (2,-2)}?(1)*\dir{>};
            (2,-3)*{};(-2,3)*{}; \endxy} \;\; }

\newcommand{\Ucupl}{\;\;
    \vcenter{\xy (2,2)*{}; (-2,2)*{} **\crv{(2,-2) & (-2,-2)}?(1)*\dir{>};
            (2,-3)*{};(-2,3)*{}; \endxy} \;\; }

\newcommand{\Ucapr}{\;\;
    \vcenter{\xy (-2,-1)*{}; (2,-1)*{} **\crv{(-2,3) & (2,3)}?(1)*\dir{>};
            (2,-3)*{};(-2,3)*{}; \endxy} \;\; }

\newcommand{\Ucapl}{\;\;
    \vcenter{\xy (2,-2)*{}; (-2,-2)*{} **\crv{(2,2) & (-2,2)}?(1)*\dir{>};
            (2,-3)*{};(-2,3)*{}; \endxy} \;\; }

\newcommand{\sUup}{
    \xy {\ar (0,-2)*{};(0,2)*{} };(1.5,0)*{};(-1.5,0)*{};\endxy}

\newcommand{\sUdown}{
    \xy {\ar (0,2)*{};(0,-2)*{} };(1.5,0)*{};(-1.5,0)*{};\endxy}

\newcommand{\sUupdot}{
   \xy {\ar (0,-2)*{};(0,2)*{} };(0,0)*{\scs \bullet};(1.5,0)*{};(-1.5,0)*{};\endxy}

\newcommand{\ccbub}[1]{
\xybox{%
 (-6,0)*{};
  (6,0)*{};
  (-4,0)*{}="t1";
  (4,0)*{}="t2";
  "t2";"t1" **\crv{(4,6) & (-4,6)};
  ?(.05)*\dir{>} ?(1)*\dir{>};
  "t2";"t1" **\crv{(4,-6) & (-4,-6)};
   ?(.25)*\dir{}+(0,0)*{\bullet}+(0,-3)*{\scs {#1}};
}}

\newcommand{\iccbub}[2]{
\xybox{%
 (-6,0)*{};
  (6,0)*{};
  (-4,0)*{}="t1";
  (4,0)*{}="t2";
  "t2";"t1" **\crv{(4,6) & (-4,6)}; ?(.7)*\dir{}+(-2,0)*{\scs #2}
  ?(.05)*\dir{>} ?(1)*\dir{>};
  "t2";"t1" **\crv{(4,-6) & (-4,-6)};
   ?(.3)*\dir{}+(0,0)*{\bullet}+(0,-3)*{\scs {#1}};
}}
\newcommand{\icbub}[2]{
\xybox{%
 (-6,0)*{};
  (6,0)*{};
  (-4,0)*{}="t1";
  (4,0)*{}="t2";
  "t2";"t1" **\crv{(4,6) & (-4,6)};?(.7)*\dir{}+(-2,0)*{\scs #2};
   ?(0)*\dir{<} ?(.95)*\dir{<};
  "t2";"t1" **\crv{(4,-6) & (-4,-6)};
   ?(.3)*\dir{}+(0,0)*{\bullet}+(0,-3)*{\scs {#1}};
}}

\newcommand{\cbub}[1]{
\xybox{%
 (-6,0)*{};
  (6,0)*{};
  (-4,0)*{}="t1";
  (4,0)*{}="t2";
  "t2";"t1" **\crv{(4,6) & (-4,6)};
   ?(0)*\dir{<} ?(.95)*\dir{<};
  "t2";"t1" **\crv{(4,-6) & (-4,-6)};
   ?(.75)*\dir{}+(0,0)*{\bullet}+(0,-3)*{\scs {#1}};
}}
\newcommand{\bbe}[1]{\xybox{%
  (-2,0)*{};
  (2,0)*{};
  (0,0);(0,-18) **\dir{-}; ?(.5)*\dir{<}+(2.3,0)*{\scriptstyle{#1}};
}}

\newcommand{\bbpef}[1]{\xybox{%
  (-6,0)*{};
  (6,0)*{};
  (-4,0)*{}="t1";
  (4,0)*{}="t2";
  "t1";"t2" **\crv{(-4,-6) & (4,-6)}; ?(.15)*\dir{>} ?(.9)*\dir{>}
    ?(.5)*\dir{}+(0,-2)*{\scriptstyle{#1}};
}}
\newcommand{\bbpfe}[1]{\xybox{%
  (-6,0)*{};
  (6,0)*{};
  (-4,0)*{}="t1";
  (4,0)*{}="t2";
  "t2";"t1" **\crv{(4,-6) & (-4,-6)}; ?(.15)*\dir{>} ?(.9)*\dir{>} ?(.5)*\dir{}+(0,-2)*{\scriptstyle{#1}};
}}

\newcommand{\bbcfe}[1]{\xybox{%
  (-6,0)*{};
  (6,0)*{};
  (-4,0)*{}="t1";
  (4,0)*{}="t2";
  "t1";"t2" **\crv{(-4,6) & (4,6)}; ?(.15)*\dir{>} ?(.9)*\dir{>}
  ?(.5)*\dir{}+(0,2)*{\scriptstyle{#1}};
}}
\newcommand{\bbcef}[1]{\xybox{%
  (-6,0)*{};
  (6,0)*{};
  (-4,0)*{}="t1";
  (4,0)*{}="t2";
  "t2";"t1" **\crv{(4,6) & (-4,6)}; ?(.15)*\dir{>}
  ?(.9)*\dir{>} ?(.5)*\dir{}+(0,2)*{\scriptstyle{#1}};
}}

\newcommand{\lowrru}[1]{\xybox{%
  (-8,0)*{};
  (8,0)*{};
  (-6,-18)*{};(6,-9)*{} **\crv{(-6,-13) & (6,-15)} ?(1)*\dir{>};
  (6,-9)*{};(6,0)*{}  **\dir{-} ?(.3)*\dir{ }+(2,0)*{\scs {\bf j}};
}}

\newcommand{\lowllu}[1]{\xybox{%
  (-8,0)*{};
  (8,0)*{};
  (6,-18)*{};(-6,-9)*{} **\crv{(6,-13) & (-6,-15)} ?(1)*\dir{>};
  (-6,-9)*{};(-6,0)*{}  **\dir{-} ?(.3)*\dir{ }+(-2,0)*{\scs {\bf j}};
}}

\newcommand{\nccbub}{
\xybox{%
 (-6,0)*{};
  (6,0)*{};
  (-4,0)*{}="t1";
  (4,0)*{}="t2";
  "t2";"t1" **\crv{(4,6) & (-4,6)}; ?(.05)*\dir{>} ?(1)*\dir{>};
  "t2";"t1" **\crv{(4,-6) & (-4,-6)}; ?(.3)*\dir{};
}}
\newcommand{\ncbub}{
\xybox{%
 (-6,0)*{};
  (6,0)*{};
  (-4,0)*{}="t1";
  (4,0)*{}="t2";
  "t2";"t1" **\crv{(4,6) & (-4,6)}; ?(0)*\dir{<} ?(.95)*\dir{<};
  "t2";"t1" **\crv{(4,-6) & (-4,-6)}; ?(.3)*\dir{};
}}

\newcommand{\bbdl}[1]{\xybox{%
  (2,0);(0,-8) **\crv{(2,-2)&(0,-6)}; ?(.5)*\dir{>}
}}
\newcommand{\bbdlu}[1]{\xybox{%
  (2,0);(0,-8) **\crv{(2,-2)&(0,-6)}; ?(.5)*\dir{<}
}}
\newcommand{\bbdr}[1]{\xybox{%
  (-2,0);(0,-8) **\crv{(-2,-2)&(0,-6)}; ?(.5)*\dir{>}
}}
\newcommand{\bbdru}[1]{\xybox{%
  (-2,0);(0,-8) **\crv{(-2,-2)&(0,-6)}; ?(.5)*\dir{<}
}}

\newgray{whitegray}{.9}

\newcommand{\chern}[1]{\begin{pspicture}(-0.5,-0.5)(0.5,0.5)
    \rput(0,0){\psframebox[framearc=.4,fillstyle=solid, linewidth=.8pt]{\small $\scriptstyle #1$}} \end{pspicture}
}

\newcommand{\sccbub}{%
\xybox{%
 (-6,0)*{};
  (6,0)*{};
  (-4,0)*{}="t1";
  (4,0)*{}="t2";
  "t2";"t1" **\crv{(4,6) & (-4,6)}; ?(.05)*\dir{>} ?(1)*\dir{>};
  "t2";"t1" **\crv{(4,-6) & (-4,-6)};
}}
\newcommand{\scbub}{%
\xybox{%
 (-6,0)*{};
  (6,0)*{};
  (-4,0)*{}="t1";
  (4,0)*{}="t2";
  "t2";"t1" **\crv{(4,6) & (-4,6)}; ?(0)*\dir{<} ?(.95)*\dir{<};
  "t2";"t1" **\crv{(4,-6) & (-4,-6)};
}}

\allowdisplaybreaks[1]

\title{Implicit structure in 2-representations of quantum groups}

\author{Sabin Cautis}
\address{Department of Mathematics, University of Southern California, Los Angeles, CA 90089, USA}
\email{cautis@usc.edu}

\author{Aaron D. Lauda}
\address{Department of Mathematics, University of Southern California, Los Angeles, CA 90089, USA}
\email{lauda@usc.edu}

\begin{document}
%

\begin{abstract}
Given a strong 2-representation of a Kac-Moody Lie algebra (in the sense of Rouquier) we show how to extend it to a 2-representation of categorified quantum groups (in the sense of Khovanov-Lauda). This involves checking certain extra 2-relations which are explicit in the definition by Khovanov-Lauda and, as it turns out, implicit in Rouquier's definition. Some applications are also discussed. 
\end{abstract}


\maketitle

%
\section{Introduction}
%

Higher representation theory studies actions of Lie algebras, such as ${\mathfrak{sl}}_2$, on categories rather than vector spaces. Instead of linear maps between vector spaces, such as $E$ and $F$, there are functors $\E$ and $\F$ between categories. Subsequently, the work of Chuang and Rouquier~\cite{CR} revealed the importance of studying natural transformations between these functors.  These natural transformations are invisible in the usual representation theory.

New structure in higher representation theory arises via the 2-relations -- new relations holding between natural transformations. All of this data, namely the functors and the natural transformations, is captured and formalized in the concept of a categorified quantum group. These are 2-categories which were introduced and studied by Khovanov-Lauda in \cite{Lau1,KL,KL2,KL3}. Closely related 2-categories were independently introduced by Rouquier in \cite{Rou2}. In these 2-categories, the functors correspond to 1-morphisms while the natural transformations correspond to 2-morphisms.

One obvious difference between the approaches of Khovanov-Lauda and Rouquier is certain extra 2-relations which occur in \cite{KL,KL2,KL3} but not in \cite{Rou2}. These relations induce explicit 2-isomorphisms lifting certain equalities in the quantum group. The Khovanov-Lauda approach uses a diagrammatic framework where 2-morphisms are depicted diagrammatically using string-like diagrams (as explained later). When interpreted diagrammatically, the extra 2-relations in the Khovanov-Lauda approach tell you how to simplify an arbitrary string diagram (depicting a 2-morphism). Being able to simplify such diagrams can also be useful in other contexts. For example, in \cite{C} these 2-relations were used to compute various knot homologies of torus knots.

The main result of this paper (Theorem \ref{thm:main}) shows that these extra relations are essentially forced upon us. By adding the requirement that there are no negative degree endomorphisms of the identity 1-morphism we show that a 2-representation of 2-Kac-Moody algebras in the sense of Rouquier extends to a 2-representation of categorified quantum groups in the sense of Khovanov-Lauda.  In the process of proving this we also clarify the definition of these categorified quantum groups by defining a 2-category $\Ucat_Q(\mf{g})$ for arbitrary Kac-Moody Lie algebra $\mathfrak{g}$ (and an arbitrary choice of scalars $Q$ in the definition of the KLR algebra). These definitions differ slightly from those currently in the literature but, as mentioned above, are forced upon us.

We also show that biadjointness between the functors $\E$ and $\F$ is a formal consequence of the other conditions (Proposition \ref{prop:lradj}). In other words, one gets for free that the left and right adjoints of $\E$ are isomorphic (up to a grading shift). This was also noticed by Rouquier in \cite{Rou2}.

We finish by briefly discussing some applications (Section~\ref{sec_apps}). In particular, it was conjectured in \cite{KL3} that the 2-category $\Ucat_Q(\mf{g})$ should admit various 2-representations. Theorem \ref{thm:main} resolves one of these conjectures by allowing us to deduce the existence of a 2-representation of the 2-category $\Ucat_Q(\mf{g})$ on derived categories of coherent sheaves on cotangent bundles to Grassmannians.  While the added condition that there are no negative degree endomorphisms of the identity 1-morphisms is a natural condition to verify in most geometric settings, checking this condition in algebraic settings such as categories of projective modules over cyclotomic quotients of KLR algebras can be more challenging.

\medskip

In the remaining part of the introduction we sketch the definition of the 2-category $\Ucat_Q(\mf{g})$ and what it means to have a $Q$-strong 2-representation of $\g$. A $Q$-strong 2-representation of $\g$ is essentially a 2-Kac-Moody representation in the sense of Rouquier with one additional condition that there are no negative degree endomorphisms of the identity 1-morphisms $\1_\l$. Our main result (Theorem \ref{thm:main}) states that a $Q$-strong 2-representation of $\g$ induces a 2-representation of $\Ucat_Q(\mf{g})$.

This article provides a comparison between the 2-representation theories of the 2-categories introduced by Rouquier and that of the 2-category $\Ucat_Q(\mf{g})$. In practice, one checks the $Q$-strong 2-representation conditions, which are simpler and easier to verify, and uses Theorem \ref{thm:main} to define an action of the 2-category $\Ucat_Q(\mf{g})$. The 2-category $\Ucat_Q(\mf{g})$ is more elaborate and includes more relations telling you how to simplify any string diagram.

In \cite{Rou2} the notion of a 2-Kac-Moody representation is formalized by introducing several 2-categories associated to Kac-Moody algebras. Isomorphisms lifting quantum Serre relations are obtained via a localization procedure on certain 1-morphisms in these 2-categories. A 2-Kac-Moody representation is then defined as a 2-functor from one of these 2-categories to an appropriate 2-category $\cal{K}$. The results in this article are insufficient to deduce that any version of the 2-categories Rouquier defines has Grothendieck ring isomorphic to the quantum group. This is because we do not know how to verify the condition that there are no negative degree endomorphisms of $\1_\l$ (which is required for Theorem~\ref{thm:main} to hold) from the localization procedure used in Rouquier's approach.

\bigskip

%
\subsection{The 2-category $\Ucat_Q(\mf{g})$}
%

To a Cartan datum and a choice of scalars $Q$ (both defined in section~\ref{sec:datum}) we define a 2-category $\Ucat_Q(\mf{g})$. Note that the 2-category $\Ucat(\mf{g})$ from \cite{KL3} corresponds to the choice of scalars with  $t_{ij}=t_{ji}=1$ and $s_{ij}^{pq}=0$ for all $i,j \in I$.

\begin{defn} \label{defU_cat}
The 2-category $\Ucat_Q(\mf{g})$ is the graded additive $\Bbbk$-linear 2-category(see section~\ref{sec:categories}) consisting of:
\begin{itemize}
\item objects $\lambda$ for $\lambda \in X$.
\item 1-morphisms are formal direct sums of compositions of
$$\onel, \quad \onenn{\l+\alpha_i} \sE_i = \onenn{\l+\alpha_i} \sE_i\onel = \sE_i \onel \quad \text{ and }\quad \onenn{\lambda-\alpha_i} \sF_i = \onenn{\lambda-\alpha_i} \sF_i\onel = \sF_i\onel.$$
for $i \in I$ and $\l \in X$.
\item 2-morphisms are $\k$-vector spaces spanned by compositions of (decorated) tangle-like diagrams illustrated below.
\begin{align}
  \xy 0;/r.17pc/:
 (0,7);(0,-7); **\dir{-} ?(.75)*\dir{>};
 (0,0)*{\bullet};
 (7,3)*{ \scs \lambda};
 (-9,3)*{\scs  \lambda+\alpha_i};
 (-2.5,-6)*{\scs i};
 (-10,0)*{};(10,0)*{};
 \endxy &\maps \cal{E}_i\onel \to \cal{E}\onel\la (\alpha_i,\alpha_i) \ra  & \quad
 &
    \xy 0;/r.17pc/:
 (0,7);(0,-7); **\dir{-} ?(.75)*\dir{<};
 (0,0)*{\bullet};
 (7,3)*{ \scs \lambda};
 (-9,3)*{\scs  \lambda-\alpha_i};
 (-2.5,-6)*{\scs i};
 (-10,0)*{};(10,0)*{};
 \endxy\maps \cal{F}_i\onel \to \cal{F}_i\onel\la (\alpha_i,\alpha_i) \ra  \nn \\
   & & & \nn \\
   \xy 0;/r.17pc/:
  (0,0)*{\xybox{
    (-4,-4)*{};(4,4)*{} **\crv{(-4,-1) & (4,1)}?(1)*\dir{>} ;
    (4,-4)*{};(-4,4)*{} **\crv{(4,-1) & (-4,1)}?(1)*\dir{>};
    (-5.5,-3)*{\scs i};
     (5.5,-3)*{\scs j};
     (9,1)*{\scs  \lambda};
     (-10,0)*{};(10,0)*{};
     }};
  \endxy \;\;&\maps \cal{E}_i\cal{E}_j\onel  \to \cal{E}_j\cal{E}_i\onel\la -(\alpha_i,\alpha_j) \ra  &
  &
   \xy 0;/r.17pc/:
  (0,0)*{\xybox{
    (-4,4)*{};(4,-4)*{} **\crv{(-4,1) & (4,-1)}?(1)*\dir{>} ;
    (4,4)*{};(-4,-4)*{} **\crv{(4,1) & (-4,-1)}?(1)*\dir{>};
    (-6.5,-3)*{\scs i};
     (6.5,-3)*{\scs j};
     (9,1)*{\scs  \lambda};
     (-10,0)*{};(10,0)*{};
     }};
  \endxy\;\; \maps \cal{F}_i\cal{F}_j\onel  \to \cal{F}_j\cal{F}_i\onel\la -(\alpha_i,\alpha_j) \ra  \nn \\
  & & & \nn \\
     \xy 0;/r.17pc/:
    (0,-3)*{\bbpef{i}};
    (8,-5)*{\scs  \lambda};
    (-10,0)*{};(10,0)*{};
    \endxy &\maps \onel  \to \cal{F}_i\cal{E}_i\onel\la d_i+(\l, \alpha_i) \ra   &
    &
   \xy 0;/r.17pc/:
    (0,-3)*{\bbpfe{i}};
    (8,-5)*{\scs \lambda};
    (-10,0)*{};(10,0)*{};
    \endxy \maps \onel  \to\cal{E}_i\cal{F}_i\onel\la d_i-(\l, \alpha_i) \ra  \nn \\
      & & & \nn \\
  \xy 0;/r.17pc/:
    (0,0)*{\bbcef{i}};
    (8,4)*{\scs  \lambda};
    (-10,0)*{};(10,0)*{};
    \endxy & \maps \cal{F}_i\cal{E}_i\onel \to\onel\la d_i+(\l, \alpha_i) \ra  &
    &
 \xy 0;/r.17pc/:
    (0,0)*{\bbcfe{i}};
    (8,4)*{\scs  \lambda};
    (-10,0)*{};(10,0)*{};
    \endxy \maps\cal{E}_i\cal{F}_i\onel  \to\onel\la d_i-(\l, \alpha_i) \ra \nn
\end{align}
\end{itemize}
\end{defn}
Here we follow the conventions in \cite{Lau1} and read diagrams from right to left and bottom to top.  The identity 2-morphism of the 1-morphism $\cal{E}_i \onel$ is represented by an upward-oriented line labeled by $i$ (likewise, the identity 2-morphism of $\cal{F}_i \onel$ is represented by a downward-oriented line labeled $i$).

The 2-morphisms satisfy the following relations:
\begin{enumerate}
\item \label{item_cycbiadjoint} The 1-morphisms $\cal{E}_i \onel$ and $\cal{F}_i \onel$ are biadjoint (up to a specified degree shift). Moreover, the 2-morphisms are $Q$-cyclic with respect to this biadjoint structure (see section \ref{sec:cycbiadjoint}).

\item The $\cal{E}$'s carry an action of the KLR algebra\footnote{The KLR algebra is also called the quiver Hecke algebra in the literature.} associated to $Q$, while the $\cal{F}$'s carry an action of the KLR algebra associated to $Q'$, where $Q'$ is determined from $Q$ by a simple procedure (see sections \ref{sec:datum} and \ref{sec:KLR}).

\item When $i \ne j$ one has the mixed relations (see section \ref{sec:mixedrels}) relating $\cal{E}_i \cal{F}_j$ and $\cal{F}_j \cal{E}_i$.

\item \label{item_positivity} Dotted bubbles of negative degree are zero while dotted bubbles of degree zero are equal to the identity (see section \ref{sec:dotbubbles}).

\item \label{item_highersl2} For any $i \in I$ one has the extended ${\mf{sl}}_2$ relations (see section \ref{sec:highersl2}).
\end{enumerate}

The 2-category $\UcatD_Q(\mf{g})$ is the idempotent completion of the 2-category $\Ucat_Q(\mf{g})$.

\begin{rem}
Note that the notion of $Q$-cyclic biadjointness, and the precise form of the mixed relations in the 2-category $\Ucat_Q(\mf{g})$, have not appeared in the literature before.
\end{rem}
%
\subsection{A $Q$-strong 2-representation of $\g$}\label{sec:strongcat}
%

\begin{defn} \label{def_Qstrong}
A $Q$-strong 2-representation of $\g$ consists of a graded, additive $\Bbbk$-linear idempotent complete 2-category $\cal{K}$ where:
\begin{itemize}
\item The objects of $\cal{K}$ are indexed by the weights $\l \in X$.
\item There are identity 1-morphisms $\1_\l$ for each $\l$, as well as 1-morphisms $\E_i \1_{\l}: \lambda \rightarrow\l + \alpha_i$ in $\cal{K}$. We also assume that $\E_i \1_\l$ has both right and left adjoints and define the 1-morphism $\1_\l \F_i: \l + \alpha_i \rightarrow \l$ as
\begin{equation} \label{eq_defF}
  \1_\l \F_i := (\E_i \1_\l)_R \la - (\alpha_i, \l) - d_i \ra ,
\end{equation}
where $d_i = (\alpha_i, \alpha_i)/2$. Note that, because our adjoints are not yet fixed, the 1-morphism $\1_\l \F_i$ is only defined up to non-canonical isomorphism.
See section \ref{sec:datum} for more details on the notation above.
\end{itemize}

On this data we impose the following conditions:
\begin{enumerate}
\item (Integrability) The identity 1-morphism $\1_{\l+r\alpha_i}$ of the object $\l+r\alpha_i$ is isomorphic to the zero 1-morphism for $r \gg 0$ or $r \ll 0$.
\item \label{co:hom} $\Hom_{\cal{K}}(\1_{\l}, \1_{\l} \la l \ra)$ is zero if $l < 0$ and one-dimensional if $l=0$. Moreover, the space of 2-morphisms between any two 1-morphisms is finite dimensional.

\item \label{co:EF} We have the following isomorphisms in $\cal{K}$:
$$\F_i \E_i \1_{\l} \cong \E_i \F_i \1_{\l} \oplus_{[- \la i,\l \ra]_i} \1_\l \text{ if } \la i,\l \ra \le 0$$
$$\E_i \F_i \1_{\l} \cong \F_i \E_i \1_{\l} \oplus_{[\la i,\l \ra]_i} \1_\l \text{ if } \la i,\l \ra \ge 0,$$
where the direct sum over a Laurent polynomial is defined in section~\ref{sec:categories}.
\item \label{co:KLR} The $\E$'s carry an action of the KLR algebra associated to $Q$.
\item \label{co:EiFj} If $i \ne j \in I$ then $\F_j \E_i \1_{\l} \cong \E_i \F_j \1_{\l}$ in $\cal{K}$.
\end{enumerate}
\end{defn}

When $Q$ is clear from the context we will simply call this a strong 2-representation of $\mf{g}$. Note that we do not require an action of the KLR algebra on the $\F$'s. The existence of such an action will follow formally from the action on the $\E$'s.

{\bf{Important convention.}} The integrability condition above implies that ``most'' objects are isomorphic to the zero object  (by which we mean, their identity 1-morphisms are isomorphic to the zero 1-morphism). If $\l$ is the zero object then, by definition, $\Hom_{\cal{K}}(\1_\l,\1_\l \la l \ra) = 0$ for all $l$. So, to be precise, condition (\ref{co:hom}) above should say that $\Hom_{\cal{K}}(\1_\l,\1_\l)$ is one-dimensional if $\l$ is non-zero. There are many other such instances later in this paper. The convention is that any statement about a certain Hom being non-zero (for example, the claims in lemma \ref{lem:1} or corollaries \ref{cor:0} and \ref{cor:2}) assumes that all weights involved are non-zero (otherwise the Hom space is automatically zero).

\begin{rem} \hfill
\begin{enumerate}
 \item Our definition of a $Q$-strong 2-representation of $\g$ is nearly identical to Rouquier's definition of a 2-Kac-Moody representation. The exception is condition (2) which is used repeatedly later in this article.

  \item The integrability condition is only used to show that the functors $\E$s and $\F$s are biadjoint and that the 1-morphisms $\E_i$ have no negative degree endomorphisms. If one adds these conditions to the definition then one can drop the integrability condition. Usually this is a poor compromise since checking biadjointness can be difficult. However, there are situations where checking biadjointness and the vanishing of negative degree 2-morphisms is easy but where the integrability condition fails.

 \item The condition about finite dimensional Homs is used to conclude that the categories have the Krull-Schmidt property. This is only used to define the concept of rank of a map which is subsequently only used in Lemma~\ref{lem:Xind}.
\end{enumerate}
\end{rem}

In \cite[Definition 9.3]{Lau1} it is shown that whenever the KLR algebra acts on 1-morphisms in an additive $\Bbbk$-linear idempotent complete 2-category $\cal{K}$ one can define divided powers $\E^{(a)}_i\1_{\l}$ as the image of certain idempotents (with an appropriate shift). So the existence of these divided powers are a consequence of our definition, rather than part of it.

%
\subsection{Main results}
%

\begin{defn}
A 2-representation of $\UcatD_Q(\mf{g})$ is a graded additive $\Bbbk$-linear 2-functor $\UcatD_Q(\mf{g})\to \cal{K}$ for some graded, additive 2-category $\cal{K}$.
\end{defn}

When all of the Hom categories $\cal{K}(x,y)$ between objects $x$ and $y$ of $\cal{K}$ are idempotent complete, in other words {$\cal{K}$ is isomorphic to its Karoubi completion $\dot{\cal{K}}$}, then any graded additive $\Bbbk$-linear 2-functor $\Ucat_{Q}(\mf{g}) \to \cal{K}$ extends uniquely to a 2-representation of $\UcatD_Q(\mf{g})$ (see section~\ref{sec:categories}).

\begin{thm}\label{thm:main}
Any $Q$-strong 2-representation of $\g$ on $\cal{K}$ extends to a 2-representation $\UcatD_Q(\mf{g}) \to \cal{K}$.
\end{thm}

\noindent {\bf Acknowledgments:}
The authors would like to thank Mikhail Khovanov, Anthony Licata, Joshua Sussan and Ben Webster for helpful discussions and Masaki Kashiwara for helpful correspondences. S.C. was supported by NSF grants DMS-0964439, DMS-1101439 and the Alfred P. Sloan foundation. A.L. was supported by the NSF grants DMS-0739392, DMS-0855713, and the Alfred P. Sloan foundation.

%
\section{The full definition of $\UcatD_Q(\mf{g})$}
%

%
\subsection{Some conventions}\label{sec:conventions}
%

Fix a base field $\Bbbk$. We do not assume $\Bbbk$ to be of characteristic 0, nor algebraically closed.

%
\subsubsection{The Cartan datum and choice of scalars $Q$}\label{sec:datum}
%

We fix a Cartan datum consisting of
\begin{itemize}
\item a free $\Z$-module $X$ (the weight lattice),
\item for $i \in I$ ($I$ is an indexing set) there are elements $\alpha_i \in X$ (simple roots) and $\Lambda_i \in X$ (fundamental weights),
\item for $i \in I$ an element $h_i \in X^\vee = \Hom_{\Z}(X,\Z)$ (simple coroots),
\item a bilinear form $(\cdot,\cdot )$ on $X$.
\end{itemize}
Write $\langle \cdot, \cdot \rangle \maps X^{\vee} \times X
\to \Z$ for the canonical pairing. These data should satisfy:
\begin{itemize}
\item $(\alpha_i, \alpha_i) \in 2\Z_{>0}$ for any $i\in I$,
\item $\la i,\lambda\ra :=\langle h_i, \lambda \rangle = 2 \frac{(\alpha_i,\lambda)}{(\alpha_i,\alpha_i)}$
  for $i \in I$ and $\lambda \in X$,
\item $(\alpha_i,\alpha_j) \leq 0$  for $i,j\in I$ with $i \neq j$,
\item $\langle h_j, \Lambda_i \rangle =\delta_{ij}$ for all $i,j \in I$.
\end{itemize}
Hence $C_{i,j}=\{\langle h_i, \alpha_j \rangle\}_{i,j\in I}$ is a symmetrizable generalized Cartan matrix. In what follows we write
\begin{equation}
 d_{ij} =-\langle i, \alpha_j \rangle
\end{equation}
for $i,j\in I$. We denote by $X^+ \subset X$ the dominant weights which are of the form $\sum_i \lambda_i \Lambda_i$ where $\lambda_i \ge 0$.

Write $d_i=\frac{(\alpha_i,\alpha_i)}{2}$ and let $q_i=q^{d_i}$, $[n]_i=q_i^{n-1}+q_i^{n-3}+\dots + q_i^{1-n}$,
$[n]_i!=[n]_i[n-1]_i\dots [1]_i$.
\medskip

Associated to a Cartan datum we also fix a choice of scalars $Q$ consisting of
\begin{itemize}
  \item $t_{ij}$ for all $i,j \in I$ ,
  \item $s_{ij}^{pq}\in \Bbbk$ for $i \neq j$,  $0 \leq p < d_{ij}$, and $0 \leq q < d_{ji}$,
  \item $r_i\in \Bbbk^{\times}$ for all $i \in I$,
\end{itemize}
such that
\begin{itemize}
\item $t_{ii}=1$ for all $i \in I$ and $t_{ij} \in \Bbbk^{\times}$ for $i\neq j$,
 \item $s_{ij}^{pq}=s_{ji}^{qp}$,
 \item $t_{ij}=t_{ji}$ when $d_{ij}=0$.
\end{itemize}
We set $s_{ij}^{pq}=0$ when $p,q<0$ or $d_{ij} \geq p$ or $d_{ji} \geq q$.

Given a set of scalars $Q$ we denote by $Q'$ another set of scalars given by
$$r'_{i}=-r_{i}, \hspace{.5cm} t'_{ij}=t_{ji}^{-1} \quad \text{$i \neq j$,} \hspace{.5cm} \text{ and } \hspace{.5cm} (s')_{ij}^{pq}=t_{ij}^{-1}t_{ji}^{-1}s_{ij}^{pq}.$$

%
\subsubsection{2-categories and idempotent completion}\label{sec:categories}
%

By a graded category we will mean a category equipped with an auto-equivalence $\la 1 \ra$. We denote by $\la l \ra$ the auto-equivalence obtained by applying $\la 1 \ra$ $l$ times. If $A,B$ are two objects then $\Hom^l(A,B)$ will be short-hand for $\Hom(A,B \la l \ra)$.
A graded additive $\Bbbk$-linear 2-category is a category enriched over graded additive $\Bbbk$-linear categories, that is, a 2-category $\cal{K}$ such that the Hom categories $\Hom_{\cal{K}}(A,B)$ between objects $A$ and $B$ are graded additive $\Bbbk$-linear categories and the composition maps $\Hom_{\cal{K}}(A,B) \times \Hom_{\cal{K}}(B,C) \to \Hom_{\cal{K}}(A,C)$ form graded additive $\Bbbk$-linear functor.

A graded additive $\Bbbk$-linear 2-functor $F \maps \cal{K}\to \cal{K}'$ is a (weak) 2-functor that maps the Hom categories $\Hom_{\cal{K}}(A,B)$ to $\Hom_{\cal{K}'}(FA,FB)$ by additive functors that commute with the auto-equivalence $\la 1 \ra$.

Given a 1-morphism $A$ in an additive 2-category $\cal{K}$ and a Laurent polynomial $f=\sum f_a q^a\in \Z[q,q^{-1}]$ we write $\oplus_f A$ or $A^{\oplus f}$ for the direct sum over $a\in \Z$ of $f_a$ copies of $A\la a\ra$.  In particular, if $f = [n]_i$, then we write $\oplus_{[n]_i} A$ to denote the direct sum $\oplus_{k=0}^{n-1} A \la d_i(n-1-2k) \ra$.

An additive category $\cal{C}$ is said to be idempotent complete when every idempotent 1-morphism splits in $\cal{C}$. We say that the additive 2-category $\cal{K}$ is idempotent complete when the Hom categories $\Hom_{\cal{K}}(A,B)$ are idempotent complete for any pair of objects $A, B$ of $\cal{K}$, so that all idempotent 2-morphisms split. The idempotent completion, or Karoubi envelope $\dot{\cal{C}}$, of an additive category $\cal{C}$ can be viewed as a minimal enlargement of the category $\cal{C}$ so that idempotents split. The idempotent completion $\dot{\cal{K}}$ of a 2-category $\cal{K}$ is the 2-category with the same objects as $\cal{K}$, but with Hom categories given by the usual Karoubi envelope of $\Hom_{\cal{K}}(A,B)$. Any additive 2-functor $\cal{K} \to \cal{K}'$ that has splitting of idempotent 2-morphisms in $\cal{K}'$ extends uniquely to an additive 2-functor $\dot{\cal{K}} \to \cal{K'}$; see \cite[Section 3.4]{KL3} for more details.

With the above conventions in place, we now proceed to fill in the missing details in Definition \ref{defU_cat}, by describing all the 2-relations diagrammatically.

%
\subsection{$Q$-cyclic biadjointness}\label{sec:cycbiadjoint}
%

Diagrammatically, biadjointness corresponds to the following equalities of diagrams:
\begin{equation} \label{eq_biadjoint1}
 \xy   0;/r.17pc/:
    (-8,0)*{}="1";
    (0,0)*{}="2";
    (8,0)*{}="3";
    (-8,-10);"1" **\dir{-};
    "1";"2" **\crv{(-8,8) & (0,8)} ?(0)*\dir{>} ?(1)*\dir{>};
    "2";"3" **\crv{(0,-8) & (8,-8)}?(1)*\dir{>};
    "3"; (8,10) **\dir{-};
    (12,-9)*{\lambda};
    (-6,9)*{\lambda+\alpha_i};
    \endxy
    \; =
    \;
\xy   0;/r.17pc/:
    (-8,0)*{}="1";
    (0,0)*{}="2";
    (8,0)*{}="3";
    (0,-10);(0,10)**\dir{-} ?(.5)*\dir{>};
    (5,8)*{\lambda};
    (-9,8)*{\lambda+\alpha_i};
    \endxy
\qquad \quad \xy  0;/r.17pc/:
    (8,0)*{}="1";
    (0,0)*{}="2";
    (-8,0)*{}="3";
    (8,-10);"1" **\dir{-};
    "1";"2" **\crv{(8,8) & (0,8)} ?(0)*\dir{<} ?(1)*\dir{<};
    "2";"3" **\crv{(0,-8) & (-8,-8)}?(1)*\dir{<};
    "3"; (-8,10) **\dir{-};
    (12,9)*{\lambda+\alpha_i};
    (-6,-9)*{\lambda};
    \endxy
    \; =
    \;
\xy  0;/r.17pc/:
    (8,0)*{}="1";
    (0,0)*{}="2";
    (-8,0)*{}="3";
    (0,-10);(0,10)**\dir{-} ?(.5)*\dir{<};
    (9,-8)*{\lambda+\alpha_i};
    (-6,-8)*{\lambda};
    \endxy
\end{equation}

\begin{equation}\label{eq_biadjoint2}
 \xy   0;/r.17pc/:
    (8,0)*{}="1";
    (0,0)*{}="2";
    (-8,0)*{}="3";
    (8,-10);"1" **\dir{-};
    "1";"2" **\crv{(8,8) & (0,8)} ?(0)*\dir{>} ?(1)*\dir{>};
    "2";"3" **\crv{(0,-8) & (-8,-8)}?(1)*\dir{>};
    "3"; (-8,10) **\dir{-};
    (12,9)*{\lambda};
    (-5,-9)*{\lambda+\alpha_i};
    \endxy
    \; =
    \;
    \xy 0;/r.17pc/:
    (8,0)*{}="1";
    (0,0)*{}="2";
    (-8,0)*{}="3";
    (0,-10);(0,10)**\dir{-} ?(.5)*\dir{>};
    (5,-8)*{\lambda};
    (-9,-8)*{\lambda+\alpha_i};
    \endxy
\qquad \quad \xy   0;/r.17pc/:
    (-8,0)*{}="1";
    (0,0)*{}="2";
    (8,0)*{}="3";
    (-8,-10);"1" **\dir{-};
    "1";"2" **\crv{(-8,8) & (0,8)} ?(0)*\dir{<} ?(1)*\dir{<};
    "2";"3" **\crv{(0,-8) & (8,-8)}?(1)*\dir{<};
    "3"; (8,10) **\dir{-};
    (12,-9)*{\lambda+\alpha_i};
    (-6,9)*{\lambda};
    \endxy
    \; =
    \;
\xy   0;/r.17pc/:
    (-8,0)*{}="1";
    (0,0)*{}="2";
    (8,0)*{}="3";
    (0,-10);(0,10)**\dir{-} ?(.5)*\dir{<};
   (9,8)*{\lambda+\alpha_i};
    (-6,8)*{\lambda};
    \endxy
\end{equation}

The $Q$-cyclic condition for dots is equivalent to:
\begin{equation} \label{eq_cyclic_dot}
 \xy 0;/r.17pc/:
    (-8,5)*{}="1";
    (0,5)*{}="2";
    (0,-5)*{}="2'";
    (8,-5)*{}="3";
    (-8,-10);"1" **\dir{-};
    "2";"2'" **\dir{-} ?(.5)*\dir{<};
    "1";"2" **\crv{(-8,12) & (0,12)} ?(0)*\dir{<};
    "2'";"3" **\crv{(0,-12) & (8,-12)}?(1)*\dir{<};
    "3"; (8,10) **\dir{-};
    (15,-9)*{\lambda+\alpha_i};
    (-12,9)*{\lambda};
    (0,4)*{\bullet};
    (10,8)*{\scs };
    (-10,-8)*{\scs };
    \endxy
    \quad = \quad
      \xy 0;/r.17pc/:
 (0,10);(0,-10); **\dir{-} ?(.75)*\dir{<}+(2.3,0)*{\scriptstyle{}}
 ?(.1)*\dir{ }+(2,0)*{\scs };
 (0,0)*{\bullet};
 (-6,5)*{\lambda};
 (8,5)*{\lambda+\alpha_i};
 (-10,0)*{};(10,0)*{};(-2,-8)*{\scs };
 \endxy
    \quad = \quad
   \xy 0;/r.17pc/:
    (8,5)*{}="1";
    (0,5)*{}="2";
    (0,-5)*{}="2'";
    (-8,-5)*{}="3";
    (8,-10);"1" **\dir{-};
    "2";"2'" **\dir{-} ?(.5)*\dir{<};
    "1";"2" **\crv{(8,12) & (0,12)} ?(0)*\dir{<};
    "2'";"3" **\crv{(0,-12) & (-8,-12)}?(1)*\dir{<};
    "3"; (-8,10) **\dir{-};
    (15,9)*{\lambda+\alpha_i};
    (-12,-9)*{\lambda};
    (0,4)*{\bullet};
    (-10,8)*{\scs };
    (10,-8)*{\scs };
    \endxy
\end{equation}
The $Q$-cyclic relations for crossings are given by
\begin{equation} \label{eq_almost_cyclic}
   \xy 0;/r.17pc/:
  (0,0)*{\xybox{
    (-4,4)*{};(4,-4)*{} **\crv{(-4,1) & (4,-1)}?(1)*\dir{>} ;
    (4,4)*{};(-4,-4)*{} **\crv{(4,1) & (-4,-1)}?(1)*\dir{>};
    (-6.5,-3)*{\scs i};
     (6.5,-3)*{\scs j};
     (- 9,1)*{\scs  \lambda};
     (-10,0)*{};(10,0)*{};
     }};
  \endxy \quad = \quad
  t_{ij}^{-1}\xy 0;/r.17pc/:
  (0,0)*{\xybox{
    (4,-4)*{};(-4,4)*{} **\crv{(4,-1) & (-4,1)}?(1)*\dir{>};
    (-4,-4)*{};(4,4)*{} **\crv{(-4,-1) & (4,1)};
     (-4,4)*{};(18,4)*{} **\crv{(-4,16) & (18,16)} ?(1)*\dir{>};
     (4,-4)*{};(-18,-4)*{} **\crv{(4,-16) & (-18,-16)} ?(1)*\dir{<}?(0)*\dir{<};
     (-18,-4);(-18,12) **\dir{-};(-12,-4);(-12,12) **\dir{-};
     (18,4);(18,-12) **\dir{-};(12,4);(12,-12) **\dir{-};
     (8,1)*{ \lambda};
     (-10,0)*{};(10,0)*{};
     (-4,-4)*{};(-12,-4)*{} **\crv{(-4,-10) & (-12,-10)}?(1)*\dir{<}?(0)*\dir{<};
      (4,4)*{};(12,4)*{} **\crv{(4,10) & (12,10)}?(1)*\dir{>}?(0)*\dir{>};
      (-20,11)*{\scs j};(-10,11)*{\scs i};
      (20,-11)*{\scs j};(10,-11)*{\scs i};
     }};
  \endxy
\quad =  \quad t_{ji}^{-1}
\xy 0;/r.17pc/:
  (0,0)*{\xybox{
    (-4,-4)*{};(4,4)*{} **\crv{(-4,-1) & (4,1)}?(1)*\dir{>};
    (4,-4)*{};(-4,4)*{} **\crv{(4,-1) & (-4,1)};
     (4,4)*{};(-18,4)*{} **\crv{(4,16) & (-18,16)} ?(1)*\dir{>};
     (-4,-4)*{};(18,-4)*{} **\crv{(-4,-16) & (18,-16)} ?(1)*\dir{<}?(0)*\dir{<};
     (18,-4);(18,12) **\dir{-};(12,-4);(12,12) **\dir{-};
     (-18,4);(-18,-12) **\dir{-};(-12,4);(-12,-12) **\dir{-};
     (8,1)*{ \lambda};
     (-10,0)*{};(10,0)*{};
      (4,-4)*{};(12,-4)*{} **\crv{(4,-10) & (12,-10)}?(1)*\dir{<}?(0)*\dir{<};
      (-4,4)*{};(-12,4)*{} **\crv{(-4,10) & (-12,10)}?(1)*\dir{>}?(0)*\dir{>};
      (20,11)*{\scs i};(10,11)*{\scs j};
      (-20,-11)*{\scs i};(-10,-11)*{\scs j};
     }};
  \endxy
\end{equation}
This definition ensures that downward strands are equipped with an action of the KLR algebra associated with the choice of scalars $Q'$ (we define $Q'$ in section \ref{sec:datum} and the KLR algebras in section \ref{sec:KLR}).

The $Q$-cyclic condition also relates our definition of sideways crossings to the other possible definition of sideways crossings:
\begin{equation} \label{eq_crossl-gen}
  \xy 0;/r.18pc/:
  (0,0)*{\xybox{
    (-4,-4)*{};(4,4)*{} **\crv{(-4,-1) & (4,1)}?(1)*\dir{>} ;
    (4,-4)*{};(-4,4)*{} **\crv{(4,-1) & (-4,1)}?(0)*\dir{<};
    (-5,-3)*{\scs j};
     (6.5,-3)*{\scs i};
     (9,2)*{ \lambda};
     (-12,0)*{};(12,0)*{};
     }};
  \endxy
\quad := \quad
 \xy 0;/r.17pc/:
  (0,0)*{\xybox{
    (4,-4)*{};(-4,4)*{} **\crv{(4,-1) & (-4,1)}?(1)*\dir{>};
    (-4,-4)*{};(4,4)*{} **\crv{(-4,-1) & (4,1)};
     (-4,4);(-4,12) **\dir{-};
     (-12,-4);(-12,12) **\dir{-};
     (4,-4);(4,-12) **\dir{-};(12,4);(12,-12) **\dir{-};
     (16,1)*{\lambda};
     (-10,0)*{};(10,0)*{};
     (-4,-4)*{};(-12,-4)*{} **\crv{(-4,-10) & (-12,-10)}?(1)*\dir{<}?(0)*\dir{<};
      (4,4)*{};(12,4)*{} **\crv{(4,10) & (12,10)}?(1)*\dir{>}?(0)*\dir{>};
      (-14,11)*{\scs i};(-2,11)*{\scs j};
      (14,-11)*{\scs i};(2,-11)*{\scs j};
     }};
  \endxy
  \quad = \quad t_{ij} \;\;
 \xy 0;/r.17pc/:
  (0,0)*{\xybox{
    (-4,-4)*{};(4,4)*{} **\crv{(-4,-1) & (4,1)}?(1)*\dir{<};
    (4,-4)*{};(-4,4)*{} **\crv{(4,-1) & (-4,1)};
     (4,4);(4,12) **\dir{-};
     (12,-4);(12,12) **\dir{-};
     (-4,-4);(-4,-12) **\dir{-};(-12,4);(-12,-12) **\dir{-};
     (16,1)*{\lambda};
     (10,0)*{};(-10,0)*{};
     (4,-4)*{};(12,-4)*{} **\crv{(4,-10) & (12,-10)}?(1)*\dir{>}?(0)*\dir{>};
      (-4,4)*{};(-12,4)*{} **\crv{(-4,10) & (-12,10)}?(1)*\dir{<}?(0)*\dir{<};
     }};
     (12,11)*{\scs j};(0,11)*{\scs i};
      (-17,-11)*{\scs j};(-5,-11)*{\scs i};
  \endxy
\end{equation}
\begin{equation} \label{eq_crossr-gen}
  \xy 0;/r.18pc/:
  (0,0)*{\xybox{
    (-4,-4)*{};(4,4)*{} **\crv{(-4,-1) & (4,1)}?(0)*\dir{<} ;
    (4,-4)*{};(-4,4)*{} **\crv{(4,-1) & (-4,1)}?(1)*\dir{>};
    (5.1,-3)*{\scs i};
     (-6.5,-3)*{\scs j};
     (9,2)*{ \lambda};
     (-12,0)*{};(12,0)*{};
     }};
  \endxy
\quad := \quad
 \xy 0;/r.17pc/:
  (0,0)*{\xybox{
    (-4,-4)*{};(4,4)*{} **\crv{(-4,-1) & (4,1)}?(1)*\dir{>};
    (4,-4)*{};(-4,4)*{} **\crv{(4,-1) & (-4,1)};
     (4,4);(4,12) **\dir{-};
     (12,-4);(12,12) **\dir{-};
     (-4,-4);(-4,-12) **\dir{-};(-12,4);(-12,-12) **\dir{-};
     (16,-6)*{\lambda};
     (10,0)*{};(-10,0)*{};
     (4,-4)*{};(12,-4)*{} **\crv{(4,-10) & (12,-10)}?(1)*\dir{<}?(0)*\dir{<};
      (-4,4)*{};(-12,4)*{} **\crv{(-4,10) & (-12,10)}?(1)*\dir{>}?(0)*\dir{>};
      (14,11)*{\scs j};(2,11)*{\scs i};
      (-14,-11)*{\scs j};(-2,-11)*{\scs i};
     }};
  \endxy
  \quad = \quad t_{ji} \;\;
  \xy 0;/r.17pc/:
  (0,0)*{\xybox{
    (4,-4)*{};(-4,4)*{} **\crv{(4,-1) & (-4,1)}?(1)*\dir{<};
    (-4,-4)*{};(4,4)*{} **\crv{(-4,-1) & (4,1)};
     (-4,4);(-4,12) **\dir{-};
     (-12,-4);(-12,12) **\dir{-};
     (4,-4);(4,-12) **\dir{-};(12,4);(12,-12) **\dir{-};
     (16,6)*{\lambda};
     (-10,0)*{};(10,0)*{};
     (-4,-4)*{};(-12,-4)*{} **\crv{(-4,-10) & (-12,-10)}?(1)*\dir{>}?(0)*\dir{>};
      (4,4)*{};(12,4)*{} **\crv{(4,10) & (12,10)}?(1)*\dir{<}?(0)*\dir{<};
      (-14,11)*{\scs i};(-2,11)*{\scs j};(14,-11)*{\scs i};(2,-11)*{\scs j};
     }};
  \endxy
\end{equation}
where the  equality in \eqref{eq_crossl-gen} and \eqref{eq_crossr-gen}
follow from \eqref{eq_almost_cyclic}.

%
\subsection{KLR algebras} \label{sec:KLR}
%

The KLR algebra $R=R_Q$ associated to $Q$ is defined by finite $\Bbbk$-linear combinations of braid--like diagrams in the plane, where each strand is labeled by a vertex $i \in I$.  Strands can intersect and can carry dots but triple intersections are not allowed.  Diagrams are considered up to planar isotopy that do not change the combinatorial type of the diagram. We recall the local relations:

\begin{enumerate}[i)]
\item
If all strands are labeled by the same $i \in I$ then the  NilHecke algebra axioms hold
 \begin{equation}
 \vcenter{\xy 0;/r.17pc/:
    (-4,-4)*{};(4,4)*{} **\crv{(-4,-1) & (4,1)}?(1)*\dir{};
    (4,-4)*{};(-4,4)*{} **\crv{(4,-1) & (-4,1)}?(1)*\dir{};
    (-4,4)*{};(4,12)*{} **\crv{(-4,7) & (4,9)}?(1)*\dir{};
    (4,4)*{};(-4,12)*{} **\crv{(4,7) & (-4,9)}?(1)*\dir{};
 \endxy}
 \;\; =\;\; 0, \qquad \quad
 \vcenter{
 \xy 0;/r.17pc/:
    (-4,-4)*{};(4,4)*{} **\crv{(-4,-1) & (4,1)}?(1)*\dir{};
    (4,-4)*{};(-4,4)*{} **\crv{(4,-1) & (-4,1)}?(1)*\dir{};
    (4,4)*{};(12,12)*{} **\crv{(4,7) & (12,9)}?(1)*\dir{};
    (12,4)*{};(4,12)*{} **\crv{(12,7) & (4,9)}?(1)*\dir{};
    (-4,12)*{};(4,20)*{} **\crv{(-4,15) & (4,17)}?(1)*\dir{};
    (4,12)*{};(-4,20)*{} **\crv{(4,15) & (-4,17)}?(1)*\dir{};
    (-4,4)*{}; (-4,12) **\dir{-};
    (12,-4)*{}; (12,4) **\dir{-};
    (12,12)*{}; (12,20) **\dir{-};
\endxy}
 \;\; =\;\;
 \vcenter{
 \xy 0;/r.17pc/:
    (4,-4)*{};(-4,4)*{} **\crv{(4,-1) & (-4,1)}?(1)*\dir{};
    (-4,-4)*{};(4,4)*{} **\crv{(-4,-1) & (4,1)}?(1)*\dir{};
    (-4,4)*{};(-12,12)*{} **\crv{(-4,7) & (-12,9)}?(1)*\dir{};
    (-12,4)*{};(-4,12)*{} **\crv{(-12,7) & (-4,9)}?(1)*\dir{};
    (4,12)*{};(-4,20)*{} **\crv{(4,15) & (-4,17)}?(1)*\dir{};
    (-4,12)*{};(4,20)*{} **\crv{(-4,15) & (4,17)}?(1)*\dir{};
    (4,4)*{}; (4,12) **\dir{-};
    (-12,-4)*{}; (-12,4) **\dir{-};
    (-12,12)*{}; (-12,20) **\dir{-};
\endxy} \label{eq_nil_rels}
  \end{equation}
\begin{eqnarray}
 r_{i}\;\;\xy 0;/r.18pc/:
  (4,4);(4,-4) **\dir{-}?(0)*\dir{}+(2.3,0)*{};
  (-4,4);(-4,-4) **\dir{-}?(0)*\dir{}+(2.3,0)*{};
 \endxy
 \quad =
\xy 0;/r.18pc/:
  (0,0)*{\xybox{
    (-4,-4)*{};(4,4)*{} **\crv{(-4,-1) & (4,1)}?(1)*\dir{}?(.25)*{\bullet};
    (4,-4)*{};(-4,4)*{} **\crv{(4,-1) & (-4,1)}?(1)*\dir{};
     (-10,0)*{};(10,0)*{};
     }};
  \endxy
 \;\; -
\xy 0;/r.18pc/:
  (0,0)*{\xybox{
    (-4,-4)*{};(4,4)*{} **\crv{(-4,-1) & (4,1)}?(1)*\dir{}?(.75)*{\bullet};
    (4,-4)*{};(-4,4)*{} **\crv{(4,-1) & (-4,1)}?(1)*\dir{};
     (-10,0)*{};(10,0)*{};
     }};
  \endxy
 \;\; =
\xy 0;/r.18pc/:
  (0,0)*{\xybox{
    (-4,-4)*{};(4,4)*{} **\crv{(-4,-1) & (4,1)}?(1)*\dir{};
    (4,-4)*{};(-4,4)*{} **\crv{(4,-1) & (-4,1)}?(1)*\dir{}?(.75)*{\bullet};
     (-10,0)*{};(10,0)*{};
     }};
  \endxy
 \;\; -
  \xy 0;/r.18pc/:
  (0,0)*{\xybox{
    (-4,-4)*{};(4,4)*{} **\crv{(-4,-1) & (4,1)}?(1)*\dir{} ;
    (4,-4)*{};(-4,4)*{} **\crv{(4,-1) & (-4,1)}?(1)*\dir{}?(.25)*{\bullet};
     (-10,0)*{};(10,0)*{};
     }};
  \endxy \nn \\ \label{eq_nil_dotslide}
\end{eqnarray}

\item For $i \neq j$
\begin{eqnarray}
  \vcenter{\xy 0;/r.17pc/:
    (-4,-4)*{};(4,4)*{} **\crv{(-4,-1) & (4,1)}?(1)*\dir{};
    (4,-4)*{};(-4,4)*{} **\crv{(4,-1) & (-4,1)}?(1)*\dir{};
    (-4,4)*{};(4,12)*{} **\crv{(-4,7) & (4,9)}?(1)*\dir{};
    (4,4)*{};(-4,12)*{} **\crv{(4,7) & (-4,9)}?(1)*\dir{};
  (-5,-3)*{\scs i};
     (5.1,-3)*{\scs j};
 \endxy}
 \qquad = \qquad
 \left\{
 \begin{array}{ccc}
     t_{ij}\;\xy 0;/r.17pc/:
  (3,9);(3,-9) **\dir{-}?(.5)*\dir{}+(2.3,0)*{};
  (-3,9);(-3,-9) **\dir{-}?(.5)*\dir{}+(2.3,0)*{};
  (-5,-6)*{\scs i};     (5.1,-6)*{\scs j};
 \endxy &  &  \text{if $(\alpha_i, \alpha_j)=0$,}\\ \\
 t_{ij} \vcenter{\xy 0;/r.17pc/:
  (3,9);(3,-9) **\dir{-}?(.5)*\dir{}+(2.3,0)*{};
  (-3,9);(-3,-9) **\dir{-}?(.5)*\dir{}+(2.3,0)*{};
  (-3,4)*{\bullet};(-6.5,5)*{\scs d_{ij}};
  (-5,-6)*{\scs i};     (5.1,-6)*{\scs j};
 \endxy} \;\; + \;\; t_{ji}
  \vcenter{\xy 0;/r.17pc/:
  (3,9);(3,-9) **\dir{-}?(.5)*\dir{}+(2.3,0)*{};
  (-3,9);(-3,-9) **\dir{-}?(.5)*\dir{}+(2.3,0)*{};
  (3,4)*{\bullet};(7,5)*{\scs d_{ji}};
  (-5,-6)*{\scs i};     (5.1,-6)*{\scs j};
 \endxy}
 \;\; + \;\;
 \xsum{p,q}{} s_{ij}^{pq}\;
 \vcenter{\xy 0;/r.17pc/:
  (3,9);(3,-9) **\dir{-}?(.5)*\dir{}+(2.3,0)*{};
  (-3,9);(-3,-9) **\dir{-}?(.5)*\dir{}+(2.3,0)*{};
  (-3,4)*{\bullet};(-5.5,5)*{\scs p};
  (3,4)*{\bullet};(6,5)*{\scs q};
  (-5,-6)*{\scs i};     (5.1,-6)*{\scs j};
 \endxy}
   &  & \text{if $(\alpha_i, \alpha_j) \neq 0$,}
 \end{array}
 \right. \nn \\\label{eq_r2_ij-gen}
\end{eqnarray}
where the summation in the case $(\alpha_i, \alpha_j) \neq 0$ is over all $p$, $q$ such that
\begin{equation} \label{eq_pq}
(\alpha_i,\alpha_i)p+(\alpha_j,\alpha_j)q = -2 (\alpha_i,\alpha_j).
\end{equation}

\item For $i \neq j$ the dot sliding relations
\begin{eqnarray} \label{eq_dot_slide_ij-gen}
\xy 0;/r.18pc/:
  (0,0)*{\xybox{
    (-4,-4)*{};(4,4)*{} **\crv{(-4,-1) & (4,1)}?(1)*\dir{}?(.75)*{\bullet};
    (4,-4)*{};(-4,4)*{} **\crv{(4,-1) & (-4,1)}?(1)*\dir{};
    (-5,-3)*{\scs i};
     (5.1,-3)*{\scs j};
     (-10,0)*{};(10,0)*{};
     }};
  \endxy
 \;\; =
\xy 0;/r.18pc/:
  (0,0)*{\xybox{
    (-4,-4)*{};(4,4)*{} **\crv{(-4,-1) & (4,1)}?(1)*\dir{}?(.25)*{\bullet};
    (4,-4)*{};(-4,4)*{} **\crv{(4,-1) & (-4,1)}?(1)*\dir{};
    (-5,-3)*{\scs i};
     (5.1,-3)*{\scs j};
     (-10,0)*{};(10,0)*{};
     }};
  \endxy
\qquad  \xy 0;/r.18pc/:
  (0,0)*{\xybox{
    (-4,-4)*{};(4,4)*{} **\crv{(-4,-1) & (4,1)}?(1)*\dir{};
    (4,-4)*{};(-4,4)*{} **\crv{(4,-1) & (-4,1)}?(1)*\dir{}?(.75)*{\bullet};
    (-5,-3)*{\scs i};
     (5.1,-3)*{\scs j};
     (-10,0)*{};(10,0)*{};
     }};
  \endxy
\;\;  =
  \xy 0;/r.18pc/:
  (0,0)*{\xybox{
    (-4,-4)*{};(4,4)*{} **\crv{(-4,-1) & (4,1)}?(1)*\dir{} ;
    (4,-4)*{};(-4,4)*{} **\crv{(4,-1) & (-4,1)}?(1)*\dir{}?(.25)*{\bullet};
    (-5,-3)*{\scs i};
     (5.1,-3)*{\scs j};
     (-10,0)*{};(12,0)*{};
     }};
  \endxy
\end{eqnarray}
hold.

\item Unless $i = k$ and $(\alpha_i, \alpha_j) < 0$ the relation
\begin{equation}
 \vcenter{
 \xy 0;/r.17pc/:
    (-4,-4)*{};(4,4)*{} **\crv{(-4,-1) & (4,1)}?(1)*\dir{};
    (4,-4)*{};(-4,4)*{} **\crv{(4,-1) & (-4,1)}?(1)*\dir{};
    (4,4)*{};(12,12)*{} **\crv{(4,7) & (12,9)}?(1)*\dir{};
    (12,4)*{};(4,12)*{} **\crv{(12,7) & (4,9)}?(1)*\dir{};
    (-4,12)*{};(4,20)*{} **\crv{(-4,15) & (4,17)}?(1)*\dir{};
    (4,12)*{};(-4,20)*{} **\crv{(4,15) & (-4,17)}?(1)*\dir{};
    (-4,4)*{}; (-4,12) **\dir{-};
    (12,-4)*{}; (12,4) **\dir{-};
    (12,12)*{}; (12,20) **\dir{-};
  (-6,-3)*{\scs i};
  (6,-3)*{\scs j};
  (15,-3)*{\scs k};
\endxy}
 \;\; =\;\;
 \vcenter{
 \xy 0;/r.17pc/:
    (4,-4)*{};(-4,4)*{} **\crv{(4,-1) & (-4,1)}?(1)*\dir{};
    (-4,-4)*{};(4,4)*{} **\crv{(-4,-1) & (4,1)}?(1)*\dir{};
    (-4,4)*{};(-12,12)*{} **\crv{(-4,7) & (-12,9)}?(1)*\dir{};
    (-12,4)*{};(-4,12)*{} **\crv{(-12,7) & (-4,9)}?(1)*\dir{};
    (4,12)*{};(-4,20)*{} **\crv{(4,15) & (-4,17)}?(1)*\dir{};
    (-4,12)*{};(4,20)*{} **\crv{(-4,15) & (4,17)}?(1)*\dir{};
    (4,4)*{}; (4,12) **\dir{-};
    (-12,-4)*{}; (-12,4) **\dir{-};
    (-12,12)*{}; (-12,20) **\dir{-};
  (7,-3)*{\scs k};
  (-6,-3)*{\scs j};
  (-14,-3)*{\scs i};
\endxy} \label{eq_r3_easy-gen}
\end{equation}
holds. Otherwise, $(\alpha_i, \alpha_j) < 0$ and
\begin{equation}
r_{i}^{-1}\left(\;\; \vcenter{
 \xy 0;/r.17pc/:
    (-4,-4)*{};(4,4)*{} **\crv{(-4,-1) & (4,1)}?(1)*\dir{};
    (4,-4)*{};(-4,4)*{} **\crv{(4,-1) & (-4,1)}?(1)*\dir{};
    (4,4)*{};(12,12)*{} **\crv{(4,7) & (12,9)}?(1)*\dir{};
    (12,4)*{};(4,12)*{} **\crv{(12,7) & (4,9)}?(1)*\dir{};
    (-4,12)*{};(4,20)*{} **\crv{(-4,15) & (4,17)}?(1)*\dir{};
    (4,12)*{};(-4,20)*{} **\crv{(4,15) & (-4,17)}?(1)*\dir{};
    (-4,4)*{}; (-4,12) **\dir{-};
    (12,-4)*{}; (12,4) **\dir{-};
    (12,12)*{}; (12,20) **\dir{-};
  (-6,-3)*{\scs i};
  (6,-3)*{\scs j};
  (14,-3)*{\scs i};
\endxy}
\quad - \quad
 \vcenter{
 \xy 0;/r.17pc/:
    (4,-4)*{};(-4,4)*{} **\crv{(4,-1) & (-4,1)}?(1)*\dir{};
    (-4,-4)*{};(4,4)*{} **\crv{(-4,-1) & (4,1)}?(1)*\dir{};
    (-4,4)*{};(-12,12)*{} **\crv{(-4,7) & (-12,9)}?(1)*\dir{};
    (-12,4)*{};(-4,12)*{} **\crv{(-12,7) & (-4,9)}?(1)*\dir{};
    (4,12)*{};(-4,20)*{} **\crv{(4,15) & (-4,17)}?(1)*\dir{};
    (-4,12)*{};(4,20)*{} **\crv{(-4,15) & (4,17)}?(1)*\dir{};
    (4,4)*{}; (4,12) **\dir{-};
    (-12,-4)*{}; (-12,4) **\dir{-};
    (-12,12)*{}; (-12,20) **\dir{-};
  (6,-3)*{\scs i};
  (-6,-3)*{\scs j};
  (-14,-3)*{\scs i};
\endxy}\;\; \right)
 \;\; =\;\;
 t_{ij}\sum_{\ell_1+\ell_2=d_{ij}-1} \;\;
\xy 0;/r.17pc/:
  (4,12);(4,-12) **\dir{-}?(.5)*\dir{};
  (-4,12);(-4,-12) **\dir{-}?(.5)*\dir{}?(.25)*\dir{}+(0,0)*{\bullet}+(-3,0)*{\scs \ell_1};
  (12,12);(12,-12) **\dir{-}?(.5)*\dir{}?(.25)*\dir{}+(0,0)*{\bullet}+(3,0)*{\scs \ell_2};
  (-6,-9)*{\scs i};     (6.1,-9)*{\scs j};
  (14,-9)*{\scs i};
 \endxy
 \;\; + \;\;
 \sum_{p,q} s_{ij}^{pq} \sum_{\xy (0,2.5)*{\scs \ell_1+\ell_2}; (0,-1)*{\scs =p-1}; \endxy}
 \xy 0;/r.17pc/:
  (4,12);(4,-12)   **\dir{-}?(.5)*\dir{} ?(.25)*\dir{}+(0,0)*{\bullet}+(-3,0)*{\scs q};
  (-4,12);(-4,-12) **\dir{-}?(.5)*\dir{}?(.25)*\dir{}+(0,0)*{\bullet}+(-3,0)*{\scs \ell_1};
  (12,12);(12,-12) **\dir{-}?(.5)*\dir{}?(.25)*\dir{}+(0,0)*{\bullet}+(3,0)*{\scs \ell_2};
  (-6,-9)*{\scs i};     (6.1,-9)*{\scs j};
  (14,-9)*{\scs i};
 \endxy
 \label{eq_r3_hard-gen}
\end{equation}
where the $p,q$ summation is as in \eqref{eq_pq}.
\end{enumerate}

\begin{rem} \hfill
\begin{enumerate}
\item It is always possible to rescale the $ii$-crossing by $r_{i}$ so that $r_{i}=1$ in the definition above. In the literature it is common to see both $r_{i}=1$ or $r_{i}=-1$.

\item The 2-category $\Ucat_Q(\mf{g})$ is a cyclic 2-category (as defined in \cite{Lau1}) if $t_{ij}=t_{ji}$ for all $i,j$ with $d_{ij}>0$.

\item When the Cartan datum is symmetric, i.e. simply-laced,  so that $(\alpha_i,\alpha_i)=2$ and $(\alpha_i,\alpha_j)=-1$, then the KLR algebra takes on a simplified form. In particular, $s_{ij}^{pq}=0$ for all $i,j$ and $p,q$, making
the summations in \eqref{eq_r2_ij-gen} vanish. Since $d_{ij}=1$, the right-hand side in \eqref{eq_r3_hard-gen} has only one term.  Furthermore, if the graph corresponding to the symmetric Cartan datum is a tree, then it is always possible to rescale the coefficients so that $t_{ij}=t_{ji}=1$, see \cite{KL2}.
\end{enumerate}
\end{rem}

Inductively applying the NilHecke dot slide relation gives the equation
\begin{equation} \label{eq_ind_dotslide}
\xy 0;/r.18pc/:
  (0,0)*{\xybox{
    (-4,-4)*{};(4,4)*{} **\crv{(-4,-1) & (4,1)}?(1)*\dir{>}?(.25)*{\bullet}+(-2.5,1)*{\scs m};
    (4,-4)*{};(-4,4)*{} **\crv{(4,-1) & (-4,1)}?(1)*\dir{>};
     (8,-4)*{n};
     (-10,0)*{};(10,0)*{};
     }};
  \endxy
 \;\; -
 \xy 0;/r.18pc/:
  (0,0)*{\xybox{
    (-4,-4)*{};(4,4)*{} **\crv{(-4,-1) & (4,1)}?(1)*\dir{>}?(.75)*{\bullet}+(2.5,-1)*{\scs m};
    (4,-4)*{};(-4,4)*{} **\crv{(4,-1) & (-4,1)}?(1)*\dir{>};
     (8,-4)*{n};
     (-10,0)*{};(10,0)*{};
     }};
  \endxy
 \;\; =
\xy 0;/r.18pc/:
  (0,0)*{\xybox{
    (-4,-4)*{};(4,4)*{} **\crv{(-4,-1) & (4,1)}?(1)*\dir{>};
    (4,-4)*{};(-4,4)*{} **\crv{(4,-1) & (-4,1)}?(1)*\dir{>}?(.75)*{\bullet}+(-2.5,-1)*{\scs m};
     (8,3)*{ n};
     (-10,0)*{};(10,0)*{};
     }};
  \endxy
 \;\; -
  \xy 0;/r.18pc/:
  (0,0)*{\xybox{
    (-4,-4)*{};(4,4)*{} **\crv{(-4,-1) & (4,1)}?(1)*\dir{>} ;
    (4,-4)*{};(-4,4)*{} **\crv{(4,-1) & (-4,1)}?(1)*\dir{>}?(.25)*{\bullet}+(2.5,1)*{\scs m};
     (8,3)*{n};
     (-10,0)*{};(10,0)*{};
     }};
  \endxy
  \;\; = \;\; r_{i}
  \sum_{f_1 + f_2 = m-1}
  \xy  0;/r.18pc/:
  (3,4);(3,-4) **\dir{-}?(0)*\dir{<} ?(.5)*\dir{}+(0,0)*{\bullet}+(2.5,1)*{\scs f_2};
  (-3,4);(-3,-4) **\dir{-}?(0)*\dir{<}?(.5)*\dir{}+(0,0)*{\bullet}+(-2.5,1)*{\scs f_1};;
  (9,-4)*{n};
 \endxy
\end{equation}
that will be useful in what follows.

%
\subsection{Mixed relations}\label{sec:mixedrels}
%

For $i \neq j$ relations
\begin{equation}
 \vcenter{   \xy 0;/r.18pc/:
    (-4,-4)*{};(4,4)*{} **\crv{(-4,-1) & (4,1)}?(1)*\dir{>};
    (4,-4)*{};(-4,4)*{} **\crv{(4,-1) & (-4,1)}?(1)*\dir{<};?(0)*\dir{<};
    (-4,4)*{};(4,12)*{} **\crv{(-4,7) & (4,9)};
    (4,4)*{};(-4,12)*{} **\crv{(4,7) & (-4,9)}?(1)*\dir{>};
  (8,8)*{\lambda};(-6,-3)*{\scs i};
     (6,-3)*{\scs j};
 \endxy}
 \;\; = \;\; t_{ji}\;\;
\xy 0;/r.18pc/:
  (3,9);(3,-9) **\dir{-}?(.55)*\dir{>}+(2.3,0)*{};
  (-3,9);(-3,-9) **\dir{-}?(.5)*\dir{<}+(2.3,0)*{};
  (8,2)*{\lambda};(-5,-6)*{\scs i};     (5.1,-6)*{\scs j};
 \endxy
\qquad \quad
    \vcenter{\xy 0;/r.18pc/:
    (-4,-4)*{};(4,4)*{} **\crv{(-4,-1) & (4,1)}?(1)*\dir{<};?(0)*\dir{<};
    (4,-4)*{};(-4,4)*{} **\crv{(4,-1) & (-4,1)}?(1)*\dir{>};
    (-4,4)*{};(4,12)*{} **\crv{(-4,7) & (4,9)}?(1)*\dir{>};
    (4,4)*{};(-4,12)*{} **\crv{(4,7) & (-4,9)};
  (8,8)*{\lambda};(-6,-3)*{\scs i};
     (6,-3)*{\scs j};
 \endxy}
 \;\;=\;\; t_{ij}\;\;
\xy 0;/r.18pc/:
  (3,9);(3,-9) **\dir{-}?(.5)*\dir{<}+(2.3,0)*{};
  (-3,9);(-3,-9) **\dir{-}?(.55)*\dir{>}+(2.3,0)*{};
  (8,2)*{\lambda};(-5,-6)*{\scs i};     (5.1,-6)*{\scs j};
 \endxy
\end{equation}
hold.

%
\subsection{Dotted bubbles}\label{sec:dotbubbles}
%

Condition (\ref{item_positivity}) of Definition~\ref{defU_cat} regarding dotted bubbles can be summarized diagrammatically as follows.  Firstly, for all $m \in \Z_+$ one has
\begin{equation} \label{eq_positivity_bubbles}
\xy 0;/r.18pc/:
 (-12,0)*{\icbub{m}{i}};
 (-8,8)*{\lambda};
 \endxy
  = 0
 \qquad  \text{if $m<\la i, \lambda\ra-1$,} \qquad \xy 0;/r.18pc/: (-12,0)*{\iccbub{m}{i}};
 (-8,8)*{\lambda};
 \endxy = 0\quad
  \text{if $m< -\la i,\lambda \ra-1$}
\end{equation}
which means that dotted bubbles of negative degree are zero. On the other hand, a dotted bubble of degree zero is just the identity 2-morphism:
\[
\xy 0;/r.18pc/:
 (0,0)*{\cbub{\la i,\lambda \ra-1}};
  (4,8)*{\lambda};
 \endxy
  =  \Id_{\onenn{\lambda}} \quad \text{for $\la i,\lambda \ra \geq 1$,}
  \qquad \quad
  \xy 0;/r.18pc/:
 (0,0)*{\ccbub{-\la i,\lambda \ra-1}};
  (4,8)*{\lambda};
 \endxy  =  \Id_{\onenn{\lambda}} \quad \text{for $\la i,\lambda \ra \leq -1$.}\]

%
\subsection{Extended ${\mf{sl}}_2$ relations}\label{sec:highersl2}
%

%
\subsubsection{Conventions for ${\mf{sl}}_2$ and Cartan data with $I=\{i\}$ }\label{sec:sl2convs}
%

When $\mf{g}=\mf{sl}_2$ we use a simplified notation for the Cartan datum. We identify the weight lattice $X$ with $\Z$ by labeling the weight $\lambda$ with $n \in \Z$ where $\la i, \lambda\ra =n$.  In this case $d_{ii}=-\la i, \alpha_i\ra =2$,  $d_i=\frac{(\alpha_i,\alpha_i)}{2}=1$, so that $q_i=q$ and $[n]_i=[n]$.  Furthermore, the choice of scalars $Q$ is determined by the single parameter $r_{i}$.  By rescaling the $ii$-crossing, it is easy to see that for all choices of $Q$ the 2-categories $\Ucat_Q(\mf{sl}_2)$ are isomorphic to the 2-category $\Ucat(\mf{sl}_2)$ given by taking $r_{i}=1$.

For other Cartan datum with a single vertex $I = \{i\}$ and $(\alpha_i,\alpha_i) \in 2\Z_{>0}+2$, the 2-categories $\Ucat_Q(\mf{g})$ are again determined by the parameter $r_{i}$. By appropriate rescaling these 2-categories can be made to have $r_{i}=1$.  These 2-categories $\Ucat_Q(\mf{g})$ are isomorphic to the 2-category $\Ucat(\mf{sl}_2)$ by an additive $\Bbbk$-linear functor that only effects the degrees of 2-morphisms. In this way, we can treat the case of a single vertex $I=\{i\}$ in a similar manner as the case of $\mf{sl}_2$.

%
\subsubsection{Fake bubbles } \label{subsubsec-fake}
%

In order to describe certain extended ${\mf{sl}}_2$ relations, it is convenient to use a shorthand notation from \cite{Lau1} called fake bubbles (see also section~\ref{sec:symm} for more details). These are diagrams for dotted bubbles where the labels of the number of dots is negative, but the total degree of the dotted bubble taken with these negative dots is still positive. They allow us to write these extended ${\mf{sl}}_2$ relations more uniformly (i.e. independent of whether the weight $n$ is positive or negative). The fake bubbles are defined in weight $n=0$ as
\begin{equation}
\vcenter{\xy 0;/r.18pc/:
    (2,-11)*{\cbub{-1}{}};
  (12,-2)*{0};
 \endxy} \;\; =  \Id_{\onenn{0}}, \qquad \quad
 \vcenter{\xy 0;/r.18pc/:
    (2,-11)*{\ccbub{-1}{}};
  (12,-2)*{0};
 \endxy} \;\; =  \Id_{\onenn{0}}.
 \end{equation}
In weight $n \ne 0$, the fake bubbles are inductively defined by the equations
\begin{equation} \label{eq_homo_inf_grass}
\sum_{
 \xy (0,3)*{\scs \ell_1+\ell_2=j}; \endxy} \;\;
\xy 0;/r.17pc/:
 (0,0)*{\cbub{n-1+\ell_1}{}};
 (20,0)*{\ccbub{-n-1+\ell_2}{}};(0,10)*{};
 (8,8)*{n};
 \endxy \;\; = \;\; \delta_{j,0},
\end{equation}
where $j\leq |n|$.  The equation \eqref{eq_homo_inf_grass} for $j$ larger than $|n|$  follows from the extended ${\mf{sl}}_2$ relations given below.  These relations then relate clockwise and counter-clockwise real bubbles.

Equation \eqref{eq_homo_inf_grass} for all $j$ can be encoded in the homogeneous terms in $t$ of the `infinite Grassmannian' equation
\begin{center}
\begin{eqnarray}
 \makebox[0pt]{ $
\left( \xy 0;/r.15pc/:
 (0,0)*{\ccbub{-n-1}{}};
  (4,8)*{n};
 \endxy
 +
 \xy 0;/r.15pc/:
 (0,0)*{\ccbub{-n-1+1}{}};
  (4,8)*{n};
 \endxy t
 + \cdots +
\xy 0;/r.15pc/:
 (0,0)*{\ccbub{-n-1+\alpha}{}};
  (4,8)*{n};
 \endxy t^{\alpha}
 + \cdots
\right)
\left(\xy 0;/r.15pc/:
 (0,0)*{\cbub{n-1}{}};
  (4,8)*{n};
 \endxy
 + \xy 0;/r.15pc/:
 (0,0)*{\cbub{n-1+1}{}};
  (4,8)*{n};
 \endxy t
 +\cdots +
\xy 0;/r.15pc/:
 (0,0)*{\cbub{n-1+\alpha}{}};
 (4,8)*{n};
 \endxy t^{\alpha}
 + \cdots
\right) =1.$ } \nn \\ \label{eq_infinite_Grass}
\end{eqnarray}
\end{center}


%
\subsubsection{The relations}
%

If $n > 0$ then we have:
\begin{equation} \label{eq_reduction-ngeqz}
  \xy 0;/r.17pc/:
  (14,8)*{n};
  (-3,-10)*{};(3,5)*{} **\crv{(-3,-2) & (2,1)}?(1)*\dir{>};?(.15)*\dir{>};
    (3,-5)*{};(-3,10)*{} **\crv{(2,-1) & (-3,2)}?(.85)*\dir{>} ?(.1)*\dir{>};
  (3,5)*{}="t1";  (9,5)*{}="t2";
  (3,-5)*{}="t1'";  (9,-5)*{}="t2'";
   "t1";"t2" **\crv{(4,8) & (9, 8)};
   "t1'";"t2'" **\crv{(4,-8) & (9, -8)};
   "t2'";"t2" **\crv{(10,0)} ;
 \endxy\;\; = \;\; 0
\qquad \qquad
  \xy 0;/r.17pc/:
  (-14,8)*{n};
  (3,-10)*{};(-3,5)*{} **\crv{(3,-2) & (-2,1)}?(1)*\dir{>};?(.15)*\dir{>};
    (-3,-5)*{};(3,10)*{} **\crv{(-2,-1) & (3,2)}?(.85)*\dir{>} ?(.1)*\dir{>};
  (-3,5)*{}="t1";  (-9,5)*{}="t2";
  (-3,-5)*{}="t1'";  (-9,-5)*{}="t2'";
   "t1";"t2" **\crv{(-4,8) & (-9, 8)};
   "t1'";"t2'" **\crv{(-4,-8) & (-9, -8)};
   "t2'";"t2" **\crv{(-10,0)} ;
 \endxy \;\; = \;\; r_{i}
 \sum_{g_1+g_2=n}^{}
  \xy 0;/r.17pc/:
  (-12,8)*{n};
  (0,0)*{\bbe{}};(2,-8)*{\scs};
  (-12,-2)*{\ccbub{-n-1+g_2}};
  (0,6)*{\bullet}+(3,-1)*{\scs g_1};
 \endxy
\end{equation}
\begin{equation}
 \vcenter{\xy 0;/r.17pc/:
  (-8,0)*{};
  (8,0)*{};
  (-4,10)*{}="t1";
  (4,10)*{}="t2";
  (-4,-10)*{}="b1";
  (4,-10)*{}="b2";(-6,-8)*{\scs };(6,-8)*{\scs };
  "t1";"b1" **\dir{-} ?(.5)*\dir{<};
  "t2";"b2" **\dir{-} ?(.5)*\dir{>};
  (10,2)*{n};
  \endxy}
\;\; = \;\;
 -\;\; r_{i}^{-2}\;
 \vcenter{   \xy 0;/r.17pc/:
    (-4,-4)*{};(4,4)*{} **\crv{(-4,-1) & (4,1)}?(1)*\dir{>};
    (4,-4)*{};(-4,4)*{} **\crv{(4,-1) & (-4,1)}?(1)*\dir{<};?(0)*\dir{<};
    (-4,4)*{};(4,12)*{} **\crv{(-4,7) & (4,9)};
    (4,4)*{};(-4,12)*{} **\crv{(4,7) & (-4,9)}?(1)*\dir{>};
  (8,8)*{n};(-6,-3)*{\scs };
     (6.5,-3)*{\scs };
 \endxy}
  \;\; + \;\;
   \sum_{ \xy  (0,3)*{\scs f_1+f_2+f_3}; (0,0)*{\scs =n-1};\endxy}
    \vcenter{\xy 0;/r.17pc/:
    (-10,10)*{n};
    (-8,0)*{};
  (8,0)*{};
  (-4,-15)*{}="b1";
  (4,-15)*{}="b2";
  "b2";"b1" **\crv{(5,-8) & (-5,-8)}; ?(.05)*\dir{<} ?(.93)*\dir{<}
  ?(.8)*\dir{}+(0,-.1)*{\bullet}+(-3,2)*{\scs f_3};
  (-4,15)*{}="t1";
  (4,15)*{}="t2";
  "t2";"t1" **\crv{(5,8) & (-5,8)}; ?(.15)*\dir{>} ?(.95)*\dir{>}
  ?(.4)*\dir{}+(0,-.2)*{\bullet}+(3,-2)*{\scs \; f_1};
  (0,0)*{\ccbub{\scs \quad -n-1+f_2}};
  \endxy}
   \qquad \quad
 \vcenter{\xy 0;/r.17pc/:
  (-8,0)*{};(-6,-8)*{\scs };(6,-8)*{\scs };
  (8,0)*{};
  (-4,10)*{}="t1";
  (4,10)*{}="t2";
  (-4,-10)*{}="b1";
  (4,-10)*{}="b2";
  "t1";"b1" **\dir{-} ?(.5)*\dir{>};
  "t2";"b2" **\dir{-} ?(.5)*\dir{<};
  (10,2)*{n};
  \endxy}
\;\; = \;\;
 -\;\; r_{i}^{-2}\;
   \vcenter{\xy 0;/r.17pc/:
    (-4,-4)*{};(4,4)*{} **\crv{(-4,-1) & (4,1)}?(1)*\dir{<};?(0)*\dir{<};
    (4,-4)*{};(-4,4)*{} **\crv{(4,-1) & (-4,1)}?(1)*\dir{>};
    (-4,4)*{};(4,12)*{} **\crv{(-4,7) & (4,9)}?(1)*\dir{>};
    (4,4)*{};(-4,12)*{} **\crv{(4,7) & (-4,9)};
  (8,8)*{n};(-6,-3)*{\scs };  (6,-3)*{\scs };
 \endxy}
\end{equation}

If $n < 0$ then we have:
\begin{equation} \label{eq_reduction-nleqz}
  \xy 0;/r.17pc/:
  (14,8)*{n};
  (-3,-10)*{};(3,5)*{} **\crv{(-3,-2) & (2,1)}?(1)*\dir{>};?(.15)*\dir{>};
    (3,-5)*{};(-3,10)*{} **\crv{(2,-1) & (-3,2)}?(.85)*\dir{>} ?(.1)*\dir{>};
  (3,5)*{}="t1";  (9,5)*{}="t2";
  (3,-5)*{}="t1'";  (9,-5)*{}="t2'";
   "t1";"t2" **\crv{(4,8) & (9, 8)};
   "t1'";"t2'" **\crv{(4,-8) & (9, -8)};
   "t2'";"t2" **\crv{(10,0)} ;
 \endxy \;\; = \;\; - r_{i}\sum_{f_1+f_2=-n}
 \xy 0;/r.17pc/:
  (19,4)*{n};
  (0,0)*{\bbe{}};(-2,-8)*{\scs };
  (12,-2)*{\cbub{n-1+f_2}};
  (0,6)*{\bullet}+(3,-1)*{\scs f_1};
 \endxy
\qquad \qquad
  \xy 0;/r.17pc/:
  (-14,8)*{n};
  (3,-10)*{};(-3,5)*{} **\crv{(3,-2) & (-2,1)}?(1)*\dir{>};?(.15)*\dir{>};
    (-3,-5)*{};(3,10)*{} **\crv{(-2,-1) & (3,2)}?(.85)*\dir{>} ?(.1)*\dir{>};
  (-3,5)*{}="t1";  (-9,5)*{}="t2";
  (-3,-5)*{}="t1'";  (-9,-5)*{}="t2'";
   "t1";"t2" **\crv{(-4,8) & (-9, 8)};
   "t1'";"t2'" **\crv{(-4,-8) & (-9, -8)};
   "t2'";"t2" **\crv{(-10,0)} ;
 \endxy\;\; = \;\;
0
\end{equation}
\begin{equation}
\vcenter{\xy 0;/r.17pc/:
  (-8,0)*{};
  (8,0)*{};
  (-4,10)*{}="t1";
  (4,10)*{}="t2";
  (-4,-10)*{}="b1";
  (4,-10)*{}="b2";(-6,-8)*{\scs };(6,-8)*{\scs };
  "t1";"b1" **\dir{-} ?(.5)*\dir{<};
  "t2";"b2" **\dir{-} ?(.5)*\dir{>};
  (10,2)*{n};
  \endxy}
\;\; = \;\;
 -\;\; r_{i}^{-2}\;
\vcenter{   \xy 0;/r.17pc/:
    (-4,-4)*{};(4,4)*{} **\crv{(-4,-1) & (4,1)}?(1)*\dir{>};
    (4,-4)*{};(-4,4)*{} **\crv{(4,-1) & (-4,1)}?(1)*\dir{<};?(0)*\dir{<};
    (-4,4)*{};(4,12)*{} **\crv{(-4,7) & (4,9)};
    (4,4)*{};(-4,12)*{} **\crv{(4,7) & (-4,9)}?(1)*\dir{>};
  (8,8)*{n};(-6,-3)*{\scs };
     (6.5,-3)*{\scs };
 \endxy}
\qquad \quad
 \vcenter{\xy 0;/r.17pc/:
  (-8,0)*{};(-6,-8)*{\scs };(6,-8)*{\scs };
  (8,0)*{};
  (-4,10)*{}="t1";
  (4,10)*{}="t2";
  (-4,-10)*{}="b1";
  (4,-10)*{}="b2";
  "t1";"b1" **\dir{-} ?(.5)*\dir{>};
  "t2";"b2" **\dir{-} ?(.5)*\dir{<};
  (10,2)*{n};
  (-10,2)*{n};
  \endxy}
\;\; = \;\;
 -\;\; r_{i}^{-2}\;
  \vcenter{\xy 0;/r.17pc/:
    (-4,-4)*{};(4,4)*{} **\crv{(-4,-1) & (4,1)}?(1)*\dir{<};?(0)*\dir{<};
    (4,-4)*{};(-4,4)*{} **\crv{(4,-1) & (-4,1)}?(1)*\dir{>};
    (-4,4)*{};(4,12)*{} **\crv{(-4,7) & (4,9)}?(1)*\dir{>};
    (4,4)*{};(-4,12)*{} **\crv{(4,7) & (-4,9)};
  (8,8)*{n};(-6,-3)*{\scs };  (6,-3)*{\scs };
 \endxy}
  \;\; + \;\;
    \sum_{ \xy  (0,3)*{\scs g_1+g_2+g_3}; (0,0)*{\scs =-n-1};\endxy}
    \vcenter{\xy 0;/r.17pc/:
    (-8,0)*{};
  (8,0)*{};
  (-4,-15)*{}="b1";
  (4,-15)*{}="b2";
  "b2";"b1" **\crv{(5,-8) & (-5,-8)}; ?(.1)*\dir{>} ?(.95)*\dir{>}
  ?(.8)*\dir{}+(0,-.1)*{\bullet}+(-3,2)*{\scs g_3};
  (-4,15)*{}="t1";
  (4,15)*{}="t2";
  "t2";"t1" **\crv{(5,8) & (-5,8)}; ?(.15)*\dir{<} ?(.9)*\dir{<}
  ?(.4)*\dir{}+(0,-.2)*{\bullet}+(3,-2)*{\scs g_1};
  (0,0)*{\cbub{\scs \quad\; n-1 + g_2}};
  (-10,10)*{n};
  \endxy} \label{eq_ident_decomp-nleqz}
\end{equation}

If $n =0$ then we have:
\begin{equation}
  \xy 0;/r.17pc/:
  (14,8)*{n};
  (-3,-10)*{};(3,5)*{} **\crv{(-3,-2) & (2,1)}?(1)*\dir{>};?(.15)*\dir{>};
    (3,-5)*{};(-3,10)*{} **\crv{(2,-1) & (-3,2)}?(.85)*\dir{>} ?(.1)*\dir{>};
  (3,5)*{}="t1";  (9,5)*{}="t2";
  (3,-5)*{}="t1'";  (9,-5)*{}="t2'";
   "t1";"t2" **\crv{(4,8) & (9, 8)};
   "t1'";"t2'" **\crv{(4,-8) & (9, -8)};
   "t2'";"t2" **\crv{(10,0)} ;
 \endxy\;\; = \;\;-r_{i}
   \xy 0;/r.17pc/:
  (-12,8)*{n};
  (0,0)*{\bbe{}};
 \endxy
\qquad \qquad
  \xy 0;/r.17pc/:
  (-14,8)*{n};
  (3,-10)*{};(-3,5)*{} **\crv{(3,-2) & (-2,1)}?(1)*\dir{>};?(.15)*\dir{>};
    (-3,-5)*{};(3,10)*{} **\crv{(-2,-1) & (3,2)}?(.85)*\dir{>} ?(.1)*\dir{>};
  (-3,5)*{}="t1";  (-9,5)*{}="t2";
  (-3,-5)*{}="t1'";  (-9,-5)*{}="t2'";
   "t1";"t2" **\crv{(-4,8) & (-9, 8)};
   "t1'";"t2'" **\crv{(-4,-8) & (-9, -8)};
   "t2'";"t2" **\crv{(-10,0)} ;
 \endxy \;\; = \;\; r_{i}
  \xy 0;/r.17pc/:
  (-12,8)*{n};
  (0,0)*{\bbe{}};
 \endxy
\end{equation}
\begin{equation}
 \vcenter{\xy 0;/r.17pc/:
  (-8,0)*{};
  (8,0)*{};
  (-4,10)*{}="t1";
  (4,10)*{}="t2";
  (-4,-10)*{}="b1";
  (4,-10)*{}="b2";(-6,-8)*{\scs };(6,-8)*{\scs };
  "t1";"b1" **\dir{-} ?(.5)*\dir{<};
  "t2";"b2" **\dir{-} ?(.5)*\dir{>};
  (10,2)*{n};
  \endxy}
\;\; = \;\;
 -\;\; r_{i}^{-2}\;
 \vcenter{   \xy 0;/r.17pc/:
    (-4,-4)*{};(4,4)*{} **\crv{(-4,-1) & (4,1)}?(1)*\dir{>};
    (4,-4)*{};(-4,4)*{} **\crv{(4,-1) & (-4,1)}?(1)*\dir{<};?(0)*\dir{<};
    (-4,4)*{};(4,12)*{} **\crv{(-4,7) & (4,9)};
    (4,4)*{};(-4,12)*{} **\crv{(4,7) & (-4,9)}?(1)*\dir{>};
  (8,8)*{n};(-6,-3)*{\scs };
     (6.5,-3)*{\scs };
 \endxy}
   \qquad \quad
 \vcenter{\xy 0;/r.17pc/:
  (-8,0)*{};(-6,-8)*{\scs };(6,-8)*{\scs };
  (8,0)*{};
  (-4,10)*{}="t1";
  (4,10)*{}="t2";
  (-4,-10)*{}="b1";
  (4,-10)*{}="b2";
  "t1";"b1" **\dir{-} ?(.5)*\dir{>};
  "t2";"b2" **\dir{-} ?(.5)*\dir{<};
  (10,2)*{n};
  \endxy}
\;\; = \;\;
 -\;\; r_{i}^{-2}\;
   \vcenter{\xy 0;/r.17pc/:
    (-4,-4)*{};(4,4)*{} **\crv{(-4,-1) & (4,1)}?(1)*\dir{<};?(0)*\dir{<};
    (4,-4)*{};(-4,4)*{} **\crv{(4,-1) & (-4,1)}?(1)*\dir{>};
    (-4,4)*{};(4,12)*{} **\crv{(-4,7) & (4,9)}?(1)*\dir{>};
    (4,4)*{};(-4,12)*{} **\crv{(4,7) & (-4,9)};
  (8,8)*{n};(-6,-3)*{\scs };  (6,-3)*{\scs };
 \endxy}
\end{equation}

\begin{rem}
These extended relations are useful because they allow one to simplify any closed diagram (one without loose ends) into a sum of diagrams containing no crossings.
\end{rem}

%
\subsection{An alternative characterization of the extended ${\mf{sl}}_2$ relations}
%

The extended ${\mf{sl}}_2$ relations in section \ref{sec:highersl2} are somewhat mysterious. On the other hand, there is an alternative, simpler way, to define or summarize them.

We will use the sideways crossings defined in \eqref{eq_crossl-gen} and \eqref{eq_crossr-gen} as well as the fake bubbles defined in section~\ref{subsubsec-fake}. It turns out that the extended ${\mf{sl}}_2$ conditions are equivalent to the following conditions:
\begin{itemize}
\item If $n \ge 0$ then the 2-morphism
\begin{equation}\label{eq_phi-minus}
\xy 0;/r.15pc/:
    (-4,-4)*{};(4,4)*{} **\crv{(-4,-1) & (4,1)}?(0)*\dir{<} ;
    (4,-4)*{};(-4,4)*{} **\crv{(4,-1) & (-4,1)}?(1)*\dir{>};
  \endxy \;\; \bigoplus_{k=0}^{n-1}
\vcenter{\xy 0;/r.15pc/:
 (-4,2)*{}="t1"; (4,2)*{}="t2";
 "t2";"t1" **\crv{(4,-5) & (-4,-5)}; ?(1)*\dir{>}
 ?(.8)*\dir{}+(0,-.1)*{\bullet}+(-3,-1)*{\scs k};;
 \endxy}: \sF \sE \onen \bigoplus_{k=0}^{n-1} \onen \la n-1-2k \ra \rightarrow \sE \sF \onen
\end{equation}
is an isomorphism with the inverse given by
 \begin{equation} -r_i^{-2} \;
\xy 0;/r.15pc/:
    (-4,-4)*{};(4,4)*{} **\crv{(-4,-1) & (4,1)}?(1)*\dir{>} ;
    (4,-4)*{};(-4,4)*{} **\crv{(4,-1) & (-4,1)}?(0)*\dir{<};
  \endxy \;\; \bigoplus_{k=0}^{n-1}
\left(
 \sum_{\xy (0,0)*{\scs j_1+j_2 =}; (0,-3)*{\scs n-1 -k}; \endxy }
    \; \vcenter{\xy 0;/r.15pc/:
    (-4,-2)*{}="t1";  (4,-2)*{}="t2";
     "t2";"t1" **\crv{(4,5) & (-4,5)};  ?(.1)*\dir{<} 
     ?(.85)*\dir{}+(0,-.1)*{\bullet}+(-2,2.5)*{\scs j_1};
 (2,13)*{\ccbub{\;\;-n-1+j_2}{}};  \endxy} \right)
 : \sE \sF \onen \rightarrow\sF \sE \onen \bigoplus_{k=0}^{n-1} \onen \la n-1-2k \ra
\end{equation}
\item If $n \le 0$ then the 2-morphism
\begin{equation}\label{eq_phi-plus}
\xy 0;/r.15pc/:
    (-4,-4)*{};(4,4)*{} **\crv{(-4,-1) & (4,1)}?(1)*\dir{>} ;
    (4,-4)*{};(-4,4)*{} **\crv{(4,-1) & (-4,1)}?(0)*\dir{<};
  \endxy \;\; \bigoplus_{k=0}^{-n-1}
\vcenter{\xy 0;/r.15pc/:
 (4,2)*{}="t1"; (-4,2)*{}="t2";
 "t2";"t1" **\crv{(-4,-5) & (4,-5)}; ?(1)*\dir{>}
 ?(.8)*\dir{}+(0,-.1)*{\bullet}+(3,-1)*{\scs k};;
 \endxy}:  \sE \sF \onen \bigoplus_{k=0}^{-n-1} \onen \la -n-1-2k \ra \rightarrow \sF \sE  \onen
\end{equation}
is an isomorphism with the inverse given by
\begin{equation}
-r_i^{-2} \; \xy 0;/r.15pc/:
    (-4,-4)*{};(4,4)*{} **\crv{(-4,-1) & (4,1)}?(1)*\dir{>} ;
    (4,-4)*{};(-4,4)*{} **\crv{(4,-1) & (-4,1)}?(0)*\dir{<};
  \endxy \;\; \bigoplus_{k=0}^{-n-1}
\left(
 \sum_{\xy (0,0)*{\scs j_1+j_2 =}; (0,-3)*{\scs -n-1 -k}; \endxy }
    \vcenter{\xy 0;/r.15pc/:
    (4,-2)*{}="t1";  (-4,-2)*{}="t2";
     "t2";"t1" **\crv{(-4,5) & (4,5)};  ?(.1)*\dir{<} 
     ?(.85)*\dir{}+(0,-.1)*{\bullet}+(5,1)*{\scs j_1};
 (2,13)*{\cbub{\;\;n-1+j_2}{}};  \endxy} \right)
 :  \sF \sE \onen \rightarrow \sE \sF \onen \bigoplus_{k=0}^{-n-1} \onen \la -n-1-2k \ra.
\end{equation}
\end{itemize}
For more details see \cite[Section 3]{Lau3}.

%
\section{Formal structures in $Q$-strong 2-representations }
%

In this section we assume $\mathfrak{g} = \mathfrak{sl}_2$. By studying the spaces of maps between certain 1-morphisms in a $Q$-strong 2-representation we prove two main results. The first is that $\E$ and $\F$ are biadjoint. The second main result is Lemma \ref{lem:main} which describes $\Hom^m(\E {\bf 1}_n, \E {\bf 1}_n)$ when $m \le 2|n+1|$. This Lemma is an important tool used in the proof of Theorem \ref{thm:main}.

The proof of biadjointness proceeds in several steps.  We first show that there is a one-dimensional space of 2-morphisms for units and counits in these adjunctions.  Hence, the caps and cups are unique up to a scalar.  In this section we fix these maps arbitrarily.  In section~\ref{subsec:adjoint} we rescale these maps appropriately.

%
\subsection{Some general notions}
%

The fact that the space of maps between any two 1-morphisms in a $Q$-strong 2-representation of $\g$ is finite dimensional means that the Krull-Schmidt property holds for $\Hom$ categories. This means that any 1-morphism has a unique direct sum decomposition (see section 2.2 of \cite{Rin}). In particular, this means that if $A,B,C$ are morphisms and $V$ is a $\Z$-graded vector space then we have the following cancellation laws (see section 4 of \cite{CK3}):
\begin{eqnarray*}
A \oplus B \cong A \oplus C &\Rightarrow& B \cong C \\
A \otimes_\k V \cong B \otimes_\k V &\Rightarrow& A \cong B.
\end{eqnarray*}

A brick in a (graded) category is an indecomposable object $A$ such that $\End(A) =\Bbbk$. For example, by Lemma \ref{lem:1} below, $\E \1_m$ is a brick.

Suppose that $ A $ is a brick and that $ X, Y $ are arbitrary objects. Then a morphism $f : X \rightarrow Y $ gives rise to a bilinear pairing $ \Hom(A, X) \times \Hom(Y, A) \rightarrow \Hom(A,A) = \k $.  We define the $A$-rank $\rk_A^0(f)$ of $f$ to be the rank of this bilinear pairing.

We can also define $A$-rank as follows. Choose (non-canonical) direct sum decompositions $X = A \otimes V \oplus B$ and $Y = A \otimes V' \oplus B'$ where $V, V' $ are $\Bbbk$ vector spaces, and $B, B'$ do not contain $A$ as a direct summand. Then one of the matrix coefficients of $ f $ is a map $ A \otimes V \rightarrow A \otimes V' $, which (since $ A $ is a brick) is equivalent to a linear map $ V \rightarrow V' $.  The $A$-rank of $ f $ equals the rank of this linear map.

We define the total $A$-rank $\rk_A(f)$ of $f$ as $\sum_{i} \rk^0_{A \la i \ra}(f)$. In this paper, this will always turn out to be a finite integer. If the total $A$-rank of $f$ is $r$ then we say that ``$f$ gives an isomorphism on $r$ summands of the form $A \la \cdot \ra$''.

%
\subsection{The adjoint induction hypothesis and its consequences}
%

For a 1-morphism $u$, let $u_L$ (respectively $u_R$) denote its left (respectively right) adjoint.
Our definition of $\1_{n}\F$ implies that $(\E\1_{n})_R = \1_{n}\F \la n + 1\ra $ and that
$(\1_{n}\F)_L =\E\1_{n} \la n + 1\ra$.
In what follows we make use of the fact that  $(u_L)_R=u$ and $(v_R)_L=v$ for all 1-morphisms $u,v$ and that the adjunctions give rise to isomorphisms
\begin{alignat}{3}
 \Hom(ux,y) &\; \cong \;& \Hom (x, u_R y),
 &\qquad \Hom(x,uy)&\; \cong \;&  \Hom (u_Lx,  y), \nn\\
 \Hom(xv,y) &\; \cong \;& \Hom(x,yv_L),  &\qquad \Hom(x,yv) &\; \cong \;& \Hom(xv_R,y).
\end{alignat}

We now make use of the integrability assumption in the definition of a $Q$-strong 2-representation. For positive weight spaces $(n\geq 0)$ we proceed by decreasing induction on $n$ starting from the highest weight.  For negative weight spaces $(n \leq 0)$ we perform increasing induction on $n$ starting from the lowest weight.  To simplify the exposition we focus on the case $n \geq 0$. Throughout this section we make the following assumption:

\medskip
\noindent\fbox {
   \parbox{\linewidth}{
{\bf Adjoint induction hypothesis:}  \newline Fix $n\geq 0$.  By induction assume that if $m > n$ we have
\begin{equation} \label{eq:ind_hyp}
 (\E {\1}_m)_L \cong {\1}_m \F \la -m-1 \ra\text{ and } ({\1}_m\F)_R \cong \E{\1}_m \la -m-1 \ra,
\end{equation}
where for $m$ beyond the highest weight this claim is vacuously true since both maps are zero.
    }
}

\medskip

In this section we derive a number of consequences of the adjoint induction hypothesis \eqref{eq:ind_hyp} culminating in the proof that \eqref{eq:ind_hyp} holds for $m=n$.

\begin{lem}\label{lem:1}
Assuming the adjoint induction hypothesis \eqref{eq:ind_hyp}, if $m \ge n$ then $\Hom(\E \1_m, \E \1_m \la l \ra)$ is zero if $l < 0$ and one-dimensional if $l=0$.  The same is true for $\Hom(\1_m \F, \1_m \F \la l \ra)$.
\end{lem}
\begin{proof}
We prove the result for $\E$ by (decreasing) induction on $m$ (the result for $\F$ follows by adjunction). We have
\begin{eqnarray*}
& & \Hom(\E \1_m, \E \1_m \la l \ra) \\
&\cong& \Hom(\E (\E \1_m)_R \1_{m+2}, \1_{m+2} \la l \ra) \\
&\cong& \Hom(\E \F \1_{m+2} \la m+1 \ra, \1_{m+2} \la l \ra) \\
&\cong& \Hom(\F \E \1_{m+2} \bigoplus_{[m+2]} \1_{m+2}, \1_{m+2} \la l-m-1 \ra) \\
&\cong& \Hom((\E \1_{m+2})_L \E \1_{m+2} \la m+3 \ra, \1_{m+2} \la l-m-1 \ra) \oplus \Hom(\1_{m+2}, \bigoplus_{[m+2]} \1_{m+2} \la l-m-1 \ra) \\
&\cong& \Hom(\E \1_{m+2}, \E \1_{m+2} \la l-2m-4 \ra) \bigoplus_{k=0}^{m+1} \Hom(\1_{m+2}, \1_{m+2} \la l-2k \ra).
\end{eqnarray*}
By induction the first term above is zero and, by condition (\ref{co:hom}) of Definition \ref{def_Qstrong}, all the terms in the direct sum are zero unless $k=0$ and $l=0$. In that case we get $\Hom(\1_{m+2}, \1_{m+2}) \cong \k$ and we are done.
\end{proof}

\begin{cor}\label{cor:0} Assuming the adjoint induction hypothesis \eqref{eq:ind_hyp}, if $m \ge n-2$ then $\Hom(\E \E \1_m, \E \E \1_m \la l \ra)$ is zero if $l < -2$ and one-dimensional if $l=-2$.
\end{cor}
\begin{proof}
The proof is by (decreasing) induction on $m$. We have
\begin{eqnarray*}
& & \Hom(\E \E \1_m, \E \E \1_m) \\
&\cong& \Hom(\E \E (\E \1_m)_R \1_{m+2}, \E \1_{m+2}) \\
&\cong& \Hom(\E \E \F \1_{m+2} \la m+1 \ra, \E \1_{m+2}) \\
&\cong& \Hom(\E \F \E \1_{m+2}, \E \1_{m+2} \la -m-1 \ra) \bigoplus_{k=0}^{m+1} \Hom(\E \1_{m+2}, \E \1_{m+2} \la m+1-2k-m-1 \ra) \\
&\cong& \Hom(\F \E \E \1_{m+2}, \E \1_{m+2} \la -m-1 \ra) \bigoplus_{k=0}^{m+3} \Hom(\E \1_{m+2}, \E 1_{m+2} \la m+3-2k-m-1 \ra) \\
& & \bigoplus_{k=0}^{m+1} \Hom(\E \1_{m+2}, \E \1_{m+2} \la -2k \ra) \\
&\cong& \Hom(\E \E \1_{m+2}, \E \E \1_{m+2} \la -2m-6 \ra) \bigoplus_{k=0}^{m+3} \Hom(\E \1_{m+2}, \E_{m+2} \la -2k+2 \ra) \\
& & \bigoplus_{k=0}^{m+1} \Hom(\E \1_{m+2}, \E \1_{m+2} \la -2k \ra).
\end{eqnarray*}
Shifting by $\la l \ra$ where $l < -2$ we find that the first term is zero by induction, and the others are zero by Lemma \ref{lem:1}. If $l=-2$ the same vanishing holds with the exception of the term in the middle summation when $k=0$, which yields $\End(\E \1_{m+2}) \cong \k$.
\end{proof}

\begin{cor}\label{cor:2} Assuming the adjoint induction hypothesis \eqref{eq:ind_hyp}, if $m \ge n$ then $\Hom(\E \1_m, \E \E \F \1_m \la -m-1 \ra) \cong \k$.
\end{cor}
\begin{proof}
Moving the $\F$ past the $\E$'s using the ${\mathfrak{sl}}_2$ commutation relation we get
\begin{eqnarray*}
& & \Hom(\E \1_m, \E \E \F \1_m \la -m-1 \ra) \\
&\cong& \Hom(\E \1_m, \F \E \E \1_m \la -m-1 \ra) \bigoplus_{k=0}^{m+1} \Hom(\E \1_m,\E \1_m \la -2k \ra) \bigoplus_{k=0}^{m-1} \Hom(\E \1_m, \E \1_m \la -2k-2 \ra).
\end{eqnarray*}
By Lemma \ref{lem:1} all the terms in the middle and right summations are zero except in the middle summation when $k=0$, in which case we get $\End(\E \1_m) \cong \k$. The term on the left is equal to
$$\Hom((\1_{m+2} \F)_L \E \1_m, \E \E \1_m \la -m-1 \ra) \cong \Hom(\E \E \1_m, \E \E \1_m \la -2m-4 \ra)$$
which vanishes by Corollary \ref{cor:0} since $n \ge 0$.
\end{proof}

\begin{lem} \label{lem:XXX}
 Assuming the adjoint induction hypothesis \eqref{eq:ind_hyp}, if $m \ge n$ then $\Hom(\E\F \1_m, \F \E  \1_m \la l \ra)$ is is zero if $l < 0$ and one-dimensional if $l=0$.
\end{lem}

\begin{proof}
Our assumption \eqref{eq_defF} implies that $(\F\1_m)_L=\1_m\E\la m -1\ra$ for all $m$.  Using this fact, the claim reduces to Corollary~\ref{cor:0}.
\end{proof}

\begin{remark} \label{rem:XXX}
Lemma~\ref{lem:XXX} implies that the map $\E\F\1_{n} \cong \F\E\1_n \oplus_{[n]}\1_n \to \F\E\1_n$ which projects $\E\F\1_n$ onto the $\F\E\1_n$ summand is unique up to a scalar multiple so it must induce an isomorphism from the $\F\E\1_{n}$ summand.  This projection must also be the zero map from any summand $\1_n \la n-1-2k \ra$ on the left to $\F\E\1_n$ on the right.
\end{remark}

%
%


The following Lemma will be key for constructing the adjunction maps.  The proof would be greatly simplified by knowing that $\E\1_{n-2}$ and $\F\E\1_n$ were indecomposable.  However, we have not yet established these facts.

\begin{lem}\label{lem:Xind} Suppose that $\1_n \ne 0$. If $n > 1$ then $\sUupdot \sUdown: \E \F \1_n \rightarrow \E \F \1_n \la 2 \ra$ induces an isomorphism on $n-1$ summands of the form $\1_n \la k \ra$, in other words $\rk_{\1_n}(\sUupdot \sUdown) = n-1$.

Similarly, if $n < -1$ then $\sUdown \sUupdot: \F \E \1_n \rightarrow \F \E \1_n \la 2 \ra$ induces an isomorphism on $-n-1$ summands of the form $\1_n \la k \ra$, in other words $\rk_{\1_n}(\sUdown \sUupdot) = -n-1$.
\end{lem}
\begin{rem} If $n = -1, 0 ,1$ the statement above is vacuous (which is why we only consider $n > 1$ and $n < -1$).
\end{rem}
\begin{proof}
We prove the case $n > 1$ (the case $n < -1$ is proved in the same way). To do this we only use the commutator $\mf{sl}_2$ relation together with the fact that $\sUupdot \sUup: \E \E \1_n \rightarrow \E \E \1_n \la 2 \ra$ induces a map
$$\E^{(2)} \1_n \la -1 \ra \oplus \E^{(2)} \1_n \la 1 \ra \rightarrow \E^{(2)} \1_n \la 1 \ra \oplus \E^{(2)} \1_n \la 3 \ra$$
which is an isomorphism on the common summand $\E^{(2)} \la 1 \ra$.
 This is a consequence of the NilHecke algebra action.  Using this fact, the map
\begin{equation}\label{eq:1}
\sUupdot \sUup \sUdown: \E \E \F \1_{n-2} \rightarrow \E \E \F \1_{n-2} \la 2 \ra.
\end{equation}
induces an isomorphism on the summand $\E^{(2)} \F \1_{n-2} \la 1 \ra$ (note that $\E^{(2)}\F \1_{n-2}$ is not zero, since $\E^{(2)}\F\F\1_n$ contains $\1_n$ as a summand (this uses $n>1$), and $\1_n \ne 0$).

Temporarily assume that $\E\1_{n-2}$ is indecomposable.  Then, by  \cite[Lemma 4.2]{CKL3} we know that
$$\E^{(2)} \F \1_{n-2} \cong \F \E^{(2)} \1_{n-2} \bigoplus_{i=0}^{n-2} \E \1_{n-2} \la n-2-2i \ra$$
which means that $\rk_{\E}(\sUupdot\sUup\sUdown) \ge n-1$.

On the other hand, by decomposing $\E\F\1_{n-2}$, the map in \eqref{eq:1} induces the map
$$\sUupdot\sUdown\sUup \bigoplus_{i=0}^{n-3} \sUupdot: \E \F \E \1_{n-2} \bigoplus_{i=0}^{n-3} \E \1_{n-2} \la n-3-2i \ra \rightarrow \E \F \E \1_{n-2} \la 2 \ra \bigoplus_{i=0}^{n-3} \E \1_{n-2} \la n-1-2i \ra.$$
The maps $\sUupdot \maps  \E \1_{n-2} \la n-3-2i \ra \to \E \1_{n-2} \la n-1-2i \ra$ are clearly never isomorphisms.  Hence, the only contribution to the $\E$-rank of (\ref{eq:1}) comes from the map
$$\sUupdot\sUdown\sUup: \E \F \E \1_{n-2} \rightarrow \E \F \E \1_{n-2} \la 2 \ra.$$
To compute the $\E$-rank of this map, note that by Remark~\ref{rem:XXX} the projection map $\E\F\E\1_{n-2} \to\F\E\E\1_{n-2}$ is zero on all summands $(1_n\la n -1-2k \ra)\E\1_{n-2}$.  Then
$\E \F \E \1_{n-2} \cong \F \E \E \1_{n-2} \bigoplus_{i=0}^{n-1} \E \1_{n-2} \la n-1-2i \ra$ and $\F \E \E \1_{n-2}$ does not contribute to the $\E$-rank. Thus, we have shown
\[
\rk_{\E}(\sUupdot\sUdown\sUup) = \rk_{\E}(\sUupdot\sUup\sUdown) = n-1.
\]

To complete the proof, note that $\rk_{\E\1_{n-2}}(\sUupdot\sUdown\sUup) = n-1$ is equal to $\rk_{\1_n}(\sUupdot\sUdown)$. This is because the summand $\F \E \E \1_{n-2}$ of $\E \F \E \1_{n-2}$ does not contain any summands $\E \1_{n-2}$. To see this suppose $\F\E\E\1_{n-2} = \E\1_{n-2} + A$ for some $A$. Applying $\F$ on the right gives $\F\E\E\F\1_n = \E\F\1_n + A \F \1_n$ and simplifying both sides gives that either $\F\E\1_n$ or $\F^{(2)}\E^{(2)} \1_n$ contains a copy of $\1_n$. To see this is impossible, one shows using Lemma~\ref{lem:1} and Corollary~\ref{cor:0} that the space of inclusions $\Hom(\1_n \la l \ra, \F\E\1_n)$ or $\Hom(\1_n \la l \ra, \F^{(2)}\E^{(2)}\1_n)$ is zero.


Finally, if $\E\1_{n-2}$ is decomposable then the argument in the paragraph above shows that it must contain some indecomposable $X \1_{n-2}$ so that $\F \E \E \1_{n-2}$ does not contain $X \1_{n-2}$. Now suppose that $X \1_{n-2}$ occurs with multiplicity $\mu$ inside $\E \1_{n-2}$. Then the same argument as above shows that $\rk_{X \1_{n-2}}(\sUupdot\sUdown\sUup) = \mu(n-1)$ and subsequently $\mu \cdot \rk_{\1_n}(\sUupdot \sUdown) = \mu(n-1)$ which gives $\rk_{\1_n}(\sUupdot \sUdown) = n-1$.
\end{proof}


%
\subsection{Consequences of $Q$-strongness} \label{sec:consequence-sl2}
%

The structure of a $Q$-strong 2-representation imposes strong conditions on the Hom spaces between various maps.
For $n \ge 0$ the $\mf{sl}_2$-relation
determines a 2-morphism $\Ucupl: \1_{n+2} \la n+1 \ra \rightarrow  \E \F \1_{n+2}$ as follows:
\begin{eqnarray*}
& & \Hom({\1}_{n+2} \la n+1 \ra, \E \F {\1}_{n+2}) \\&\cong& \Hom({\1}_{n+2} \la n+1 \ra, \F \E \1_{n+2} \bigoplus_{[n+2]} \1_{n+2}) \\&\cong& \Hom(\1_{n+2} \la n+1 \ra, (\E \1_{n+2})_R \E \1_{n+2} \la -n-3 \ra) \bigoplus_{k=0}^{n+1} \Hom(\1_{n+2}, \1_{n+2} \la -2k \ra) \\
&\cong& \Hom(\E \1_{n+2}, \E \1_{n+2} \la -2n-4 \ra) \oplus \Hom(\1_{n+2}, \1_{n+2}).
\end{eqnarray*}
We take $\Ucupl$ to be the identity map in the second summand. This corresponds to the inclusion of $\1_{n+2}$ into the lowest degree summand of $\1_{n+2}$ inside $\E \F \1_{n+2}$. Similarly, we define a 2-morphism $\Ucapr: \E \F \1_{n+2} \la n+1 \ra \rightarrow \1_{n+2}$ as the projection out of the top degree summand $\1_{n+2}$ in $\E \F \1_{n+2}$.

By fixing an arbitrary choice of unit and counit realizing $\1_{\l}\F\la n+1\ra$ as the right adjoint of $\E\1_{\l}$, we can define a 2-morphism $\Ucupr: \1_n \rightarrow \F \E \1_n \la n+1 \ra$ as the unit map for this adjunction.
All of the arguments in this section only make use of the fact that this map is chosen as the unit map for some choice of adjoint structure.  In Corollary~\ref{cor:EFidhoms} we show that the units and counits are all unique up to a scalar.  This choice is fixed in Section~\ref{subsec:adjoint}.

At this stage we do not yet define the last adjunction map $\Ucapl: \F \E \1_n \rightarrow \1_n \la n+1 \ra$ for $n \geq 0$.  This map is more difficult to define and is constructed (up to scalar) in subsection~\ref{subsubsec:last}.  The proof adjunction axioms are verified by decreasing induction on $n$ starting from the highest weight. In particular, we use here that the representation is integrable.

The $\mf{sl}_2$-relations from Definition~\ref{def_Qstrong} \eqref{co:EF} defines similar maps for $n \le 0$ -- see Remark \ref{rem:nle0}. In this case, one gets maps $\Ucapl: \F \E \1_n \rightarrow \1_n \la n+1 \ra$  and $\Ucupr: \1_n \rightarrow \F \E \1_n \la n+1 \ra$ using the decomposition of $\F\E\1_n$ and the map $\Ucapr: \E \F \1_{n+2} \la n+1 \ra \rightarrow \1_{n+2}$ by fixing a choice of adjoint structure.


\begin{cor} \label{cor:degz-bubbles}
The maps
 \begin{align}
  \xy 0;/r.18pc/:
 (0,0)*{\cbub{n-1}};
  (4,8)*{n};
 \endxy &\maps \1_n \to \1_n  \qquad \text{for $n >0$,}
 \\ \nn \\
  \xy 0;/r.18pc/:
 (0,0)*{\ccbub{-n-1}};
  (4,8)*{n};
 \endxy &\maps \1_n \to \1_n  \qquad \text{for $n <0$,}
\end{align}
are both equal to some non-zero multiple of $\Id_{\1_n}$.
\end{cor}

\begin{proof}
As usual, we prove the case $n > 0$ (the case $n < 0$ follows similarly).

By construction, the map $\;\vcenter{\xy 0;/r.14pc/: (-4,2)*{}="t1"; (4,2)*{}="t2"; "t2";"t1" **\crv{(4,-5) & (-4,-5)}; ?(1)*\dir{>};
\endxy}: \1_n \rightarrow \E \F \1_n \la -n+1 \ra$ is an isomorphism between the summand $\1_n$ on the left and the corresponding $\1_n$ on the right hand side. Then applying $^{n-1}\Uupdot\Udown: \E \F \1_n \la -n+1 \ra \rightarrow \E \F \1_n \la n-1 \ra$ induces an isomorphism between the summands $\1_n$ on either side (this is a corollary of Lemma \ref{lem:Xind}).

Finally, again by construction, the map $\xy 0;/r.15pc/: (0,-4)*{};
 (-4,-2)*{}="t1"; (4,-2)*{}="t2";
 "t2";"t1" **\crv{(4,5) & (-4,5)}; ?(0)*\dir{<};
 \endxy: \E \F \1_n \la n-1 \ra \rightarrow \1_n$ is an isomorphism between $\1_n$ on the right side and the corresponding summand $\1_n$ on the left hand side. Thus the composition
\[\1_n \xrightarrow{\vcenter{\xy 0;/r.15pc/:
 (-4,2)*{}="t1"; (4,2)*{}="t2";
 "t2";"t1" **\crv{(4,-5) & (-4,-5)}; ?(1)*\dir{>};
 \endxy}} \E \F  \1_n \la -n+1 \ra \xrightarrow{n-1 \Uupdot \Udown} \E \F \1_n \la n-1 \ra \xrightarrow{\xy 0;/r.15pc/: (0,-4)*{};
 (-4,-2)*{}="t1"; (4,-2)*{}="t2";
 "t2";"t1" **\crv{(4,5) & (-4,5)}; ?(0)*\dir{<};
 \endxy} \1_n
 \]
is an isomorphism and this completes the proof since $\Hom(\1_n, \1_n) \cong \k$.
\end{proof}



\begin{lem}\label{lem:A'} If $m \ge n$ then the following two maps
\[
\xy 0;/r.15pc/:
    (-12,-4)*{};(-12,4)*{} **\dir{-}?(1)*\dir{>} ;
    (-4,-4)*{};(4,4)*{} **\crv{(-4,-1) & (4,1)};
    (4,-4)*{};(-4,4)*{} **\crv{(4,-1) & (-4,1)}?(1)*\dir{>};
    (-4,-4)*{};(-12,-4)*{} **\crv{(-4,-8) & (-12,-8)};
    (4,-10)*{}; (4,-4)*{} **\dir{-};
    (16,-2)*{m};
\endxy
\qquad \qquad \qquad
\xy 0;/r.15pc/:
    (12,-4)*{};(12,4)*{} **\dir{-} ;
    (4,-4)*{};(-4,4)*{} **\crv{(4,-1) & (-4,1)}?(1)*\dir{>};
    (-4,-4)*{};(4,4)*{} **\crv{(-4,-1) & (4,1)}?(1)*\dir{>};
    (4,-4)*{};(12,-4)*{} **\crv{(4,-8) & (12,-8)};
    (-4,-10)*{}; (-4,-4)*{} **\dir{-};
    (19,-4)*{m};
\endxy
\]
are non-zero multiples of the unique 2-morphism $\E \1_m \rightarrow \E \E \F \1_m \la -m-1 \ra$ from Corollary \ref{cor:2}.
\end{lem}
\begin{proof}
We just need to show that both maps are non-zero. Suppose the map on the right is zero. Adding a dot at the top of the middle upward pointing strand and sliding it past the crossing using the NilHecke relation~\eqref{eq_nil_dotslide} one gets two terms. One term is again zero (because it is the composition of a dot and the original map) and the other is just the adjoint map
$$\E \1_m \xrightarrow{\sUup \Ucupl} \E \E \F \1_m \la -m+1 \ra$$
which cannot be zero because it is the inclusion of $\E \1_m$ into the lowest degree summand inside $\E \E \F \1_m \la -m+1 \ra$. Thus the map on the right must be non-zero.

On the other hand, the map on the left is the composition
\[
\xy
(-40,0)*+{\E \1_m}="1";
(0,0)*+{ \E \F \E \1_m \la -m-1 \ra}="2";
(50,0)*+{ \E \E \F \1_m \la -m-1 \ra}="3";
 {\ar^-{\Ucupl \sUup} "1";"2"};
 {\ar^-{\xy 0;/r.15pc/:
    (-12,-4)*{};(-12,4)*{} **\dir{-}?(1)*\dir{>} ;
    (-4,-4)*{};(4,4)*{} **\crv{(-4,-1) & (4,1)}?(0)*\dir{<} ;
    (4,-4)*{};(-4,4)*{} **\crv{(4,-1) & (-4,1)}?(1)*\dir{>};
\endxy} "2";"3"};
\endxy
\]
where the first map is an inclusion of $\E \1_m$ into the lowest degree copy of $\E \1_m$ in $(\E \F) \E \1_m$ and the second map is induced by the inclusion of $\F \E \1_{m}$ into $\E \F \1_{m} \cong \E \F \1_{m} \oplus_{[m]} \1_{m}$. Since this is the composition of two inclusions it is also non-zero.
\end{proof}

\subsubsection{Defining the last adjunction} \label{subsubsec:last}

We are still assuming, for simplicity, that $n \ge 0$. Recall that we defined all but one of the adjunction maps, namely the 2-morphism $\Ucapl: \F \E \1_n \rightarrow \1_n \la n+1 \ra$. This map cannot be defined formally by adjunction and since $n \ge 0$ the 1-morphism $\F \E \1_n$ is indecomposable so one cannot define it as an inclusion. To overcome this problem we construct this 2-morphism by defining an up-down crossing and composing it with the map $\Ucapr$ defined in section~~\ref{sec:consequence-sl2}.

We define the crossing
$\;\xy 0;/r.14pc/:
    (-4,-4)*{};(4,4)*{} **\crv{(-4,-1) & (4,1)}?(0)*\dir{<} ;
    (4,-4)*{};(-4,4)*{} **\crv{(4,-1) & (-4,1)}?(1)*\dir{>};
\endxy$
as the map $\F \E \1_n \rightarrow \E \F \1_n \cong \F \E \1_n \oplus_{[n]} \1_n$ which includes $\F \E \1_n$ into $\E \F \1_n$. Note that this map is unique (up to scalar) because by a simple adjunction calculation there exist no non-zero maps $\F \E \1_n \rightarrow \oplus_{[n]} \1_n$. We then define the 2-morphism $\Ucapl$ as the composite
\begin{equation}\label{eq:C}
\xy
 (-40,0)*+{\F \E \1_n}="1";
 (-15,0)*+{\E \F \1_n }="2";
 (15,0)*+{\E \F \1_n \la 2n \ra}="3";
 (45,0)*+{\1_n \la n+1 \ra.}="4";
 {\ar^-{\xy 0;/r.15pc/:
    (-4,-4)*{};(4,4)*{} **\crv{(-4,-1) & (4,1)}?(0)*\dir{<} ;
    (4,-4)*{};(-4,4)*{} **\crv{(4,-1) & (-4,1)}?(1)*\dir{>};
\endxy} "1"; "2"};
 {\ar^-{n\Uupdot\Udown} "2"; "3"};
 {\ar^-{\Ucapr} "3"; "4"};
\endxy
\end{equation}
A similar definition of the cap 2-morphism appears in \cite[Proposition 5.4]{Lau1} and \cite[Section 4.1.4]{Rou2}.

\subsection{Biadjointness}

Recall that in a $Q$-strong 2-representation we define $\1_n \F$ as the right adjoint of $\E \1_n$ together with a grading shift~\eqref{eq_defF}. We now show that $\E$s and $\F$s are biadjoint to each other (up to a grading shift).

\begin{prop}\label{prop:lradj} Assuming the adjoint induction hypothesis \eqref{eq:ind_hyp}, then given a strong 2-representation of $\mf{sl}_2$ the left and right adjoints of $\E$ are isomorphic up to specified shifts. More precisely $(\E \1_n)_L \cong (\E \1_n)_R \la -2(n + 1) \ra$ given by the 2-morphisms
\begin{equation}\label{eq:A}
\Ucapr: \E \F \1_{n+2} \la n+1 \ra \rightarrow \1_{n+2} \hspace{1.0cm}
\Ucupl: \1_{n+2} \la n+1 \ra \rightarrow  \E \F \1_{n+2}
\end{equation}
\begin{equation}\label{eq:B}
\Ucapl: \F \E \1_n \rightarrow \1_n \la n+1 \ra \hspace{1.0cm}\Ucupr: \1_n \rightarrow \F \E \1_n \la n+1 \ra
\end{equation}
defined above satisfy the adjunction relations (\ref{eq_biadjoint1}) and (\ref{eq_biadjoint2}) up to non-zero multiples.
\end{prop}

\begin{rem} We will choose specific scalars for the 2-morphisms above in section \ref{subsec:adjoint}. The result above implies that
\begin{enumerate}
\item $(\E^{(r)} {\1}_n)_R \cong {\1}_n \F^{(r)} \la r(n + r) \ra$,
\item $(\E^{(r)} {\1}_n)_L \cong {\1}_n \F^{(r)} \la -r(n + r) \ra$.
\end{enumerate}
\end{rem}

\begin{proof}
We prove the adjunction identity on the left of (\ref{eq_biadjoint1}) (the second one follows formally). Since $\End(\E\1_n) \cong \k$ it suffices to show that the left side of (\ref{eq_biadjoint1}) is non-zero. Now, $\Hom(\E \F \1_{n+2} \la n+1 \ra, \1_{n+2}) \cong \Hom(\E \1_n, \E \1_n) \cong \k$ by Lemma \ref{lem:1}. So the map $\Ucapr: \E \F \1_{n+2} \la n+1 \ra \rightarrow \1_{n+2}$ must be equal to the adjunction map (up to a multiple). Since $\Ucupr: \1_n \rightarrow \F \E \1_n \la n+1 \ra$ is also a non-zero multiple of the adjunction map their composition must be non-zero.

Likewise, we now prove the adjunction identity on the left in (\ref{eq_biadjoint2}) (the second one follows formally) by showing that the left side of (\ref{eq_biadjoint2}) is non-zero. Using Lemma \ref{lem:A'} it follows that up to non-zero multiples the left side of (\ref{eq_biadjoint2}) is equal to the composition
\begin{equation}
  \xy 0;/r.17pc/:
  (14,8)*{n};
  (-3,-10)*{};(3,5)*{} **\crv{(-3,-2) & (2,1)}??(.15)*\dir{>};
  (3,-5)*{};(-3,10)*{} **\crv{(2,-1) & (-3,2)}?(.85)*\dir{>} ?(.1)*\dir{>};
  (3,5)*{}="t1";  (9,5)*{}="t2";
  (3,-5)*{}="t1'";  (9,-5)*{}="t2'";
   "t1";"t2" **\crv{(4,8) & (9, 8)} ;
   "t1'";"t2'" **\crv{(4,-8) & (9, -8)};
   "t2'";"t2" **\crv{(10,0)} ?(.5)*\dir{<};
  (3,4)*{\bullet}+(-1,4)*{\scs n};
 \endxy\;\;
\end{equation}
Moving one of the dots through the crossing using the NilHecke relation~\eqref{eq_nil_dotslide} gives
\begin{equation}
  \xy 0;/r.17pc/:
  (14,8)*{n};
  (-3,-10)*{};(3,5)*{} **\crv{(-3,-2) & (2,1)}??(.15)*\dir{>};
  (3,-5)*{};(-3,10)*{} **\crv{(2,-1) & (-3,2)}?(.85)*\dir{>} ?(.1)*\dir{>};
  (3,5)*{}="t1";  (9,5)*{}="t2";
  (3,-5)*{}="t1'";  (9,-5)*{}="t2'";
   "t1";"t2" **\crv{(4,8) & (9, 8)} ;
   "t1'";"t2'" **\crv{(4,-8) & (9, -8)};
   "t2'";"t2" **\crv{(10,0)} ?(.5)*\dir{<};
  (3,4)*{\bullet}+(-1,4)*{\scs n};
 \endxy\;\; = \quad   \xy 0;/r.17pc/:
  (16,8)*{n};
  (-3,-10)*{};(3,5)*{} **\crv{(-3,-2) & (2,1)}??(.15)*\dir{>};
  (3,-5)*{};(-3,10)*{} **\crv{(2,-1) & (-3,2)}?(.85)*\dir{>} ?(.1)*\dir{>};
  (3,5)*{}="t1";  (9,5)*{}="t2";
  (3,-5)*{}="t1'";  (9,-5)*{}="t2'";
   "t1";"t2" **\crv{(4,8) & (9, 8)} ;
   "t1'";"t2'" **\crv{(4,-8) & (9, -8)};
   "t2'";"t2" **\crv{(10,0)} ?(.5)*\dir{<};
  (-1.5,-4)*{\bullet}; (3,5)*{\bullet}+(1,4)*{\scs n-1};
 \endxy
 \;\; + \quad  r_i \;\;
   \xy 0;/r.18pc/:
 (0,0)*{\cbub{n-1}};
  (4,8)*{n};
 (-10,10);(-10,-10); **\dir{-} ?(1)*\dir{>}+(2.3,0)*{\scriptstyle{}}
 \endxy
\end{equation}
where the first term on the right hand side is zero since it is the composite of a dot and an endomorphism of $\E \1_n$ of degree $-2$.  By Corollary \ref{cor:degz-bubbles} the remaining term on the right-hand side above is some non-zero multiple of the identity.
\end{proof}

Thus we get that \newline
\noindent\fbox {
   \parbox{\linewidth}{
\[
(\E \1_n)_R \cong \1_n \F  \la n+1 \ra \text{ and } (\E\1_n)_L \cong \1_n \F  \la -n-1 \ra,
\]
\[
 \left(\F\1_n\right)_L =  \1_n\E  \la n-1\ra\text{ and }  \left(\F\1_n\right)_R =  \1_n\E  \la -n+1\ra,
\]
    }
}

\medskip
\noindent which completes the induction step for the adjoint induction hypothesis \eqref{eq:ind_hyp}.

\begin{cor}\label{cor:EFidhoms}
Given a strong 2-representation of $\mf{sl}_2$ we have
$$\Hom(\E \F {\1}_{n+2} \la n+1 \ra, {\1}_{n+2}) \cong \k, \hspace{1.0cm}
\Hom({\1}_{n+2} \la n+1 \ra, \E \F {\1}_{n+2}) \cong \k,$$
$$\Hom(\F \E {\1}_n, {\1}_n \la n+1 \ra) \cong \k, \hspace{1.0cm}
\Hom({\1}_n, \F \E \1_n \la n+1 \ra) \cong \k,$$
showing that the maps in Proposition \ref{prop:lradj} are unique up to a non-zero multiple.
\end{cor}
\begin{proof}
We calculate the first space (the other three are similar). We have
\begin{eqnarray*}
\Hom(\E \F \1_{n+2} \la n+1 \ra, \1_{n+2})
&\cong& \Hom(\F \1_{n+2} \la n+1 \ra, (\E \1_{n})_R \1_{n+2}) \\
&\cong& \Hom(\F \1_{n+2}, \F \1_{n+2}) \\
&\cong& \k.
\end{eqnarray*}

\end{proof}

\begin{remark}\label{rem:nle0}
The proof of Proposition~\ref{prop:lradj} when $n \le 0$ is analogous to the one above. But since there is a break in symmetry in the definition of a $Q$-strong 2-representation of $\g$, namely we use $\E_R$ instead of $\E_L$, we spell out a few details about how to alter the proof above.

We fix $n \le 0$. The induction hypothesis (\ref{eq:ind_hyp}) is now for $m < n$. Lemma \ref{lem:1} is now for $m \le n$ and one proves it by moving the $\E$ by adjunction from the left side to the right side. Likewise, Corollary \ref{cor:0} is for $m \le n+2$.

Now, the left map in (\ref{eq:A}) is defined by adjunction. The maps in (\ref{eq:B}) are defined using the decomposition of $\F \E \1_n$. Finally, the right map in (\ref{eq:A}) is defined by the decomposition of $\E \F \1_{n+2}$ if $n=0,-1$ and, if $n \le -2$, by the composition
\[
\xy
 (-50,0)*+{\1_{n+2}}="1";
 (-25,0)*+{\F \E \1_{n+2} \la n+3 \ra}="2";
 (20,0)*+{\F \E \1_{n+2} \la -n-1 \ra}="3";
 (55,0)*+{\E \F \1_{n+2} \la -n-1 \ra.}="4";
 {\ar^-{\Ucupr} "1"; "2"};
 {\ar^-{\Udown\Uupdot {-n-2}} "2"; "3"};
 {\ar^-{\xy 0;/r.15pc/:
    (-4,-4)*{};(4,4)*{} **\crv{(-4,-1) & (4,1)}?(0)*\dir{<} ;
    (4,-4)*{};(-4,4)*{} **\crv{(4,-1) & (-4,1)}?(1)*\dir{>};
\endxy} "3"; "4"};
\endxy
\]
\end{remark}

\subsection{Up-down crossings}
At this point we no longer need to assume the adjoint induction hypothesis.  We have established for arbitrary weights $n$ that both the left and right adjoint of $\F$ is $\E$ (up to a specified shift). Consequently we can prove the following.

\begin{lem}\label{lem:homs}
For any $n \in \Z$ we have
$$\Hom(\E \F \1_n, \F \E \1_n) \cong \k \cong \Hom(\F \E \1_n, \E \F \1_n).$$
\end{lem}
\begin{proof}
We prove that $\Hom(\E \F \1_n, \F \E \1_n) \cong \k$ (the other case follows similarly). One has
\begin{eqnarray*}
\Hom(\E \F \1_n, \F \E \1_n)
&\cong& \Hom(\F (\E \1_n)_R \1_{n+2}, (\E \1_{n-2})_R \F \1_{n+2}) \\
&\cong& \Hom(\F \F \1_{n+2} \la n+1 \ra, \F \F \1_{n+2} \la n-1 \ra) \cong \k
\end{eqnarray*}
where the last isomorphism follows from (the adjoint of) Corollary \ref{cor:0}.
\end{proof}

Subsequently, we denote these maps $\F \E \1_n \rightarrow \E \F \1_n$ and $\E \F \1_n \rightarrow \F \E \1_n$ by $\;\xy 0;/r.14pc/:
    (-4,-4)*{};(4,4)*{} **\crv{(-4,-1) & (4,1)}?(0)*\dir{<} ;
    (4,-4)*{};(-4,4)*{} **\crv{(4,-1) & (-4,1)}?(1)*\dir{>};
  \endxy$ and $\xy 0;/r.14pc/:
    (-4,-4)*{};(4,4)*{} **\crv{(-4,-1) & (4,1)}?(1)*\dir{>} ;
    (4,-4)*{};(-4,4)*{} **\crv{(4,-1) & (-4,1)}?(0)*\dir{<};
\endxy\;$ respectively. For the moment these maps are uniquely defined only up to a non-zero scalar.

\begin{cor} \label{cor:1} If $n \ge 0$ then the map
\begin{equation}\label{eq:iso1}
\zeta\;\;:=\;\;
\xy 0;/r.15pc/:
    (-4,-4)*{};(4,4)*{} **\crv{(-4,-1) & (4,1)}?(0)*\dir{<} ;
    (4,-4)*{};(-4,4)*{} **\crv{(4,-1) & (-4,1)}?(1)*\dir{>};
  \endxy \;\; \bigoplus_{k=0}^{n-1}
\vcenter{\xy 0;/r.15pc/:
 (-4,2)*{}="t1"; (4,2)*{}="t2";
 "t2";"t1" **\crv{(4,-5) & (-4,-5)}; ?(1)*\dir{>}
 ?(.8)*\dir{}+(0,-.1)*{\bullet}+(-3,-1)*{\scs k};;
 \endxy}: \F \E \1_n \bigoplus_{k=0}^{n-1} \1_n \la n-1-2k \ra \rightarrow \E \F \1_n
\end{equation}
induces an isomorphism. Likewise, if $n \le 0$ then the map
\begin{equation}\label{eq:iso2}
\zeta\;\;:=\;\;
\xy 0;/r.15pc/:
    (-4,-4)*{};(4,4)*{} **\crv{(-4,-1) & (4,1)}?(1)*\dir{>} ;
    (4,-4)*{};(-4,4)*{} **\crv{(4,-1) & (-4,1)}?(0)*\dir{<};
  \endxy \;\; \bigoplus_{k=0}^{-n-1}
\vcenter{\xy 0;/r.15pc/:
 (4,2)*{}="t1"; (-4,2)*{}="t2";
 "t2";"t1" **\crv{(-4,-5) & (4,-5)}; ?(1)*\dir{>}
 ?(.8)*\dir{}+(0,-.1)*{\bullet}+(3,-1)*{\scs k};;
 \endxy}:  \E\F \1_n \bigoplus_{k=0}^{-n-1} \1_n \la -n-1-2k \ra \rightarrow \F \E \1_n
\end{equation}
is an isomorphism.
\end{cor}
\begin{proof}
We prove the case $n \ge 0$ (the case $n \le 0$ is proved similarly).

We know that
$$\E \F \1_n \cong \F \E \1_n \bigoplus_{k=0}^{n-1} \1_n \la n-1-2k \ra$$
and by Lemma \ref{lem:homs} the map $\xy 0;/r.14pc/:
    (-4,-4)*{};(4,4)*{} **\crv{(-4,-1) & (4,1)}?(0)*\dir{<} ;
    (4,-4)*{};(-4,4)*{} **\crv{(4,-1) & (-4,1)}?(1)*\dir{>};
  \endxy $  must induce an isomorphism between the $\F \E \1_n$ summands and must induce the zero map from the $\F \E \1_n$ summand on the left to any summand $\1_n \la n-1-2k \ra$ on the right.

It remains to show that $\bigoplus_{k=0}^{n-1} \vcenter{\xy 0;/r.15pc/:
 (-4,2)*{}="t1"; (4,2)*{}="t2";
 "t2";"t1" **\crv{(4,-5) & (-4,-5)}; ?(1)*\dir{>}
 ?(.8)*\dir{}+(0,-.1)*{\bullet}+(-3,-1)*{\scs k};;
 \endxy}$ induces an isomorphism between the summands $\1_n \la n-1-2k \ra$ on either side. Since $\Hom(\1_n, \1_n \la l \ra) = 0$ if $l < 0$ it follows that the induced map
$$\bigoplus_{k=0}^{n-1} \1_n \la n-1-2k \ra \rightarrow \bigoplus_{k=0}^{n-1} \1_n \la n-1-2k \ra$$
is upper triangular (when expressed as a matrix). It remains to show that the maps on the diagonal are isomorphisms between the summands $\1_n \la n-1-2k \ra$ on either side.

Now, by construction, the map $\;\vcenter{\xy 0;/r.14pc/:
 (-4,2)*{}="t1"; (4,2)*{}="t2"; "t2";"t1" **\crv{(4,-5) & (-4,-5)}; ?(1)*\dir{>};
 \endxy}: \1_n \la n - 1 \ra \rightarrow \E \F \1_n$ is an isomorphism onto the summand  $\1_n \la n - 1 \ra$ on the right side. Consequently, by Lemma \ref{lem:Xind}, the composition
$$\1_n \la n-1-2k \ra \xrightarrow{\vcenter{\xy 0;/r.15pc/:
 (-4,2)*{}="t1"; (4,2)*{}="t2";
 "t2";"t1" **\crv{(4,-5) & (-4,-5)}; ?(1)*\dir{>};
 \endxy}} \E \F \1_n \la -2k \ra \xrightarrow{k\Uupdot\Udown} \E \F \1_n $$
must also induce an isomorphism between the summands $\1_n \la n-1-2k \ra$ on either side. This proves that all the diagonal entries are isomorphisms so we are done.
\end{proof}

%
\subsection{Endomorphisms of $\E \1_n$}
%

\begin{lem}\label{lem:main}
Suppose $m < 2|n+2|$ (or $m=2$ and $n=-1$) and $f \in \Hom^m(\E \1_n, \E \1_n)$. If $n \ge -1$ then $f$ is of the form
\begin{equation}\label{eq:main1}
\sum_i \;\;
      \xy 0;/r.17pc/:
 (0,10);(0,-10); **\dir{-} ?(1)*\dir{>}+(2.3,0)*{\scriptstyle{}}
 ?(.1)*\dir{ }+(2,0)*{\scs };
 (0,0)*{\bullet}+(3,1)*{i};
 (-10,5)*{ n+2};
 (10,5)*{ n };
 (-16,-2)*{\chern{f_i}};
 (10,0)*{};(-2,-8)*{\scs };
 \endxy
\end{equation}
where $f_i \in \Hom^{m-2i}(\1_{n+2}, \1_{n+2})$. Similarly, if $n \le -1$ (or $m=2$ and $n=-1$) then $f$ is of the form
\begin{equation}\label{eq:main2}
\sum_i \;\;
      \xy 0;/r.17pc/:
 (0,10);(0,-10); **\dir{-} ?(1)*\dir{>}+(2.3,0)*{\scriptstyle{}}
 ?(.1)*\dir{ }+(2,0)*{\scs };
 (0,0)*{\bullet}+(3,1)*{i};
 (-10,5)*{ n+2};
 (10,5)*{ n };
 (16,-2)*{\chern{f_i}};
 (10,0)*{};(-2,-8)*{\scs };
 \endxy
\end{equation}
where $f_i \in \Hom^{m-2i}(\1_n, \1_n)$. By adjunction there are analogous results for $f \in \Hom^m(\F \1_n, \F \1_n)$.
\end{lem}
\begin{proof}
We prove the case $n \ge -1$ (the case $n \le -1$ is proved similarly). We will deal with the special case $m=2$ and $n=-1$ at the end. We have
\begin{eqnarray*}
\Hom^m(\E \1_n, \E \1_n)
&\cong& \Hom^m(\1_{n+2}, \E (\E \1_n)_L \1_{n+2}) \\
&\cong& \Hom^m(\1_{n+2}, \E \F \1_{n+2} \la -(n+1) \ra) \\
&\cong& \Hom^m(\1_{n+2}, \bigoplus_{[n+2]} \1_{n+2} \la -(n+1) \ra \oplus \F \E \1_{n+2} \la -(n+1) \ra) \\
&\cong& \Hom^{m-n-1}(\1_{n+2}, \bigoplus_{[n+2]} \1_{n+2}) \oplus \Hom^m((\F \1_{n+4})_L, \E \1_{n+2} \la -(n+1) \ra) \\
&\cong& \Hom^{m-n-1}(\1_{n+2}, \bigoplus_{[n+2]} \1_{n+2}) \oplus \Hom^m(\E \1_{n+2} \la n+3 \ra, \E \1_{n+2} \la -(n+1) \ra) \\
&\cong& \Hom^{m-n-1}(\1_{n+2}, \bigoplus_{[n+2]} \1_{n+2})
\end{eqnarray*}
where the last line follows since $m<2(n+2)$ meaning $\Hom^m(\E \1_{n+2}, \E \1_{n+2} \la -2(n+2) \ra) = 0.$ Keeping track of degrees, we find that
\begin{equation}\label{eq:2}
\Hom^m(\E \1_n, \E \1_n) \cong \bigoplus_{k \ge 0} \Hom(\1_{n+2}, \1_{n+2} \la m-2k \ra).
\end{equation}
If $f \in \Hom^m(\E \1_n, \E \1_n)$ then we denote the map induced by adjunction $f' \in \Hom^m(\1_{n+2}, \E \F \1_{n+2} \la -(n+1) \ra)$ and the induced maps on the right side of (\ref{eq:2}) by $f_k \in \Hom(\1_{n+2}, \1_{n+2} \la m-2k \ra)$.

Now let us trace through the series of isomorphisms above in order to explicitly identify $f'$ with $f_k$. The critical isomorphism is the third one where one uses the isomorphism
$$\F \E \1_{n+2} \bigoplus_{k=0}^{n+1} \1_{n+2} \la -n-1+2k \ra \rightarrow \E \F \1_{n+2}$$
from Corollary \ref{cor:1}. Thus we find that $f'$ corresponds to the sum over $k$ of the composition
$${\1}_{n+2} \xrightarrow{f_k} {\1}_{n+2} \la m-2k \ra \xrightarrow{\vcenter{\xy 0;/r.15pc/:
 (-4,2)*{}="t1"; (4,2)*{}="t2";
 "t2";"t1" **\crv{(4,-5) & (-4,-5)}; ?(1)*\dir{>};
 \endxy}} \E\F {\1}_{n+2} \la m - (n+1) - 2k \ra \xrightarrow{k\Uupdot\Udown} \E \F {\1}_{n+2} \la m-(n+1) \ra.$$
Consequently, using the adjunction which relates $f$ and $f'$ we find that $f$ is given by the composition
$${\1}_{n+2} \E \xrightarrow{f_k\Uup} {\1}_{n+2} \E \la m-2k \ra \xrightarrow{k\Uupdot} {\1}_{n+2} \E \la m \ra.$$
This implies what we needed to prove.

If $m=2$ and $n=-1$ then the long calculation above yields
\begin{equation}\label{eq:new}
\Hom^2(\E \1_{-1}, \E \1_{-1}) \cong \Hom^2(\1_1, \1_1) \oplus \Hom(\E \1_1, \E \1_1).
\end{equation}
Tracing back through the isomorphisms the first summand above consists of maps of the form (\ref{eq:main1}) with $i=0$ while the second summand is one-dimensional. Since a dot is not of the form (\ref{eq:main1}) with $i=0$ it must, together with those maps of the form (\ref{eq:main1}), span the space $\Hom^2(\E \1_{-1}, \E \1_{-1})$.
\begin{remark}
Note that if one were to trace through the isomorphisms above, the right most space in (\ref{eq:new}) corresponds to a left curl. The argument above implies that this left curl must be some linear combination of a dot and a map of the form (\ref{eq:main1}) with $i=0$.
\end{remark}

\end{proof}

\section{Cyclic biadjointness}

In this section we fix all adjunction maps and prove cyclic biadjointness when $\g = {\mf{sl}}_2$. This means that the two possible ways for defining dots on downward strands, sideways crossings or downward crossings give the same 2-morphisms.

%
\subsection{Fixing adjunction maps} \label{subsec:adjoint}
%

Recall that at this point, adjunction maps are only determined up to a scalar. There are two constraints on the choices of scalar we can make: the adjunction relations \eqref{eq_biadjoint1} and \eqref{eq_biadjoint2}, and the normalization of degree zero bubbles.
This system of constraints is over determined by exactly one relation: the normalization of the degree zero bubble in weights $n=+1$ and $n=-1$ are not independent.

In this section we rescale these caps and cups to satisfy the adjoint constraints \eqref{eq_biadjoint1} and \eqref{eq_biadjoint2}  and so that
\begin{equation} \label{eq:4.1}
\xy 0;/r.18pc/:
 (0,0)*{\cbub{n-1}};
  (4,8)*{n};
 \endxy
  = \Id_{\1_n} \quad \text{for $n \geq 1$,}
  \qquad \quad
  \xy 0;/r.18pc/:
 (0,0)*{\ccbub{-n-1}};
  (4,8)*{n};
 \endxy
  =  \Id_{\1_n} \quad \text{for $n < -1$.}
\end{equation}
Recall that negative degree dotted bubbles are zero while a dotted bubble of degree zero must be a non-zero multiple of the identity map by Corollary~\ref{cor:degz-bubbles} (note that for $n=0$ there are no degree zero dotted bubbles). By rescaling the adjunction maps we can ensure the degree zero bubbles satisfy the conditions above. More precisely, we rescale in the following order.
\[
\begin{tabular}{|l|c|c|c|c|}
 \hline
 $ m \geq 0$ &
 \text{ $\xy
    (10,-5)*{m-1};
    (0,-2)*{m+1};
     (-6,0)*{}; (6,0)*{} **\crv{(-6,-8) & (6,-8)} ?(1)*\dir{>};
    \endxy$}
  & \text{$\xy
    (10,0)*{m+1};
    (0,-3)*{m-1};
    (-6,-5)*{};(6,-5)*{} **\crv{(-6,3) & (6,3)} ?(1)*\dir{>};
    \endxy$}
  & \text{$\xy
    (10,-5)*{m+1};
    (0,-2)*{m-1};
    (-6,0)*{};(6,0)*{} **\crv{(-6,-8) & (6,-8)} ?(0)*\dir{<};
    \endxy$}
  & \text{$\xy
    (10,0)*{m-1};
    (0,-3)*{m+1};
    (-6,-5)*{};(6,-5)*{} **\crv{(-6,3) & (6,3)} ?(0)*\dir{<};
    \endxy$}
    \\& & &  &\\ \hline
 & \;\;\txt{ fixed arbitrarily} \;\;
 & \;\;\txt{ determined by \\adjunction}\;\;
 & \;\;\txt{ fixed by value\\ of bubble}\;\;
 & \;\;\txt{ determined by\\ adjunction}\;\;
 \\
 \hline
\end{tabular}
\]
\[
\begin{tabular}{|l|c|c|c|c|}
 \hline
 $ m < 0$ &  \text{$\xy
    (10,-5)*{m+1};
    (0,-2)*{m-1};    (-6,0)*{};(6,0)*{} **\crv{(-6,-8) & (6,-8)} ?(0)*\dir{<};
    \endxy$}
  & \text{$\xy
    (10,0)*{m-1};
    (0,-3)*{m+1};
    (-6,-5)*{};(6,-5)*{} **\crv{(-6,3) & (6,3)} ?(0)*\dir{<};
    \endxy$}
  & \text{$\xy
    (10,-5)*{m-1};
    (0,-2)*{m+1};
    (-6,0)*{};(6,0)*{} **\crv{(-6,-8) & (6,-8)} ?(1)*\dir{>};
    \endxy$}
  & \text{$\xy
    (10,0)*{m+1};
    (0,-3)*{m-1};
    (-6,-5)*{};(6,-5)*{} **\crv{(-6,3) & (6,3)} ?(1)*\dir{>};
    \endxy$}
    \\& & &  &\\ \hline
 & \;\;\txt{ fixed arbitrarily}\;\;
 & \;\;\txt{ determined by \\adjunction}\;\;
 & \;\;\txt{ fixed by value\\ of bubble}\;\;
 & \;\;\txt{ determined by\\ adjunction}\;\;
 \\
 \hline
\end{tabular}
\]
After fixing cups and caps as above, all bubbles of degree zero are equal to the identity except for the counter-clockwise bubble in weight $n=-1$. This is why the case $n=-1$ was missing from \eqref{eq:4.1}.
This bubble is multiplication by some arbitrary scalar $c_{-1}$:
\begin{equation} \label{eq_defcmone}
    \xy 0;/r.17pc/:
 (0,0)*{\nccbub};
  (4,8)*{-1};
 \endxy
  =  c_{-1} \Id_{\1_n}
\end{equation}
where $-1$ denotes the outside region of the bubble. We will show in Lemma~\ref{lem_coeff} that in fact $c_{-1} = 1$.

%
\subsection{Cyclic biadjointness}\label{sec:proofcycbiadjoint}
%

Recall that the action of dots and crossings is originally defined only for upward pointing strands. To define dots on downward strands we use the (bi)adjunction between $\E$ and $\F$.  We define the downward dot and downward crossing by taking the `right-mates' under adjunction to their upward analogues.
\begin{equation} \label{eq_down_def}
      \xy 0;/r.17pc/:
 (0,10);(0,-10); **\dir{-} ?(.75)*\dir{<}+(2.3,0)*{\scriptstyle{}}
 ?(.1)*\dir{ }+(2,0)*{\scs };
 (0,0)*{\txt\large{$\bullet$}};
 (-6,5)*{ n};
 (8,5)*{ n +2};
 (-10,0)*{};(10,0)*{};(-2,-8)*{\scs };
 \endxy
    \;\; := \;\;
   \xy 0;/r.17pc/:
    (8,5)*{}="1";
    (0,5)*{}="2";
    (0,-5)*{}="2'";
    (-8,-5)*{}="3";
    (8,-10);"1" **\dir{-};
    "2";"2'" **\dir{-} ?(.5)*\dir{<};
    "1";"2" **\crv{(8,12) & (0,12)} ?(0)*\dir{<};
    "2'";"3" **\crv{(0,-12) & (-8,-12)}?(1)*\dir{<};
    "3"; (-8,10) **\dir{-};
    (15,9)*{n+2};
    (-12,-9)*{n};
    (0,4)*{\txt\large{$\bullet$}};
    (-10,8)*{\scs };
    (10,-8)*{\scs };
    \endxy
\qquad \qquad
\xy
  (0,0)*{\xybox{
    (-4,-4)*{};(4,4)*{} **\crv{(-4,-1) & (4,1)}?(0)*\dir{<} ;
    (4,-4)*{};(-4,4)*{} **\crv{(4,-1) & (-4,1)}?(0)*\dir{<};
     (-8,0)*{ n};
     (-12,0)*{};(12,0)*{};
     }};
  \endxy \;\; :=  \;\;
 \xy 0;/r.17pc/:
  (0,0)*{\xybox{
    (4,-4)*{};(-4,4)*{} **\crv{(4,-1) & (-4,1)}?(1)*\dir{>};
    (-4,-4)*{};(4,4)*{} **\crv{(-4,-1) & (4,1)};
     (-4,4)*{};(18,4)*{} **\crv{(-4,16) & (18,16)} ?(1)*\dir{>};
     (4,-4)*{};(-18,-4)*{} **\crv{(4,-16) & (-18,-16)} ?(1)*\dir{<}?(0)*\dir{<};
     (-18,-4);(-18,12) **\dir{-};(-12,-4);(-12,12) **\dir{-};
     (18,4);(18,-12) **\dir{-};(12,4);(12,-12) **\dir{-};
     (8,1)*{ n};
     (-10,0)*{};(10,0)*{};
     (-4,-4)*{};(-12,-4)*{} **\crv{(-4,-10) & (-12,-10)}?(1)*\dir{<}?(0)*\dir{<};
      (4,4)*{};(12,4)*{} **\crv{(4,10) & (12,10)}?(1)*\dir{>}?(0)*\dir{>};
     }};
  \endxy
\end{equation}
Because $\E$ and $\F$ are biadjoint up to shift, we could have defined the downward maps using `left-mates' under adjunction. The next two results (Lemmas \ref{lem:A} and \ref{lem:B}) show that both ways give the same map (we refer to this as the ``cyclic biadjointness property'').

\begin{lem}\label{lem:A}
Given a strong 2-representation of $\mf{sl}_2$ we have:
\begin{equation}\label{eq:cycbiadjoint-dot}
   \xy 0;/r.17pc/:
    (-8,5)*{}="1";
    (0,5)*{}="2";
    (0,-5)*{}="2'";
    (8,-5)*{}="3";
    (-8,-10);"1" **\dir{-};
    "2";"2'" **\dir{-} ?(.5)*\dir{<};
    "1";"2" **\crv{(-8,12) & (0,12)} ?(0)*\dir{<};
    "2'";"3" **\crv{(0,-12) & (8,-12)}?(1)*\dir{<};
    "3"; (8,10) **\dir{-};
    (15,-9)*{ n+2};
    (-12,9)*{n};
    (0,4)*{\bullet};
    (10,8)*{\scs };
    (-10,-8)*{\scs };
    \endxy  \quad = \quad
   \xy 0;/r.17pc/:
    (8,5)*{}="1";
    (0,5)*{}="2";
    (0,-5)*{}="2'";
    (-8,-5)*{}="3";
    (8,-10);"1" **\dir{-};
    "2";"2'" **\dir{-} ?(.5)*\dir{<};
    "1";"2" **\crv{(8,12) & (0,12)} ?(0)*\dir{<};
    "2'";"3" **\crv{(0,-12) & (-8,-12)}?(1)*\dir{<};
    "3"; (-8,10) **\dir{-};
    (15,9)*{n+2};
    (-12,-9)*{n};
    (0,4)*{\txt\large{$\bullet$}};
    (-10,8)*{\scs };
    (10,-8)*{\scs };
    \endxy
\end{equation}
\end{lem}

\begin{proof}
The analogue of Lemma~\ref{lem:main} for downward pointing strands (obtained by taking adjoints) implies that for $n \geq -1$ we have
\begin{equation} \label{eq_cyclem_dot}
 \xy 0;/r.17pc/:
    (-8,5)*{}="1";
    (0,5)*{}="2";
    (0,-5)*{}="2'";
    (8,-5)*{}="3";
    (-8,-10);"1" **\dir{-};
    "2";"2'" **\dir{-} ?(.5)*\dir{<};
    "1";"2" **\crv{(-8,12) & (0,12)} ?(0)*\dir{<};
    "2'";"3" **\crv{(0,-12) & (8,-12)}?(1)*\dir{<};
    "3"; (8,10) **\dir{-};
    (15,-9)*{ n+2};
    (-12,9)*{n};
    (0,4)*{\bullet};
    (10,8)*{\scs };
    (-10,-8)*{\scs };
    \endxy
    \quad = \quad \gamma_0
   \xy 0;/r.17pc/:
    (8,5)*{}="1";
    (0,5)*{}="2";
    (0,-5)*{}="2'";
    (-8,-5)*{}="3";
    (8,-10);"1" **\dir{-};
    "2";"2'" **\dir{-} ?(.5)*\dir{<};
    "1";"2" **\crv{(8,12) & (0,12)} ?(0)*\dir{<};
    "2'";"3" **\crv{(0,-12) & (-8,-12)}?(1)*\dir{<};
    "3"; (-8,10) **\dir{-};
    (15,9)*{n+2};
    (-12,-9)*{n};
    (0,4)*{\txt\large{$\bullet$}};
    (-10,8)*{\scs };
    (10,-8)*{\scs };
    \endxy
 \;\; + \;\;
       \xy 0;/r.17pc/:
 (0,10);(0,-10); **\dir{-} ?(.75)*\dir{<}+(2.3,0)*{\scriptstyle{}}
 ?(.1)*\dir{ }+(2,0)*{\scs };
  (16,-2)*{\chern{\gamma_1}};
 (-6,6)*{ n};
 (10,6)*{ n +2};
 (-10,0)*{};(10,0)*{};(-2,-8)*{\scs };
 \endxy
 \end{equation}
for $\gamma_0 \in \End(\1_{n+2})=\Bbbk$ and $\gamma_1 \in \End^2(\1_{n+2})$.
Closing off this relation on the left with $m$ dots gives rise to the equality
\begin{equation}
 \xy 0;/r.17pc/:
    (-8,2)*{}="1";
    (0,2)*{}="2";
    (0,-2)*{}="2'";
    (8,-2)*{}="3";
    (-8,-8);"1" **\dir{-};
    "2";"2'" **\dir{-} ?(.8)*\dir{<};
    "1";"2" **\crv{(-8,9) & (0,9)} ?(0)*\dir{<};
    "2'";"3" **\crv{(0,-9) & (8,-9)}?(1)*\dir{<};
    "3"; (8,8) **\dir{-};
    (8,8)*{}; (-16,8)*{} **\crv{(8,20) & (-16,20)};
    (-16,8)*{}; (-16,-8) **\dir{-} ?(.5)*\dir{<};
    (-16,-8)*{}; (-8,-8)*{} **\crv{(-16,-15) & (-8,-15)};
    (15,-9)*{ n+2};
    (0,1)*{\bullet}; (-16,6)*{\bullet}+(-3,1)*{\scs m};
    (10,8)*{\scs };
    (-10,-8)*{\scs };
    \endxy
 \;\; = \;\;
\gamma_0 \;\;
 \xy 0;/r.17pc/:
    (-8,-2)*{}="1";
    (0,-2)*{}="2";
    (0,2)*{}="2'";
    (8,2)*{}="3";
    (-8,8);"1" **\dir{-};
    "2";"2'" **\dir{-} ?(1)*\dir{>};
    "1";"2" **\crv{(-8,-9) & (0,-9)} ?(0)*\dir{>};
    "2'";"3" **\crv{(0,9) & (8,9)}?(1)*\dir{>};
    "3"; (8,-8) **\dir{-};
    (8,-8)*{}; (-16,-8)*{} **\crv{(8,-20) & (-16,-20)};
    (-16,-8)*{}; (-16,8) **\dir{-} ?(.5)*\dir{>};
    (-16,8)*{}; (-8,8)*{} **\crv{(-16,15) & (-8,15)};
    (15,9)*{ n+2};
    (0,-1)*{\bullet}; (-16,6)*{\bullet}+(-3,1)*{\scs m};
    (10,8)*{\scs };
    (-10,-8)*{\scs };
    \endxy
    \;\; + \;\;
  \xy 0;/r.17pc/:
    (-8,-8)*{}="1";
    (-8,8);"1" **\dir{-};
    (-16,-8)*{}; (-16,8) **\dir{-} ?(.5)*\dir{>};
    (-16,8)*{}; (-8,8)*{} **\crv{(-16,15) & (-8,15)};
    "1"; (-16,-6)*{} **\crv{(-8,-15) & (-16,-15)};
    (7,9)*{ n+2};
   (-16,6)*{\bullet}+(-3,1)*{\scs m};
  (5,-2)*{\chern{\gamma_1}};
    \endxy
\end{equation}
which simplifies to
\begin{equation} \label{eq_cyclem_dot2}
 (1-\gamma_0) \;
 \xy 0;/r.17pc/:
(0,0)*{\cbub{m+1}};
(7,9)*{ n+2};
\endxy
\;\; = \;\;
 \xy 0;/r.17pc/:
(0,0)*{\cbub{m}};
 (11,0)*{\chern{\gamma_1}};
 (7,9)*{ n+2};
\endxy
\end{equation}
For $n \geq 0$ take $m=n$ so that the bubble on the left is equal to 1 and the negative degree bubble on the right-hand side is zero. Hence $\gamma_0=1$.  Taking $m=n+1$ then implies that $\gamma_1=0$. A similar argument proves the lemma for $n<-1$.

The case $n=-1$ requires special attention as we can not set $m=n$ in \eqref{eq_cyclem_dot2}. Note that for $n=-1$ the diagram
\begin{equation}
  \xy 0;/r.17pc/:
  (17,8)*{n+2};
  (-3,-10)*{};(3,5)*{} **\crv{(-3,-2) & (2,1)}?(1)*\dir{>};?(.15)*\dir{>};
    (3,-5)*{};(-3,10)*{} **\crv{(2,-1) & (-3,2)}?(.85)*\dir{>} ?(.1)*\dir{>};
  (3,5)*{}="t1";  (9,5)*{}="t2";
  (3,-5)*{}="t1'";  (9,-5)*{}="t2'";
   "t1";"t2" **\crv{(4,8) & (9, 8)};
   "t1'";"t2'" **\crv{(4,-8) & (9, -8)};
   "t2'";"t2" **\crv{(10,0)} ;
 \endxy\;\;
\end{equation}
is zero since it has degree $-2$. Using the NilHecke relation we have
\begin{equation}
  \xy 0;/r.17pc/:
  (16,8)*{n+2};
  (-3,-10)*{};(3,5)*{} **\crv{(-3,-2) & (2,1)}?(1)*\dir{>};?(.15)*\dir{>};
    (3,-5)*{};(-3,10)*{} **\crv{(2,-1) & (-3,2)}?(.85)*\dir{>} ;
  (3,5)*{}="t1";  (9,5)*{}="t2";
  (3,-5)*{}="t1'";  (9,-5)*{}="t2'";
   "t1";"t2" **\crv{(4,8) & (9, 8)};
   "t1'";"t2'" **\crv{(4,-8) & (9, -8)};
   "t2'";"t2" **\crv{(10,0)} ;
  (2.5,-4)*{\bullet};
 \endxy\;\; = \;\;
  \xy 0;/r.17pc/:
  (16,8)*{n+2};
  (-3,-10)*{};(3,5)*{} **\crv{(-3,-2) & (2,1)}?(1)*\dir{>};?(.15)*\dir{>};
    (3,-5)*{};(-3,10)*{} **\crv{(2,-1) & (-3,2)}?(.85)*\dir{>} ?(.1)*\dir{>};
  (3,5)*{}="t1";  (9,5)*{}="t2";
  (3,-5)*{}="t1'";  (9,-5)*{}="t2'";
   "t1";"t2" **\crv{(4,8) & (9, 8)};
   "t1'";"t2'" **\crv{(4,-8) & (9, -8)};
   "t2'";"t2" **\crv{(10,0)} ;
  (-2,4)*{\bullet};
 \endxy
  - \;\; r_{i} \;\;
     \xy 0;/r.17pc/:
 (0,10);(0,-10); **\dir{-} ?(.75)*\dir{>};
 (8,-2)*{\ncbub{}{}};
 (10,6)*{ n+2};
 \endxy
 \;\; \refequal{\eqref{eq:4.1}}\;\;
 - r_{i} \;\;
     \xy 0;/r.17pc/:
 (0,10);(0,-10); **\dir{-} ?(.75)*\dir{>};
 (7,6)*{ n+2};
 \endxy  \;\; = \;\; \xy 0;/r.17pc/:
  (16,8)*{n+2};
  (-3,-10)*{};(3,5)*{} **\crv{(-3,-2) & (2,1)}?(1)*\dir{>};?(.15)*\dir{>};
    (3,-5)*{};(-3,10)*{} **\crv{(2,-1) & (-3,2)}?(.85)*\dir{>} ?(.1)*\dir{>};
  (3,5)*{}="t1";  (9,5)*{}="t2";
  (3,-5)*{}="t1'";  (9,-5)*{}="t2'";
   "t1";"t2" **\crv{(4,8) & (9, 8)};
   "t1'";"t2'" **\crv{(4,-8) & (9, -8)};
   "t2'";"t2" **\crv{(10,0)} ;
  (-2,-4)*{\bullet};
 \endxy
  - \;\; r_{i} \;\;
     \xy 0;/r.17pc/:
 (0,10);(0,-10); **\dir{-} ?(.75)*\dir{>};
 (8,-2)*{\ncbub{}{}};
 (10,6)*{ n+2};
 \endxy
\;\; = \;\;
  \xy 0;/r.17pc/:
  (16,8)*{n+2};
  (-3,-10)*{};(3,5)*{} **\crv{(-3,-2) & (2,1)}?(.15)*\dir{>};
    (3,-5)*{};(-3,10)*{} **\crv{(2,-1) & (-3,2)}?(.85)*\dir{>} ?(.1)*\dir{>};
  (3,5)*{}="t1";  (9,5)*{}="t2";
  (3,-5)*{}="t1'";  (9,-5)*{}="t2'";
   "t1";"t2" **\crv{(4,8) & (9, 8)};
   "t1'";"t2'" **\crv{(4,-8) & (9, -8)};
   "t2'";"t2" **\crv{(10,0)} ;
  (2.5,4)*{\bullet};
 \endxy
\end{equation}
so that these dotted curls are non-zero. Hence, gluing equation \eqref{eq_cyclem_dot} into the diagram
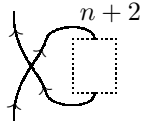
\[
  \xy 0;/r.17pc/:
  (15,10)*{n+2};
  (-3,-10)*{};(3,5)*{} **\crv{(-3,-2) & (2,1)}?(.15)*\dir{>} ?(.85)*\dir{>};
    (3,-5)*{};(-3,10)*{} **\crv{(2,-1) & (-3,2)}?(.85)*\dir{>} ?(.1)*\dir{>};
  (3,5)*{}="t1";  (12,5)*{}="t2";
  (3,-5)*{}="t1'";  (12,-5)*{}="t2'";
   "t1";"t2" **\crv{(4,8) & (12, 8)};
   "t1'";"t2'" **\crv{(4,-8) & (12, -8)};
  (8,5)*{}="tl";
  (16,5)*{}="tr";
  (8,-5)*{}="bl";
  (16,-5)*{}="br";
  "tl";"tr" **\dir{.};
  "tl";"bl" **\dir{.};
  "bl";"br" **\dir{.};
  "tr";"br" **\dir{.};
 \endxy
\]
and simplifying using the adjoint structure implies that $\gamma_0=1$. Then equation \eqref{eq_cyclem_dot2} with $m=0$  implies $\gamma_1=0$. Note that this curl proof can actually be modified to give an alternative proof of this lemma for all $n \geq -1$.
\end{proof}

\begin{lem}\label{lem:B}
In a strong 2-representation of $\mf{sl}_2$ we have:
\begin{equation}
\xy 0;/r.17pc/:
  (0,0)*{\xybox{
    (-4,-4)*{};(4,4)*{} **\crv{(-4,-1) & (4,1)}?(1)*\dir{>};
    (4,-4)*{};(-4,4)*{} **\crv{(4,-1) & (-4,1)};
     (4,4)*{};(-18,4)*{} **\crv{(4,16) & (-18,16)} ?(1)*\dir{>};
     (-4,-4)*{};(18,-4)*{} **\crv{(-4,-16) & (18,-16)} ?(1)*\dir{<}?(0)*\dir{<};
     (18,-4);(18,12) **\dir{-};(12,-4);(12,12) **\dir{-};
     (-18,4);(-18,-12) **\dir{-};(-12,4);(-12,-12) **\dir{-};
     (8,1)*{ n};
     (-10,0)*{};(10,0)*{};
      (4,-4)*{};(12,-4)*{} **\crv{(4,-10) & (12,-10)}?(1)*\dir{<}?(0)*\dir{<};
      (-4,4)*{};(-12,4)*{} **\crv{(-4,10) & (-12,10)}?(1)*\dir{>}?(0)*\dir{>};
     }};
  \endxy
\quad =  \quad
 \xy 0;/r.17pc/:
  (0,0)*{\xybox{
    (4,-4)*{};(-4,4)*{} **\crv{(4,-1) & (-4,1)}?(1)*\dir{>};
    (-4,-4)*{};(4,4)*{} **\crv{(-4,-1) & (4,1)};
     (-4,4)*{};(18,4)*{} **\crv{(-4,16) & (18,16)} ?(1)*\dir{>};
     (4,-4)*{};(-18,-4)*{} **\crv{(4,-16) & (-18,-16)} ?(1)*\dir{<}?(0)*\dir{<};
     (-18,-4);(-18,12) **\dir{-};(-12,-4);(-12,12) **\dir{-};
     (18,4);(18,-12) **\dir{-};(12,4);(12,-12) **\dir{-};
     (8,1)*{ n};
     (-10,0)*{};(10,0)*{};
     (-4,-4)*{};(-12,-4)*{} **\crv{(-4,-10) & (-12,-10)}?(1)*\dir{<}?(0)*\dir{<};
      (4,4)*{};(12,4)*{} **\crv{(4,10) & (12,10)}?(1)*\dir{>}?(0)*\dir{>};
     }};
  \endxy
\end{equation}
\end{lem}

\begin{proof}
 We must have
\begin{equation} \label{eq:cyclic-kappa}
\xy 0;/r.17pc/:
    (-4,-4)*{};(4,4)*{} **\crv{(-4,-1) & (4,1)}?(1)*\dir{>};
    (4,-4)*{};(-4,4)*{} **\crv{(4,-1) & (-4,1)};
     (4,4)*{};(-18,4)*{} **\crv{(4,16) & (-18,16)} ?(1)*\dir{>};
     (-4,-4)*{};(18,-4)*{} **\crv{(-4,-16) & (18,-16)} ?(1)*\dir{<}?(0)*\dir{<};
     (18,-4);(18,12) **\dir{-};(12,-4);(12,12) **\dir{-};
     (-18,4);(-18,-12) **\dir{-};(-12,4);(-12,-12) **\dir{-};
     (8,1)*{ n};
     (4,-4)*{};(12,-4)*{} **\crv{(4,-10) & (12,-10)}?(1)*\dir{<}?(0)*\dir{<};
      (-4,4)*{};(-12,4)*{} **\crv{(-4,10) & (-12,10)}?(1)*\dir{>}?(0)*\dir{>};
  \endxy
\quad =  \quad \kappa\;\;\;
 \xy 0;/r.17pc/:
    (4,-4)*{};(-4,4)*{} **\crv{(4,-1) & (-4,1)}?(1)*\dir{>};
    (-4,-4)*{};(4,4)*{} **\crv{(-4,-1) & (4,1)};
     (-4,4)*{};(18,4)*{} **\crv{(-4,16) & (18,16)} ?(1)*\dir{>};
     (4,-4)*{};(-18,-4)*{} **\crv{(4,-16) & (-18,-16)} ?(1)*\dir{<}?(0)*\dir{<};
     (-18,-4);(-18,12) **\dir{-};(-12,-4);(-12,12) **\dir{-};
     (18,4);(18,-12) **\dir{-};(12,4);(12,-12) **\dir{-};
     (8,1)*{ n};
     (-4,-4)*{};(-12,-4)*{} **\crv{(-4,-10) & (-12,-10)}?(1)*\dir{<}?(0)*\dir{<};
      (4,4)*{};(12,4)*{} **\crv{(4,10) & (12,10)}?(1)*\dir{>}?(0)*\dir{>};
  \endxy
\end{equation}
for some scalar $\kappa$, using Corollary~\ref{cor:0}.  Using cyclicity for dots proven above, this equation implies
\begin{equation}
\xy 0;/r.17pc/:
    (-4,-4)*{};(4,4)*{} **\crv{(-4,-1) & (4,1)}?(1)*\dir{>};
    (4,-4)*{};(-4,4)*{} **\crv{(4,-1) & (-4,1)}?(1)*\dir{}+(0,0)*{\bullet};
     (4,4)*{};(-18,4)*{} **\crv{(4,16) & (-18,16)};
     (-4,-4)*{};(18,-4)*{} **\crv{(-4,-16) & (18,-16)} ?(1)*\dir{<}?(0)*\dir{<};
     (18,-4);(18,12) **\dir{-};(12,-4);(12,12) **\dir{-};
     (-18,4);(-18,-12) **\dir{-}?(0)*\dir{>};(-12,4);(-12,-12) **\dir{-};
     (8,1)*{ n};
      (4,-4)*{};(12,-4)*{} **\crv{(4,-10) & (12,-10)}?(1)*\dir{<}?(0)*\dir{<};
      (-4,4)*{};(-12,4)*{} **\crv{(-4,10) & (-12,10)}?(1)*\dir{>};
  \endxy
\quad =  \quad \kappa\;\;\;
 \xy 0;/r.17pc/:
    (4,-4)*{};(-4,4)*{} **\crv{(4,-1) & (-4,1)}?(1)*\dir{}+(0,0)*{\bullet};
    (-4,-4)*{};(4,4)*{} **\crv{(-4,-1) & (4,1)};
     (-4,4)*{};(18,4)*{} **\crv{(-4,16) & (18,16)} ?(1)*\dir{>};
     (4,-4)*{};(-18,-4)*{} **\crv{(4,-16) & (-18,-16)} ?(1)*\dir{<}?(0)*\dir{<};
     (-18,-4);(-18,12) **\dir{-};(-12,-4);(-12,12) **\dir{-};
     (18,4);(18,-12) **\dir{-};(12,4);(12,-12) **\dir{-};
     (8,1)*{ n};
     (-4,-4)*{};(-12,-4)*{} **\crv{(-4,-10) & (-12,-10)}?(1)*\dir{<}?(0)*\dir{<};
      (4,4)*{};(12,4)*{} **\crv{(4,10) & (12,10)}?(1)*\dir{>}?(0)*\dir{>};
  \endxy
\end{equation}

Then by the NilHecke relation we have
\begin{equation} \label{eq_kappa_dot}
\xy 0;/r.17pc/:
    (-4,-4)*{};(4,4)*{} **\crv{(-4,-1) & (4,1)}?(1)*\dir{>};
    (4,-4)*{};(-4,4)*{} **\crv{(4,-1) & (-4,1)}?(1)*\dir{}+(0,0)*{\bullet};
     (4,4)*{};(-18,4)*{} **\crv{(4,16) & (-18,16)};
     (-4,-4)*{};(18,-4)*{} **\crv{(-4,-16) & (18,-16)} ?(1)*\dir{<}?(0)*\dir{<};
     (18,-4);(18,12) **\dir{-};(12,-4);(12,12) **\dir{-};
     (-18,4);(-18,-12) **\dir{-}?(0)*\dir{>};(-12,4);(-12,-12) **\dir{-};
     (8,1)*{ n};
      (4,-4)*{};(12,-4)*{} **\crv{(4,-10) & (12,-10)}?(1)*\dir{<}?(0)*\dir{<};
      (-4,4)*{};(-12,4)*{} **\crv{(-4,10) & (-12,10)}?(1)*\dir{>};
  \endxy
\quad =  \quad
\xy 0;/r.17pc/:
    (-4,-4)*{};(4,4)*{} **\crv{(-4,-1) & (4,1)}?(1)*\dir{>};
    (4,-4)*{};(-4,4)*{} **\crv{(4,-1) & (-4,1)} ?(0)*\dir{}+(0,0)*{\bullet};
     (4,4)*{};(-18,4)*{} **\crv{(4,16) & (-18,16)} ?(1)*\dir{>};
     (-4,-4)*{};(18,-4)*{} **\crv{(-4,-16) & (18,-16)} ?(1)*\dir{<}?(0)*\dir{<};
     (18,-4);(18,12) **\dir{-};(12,-4);(12,12) **\dir{-};
     (-18,4);(-18,-12) **\dir{-};(-12,4);(-12,-12) **\dir{-};
     (8,1)*{ n};
      (4,-4)*{};(12,-4)*{} **\crv{(4,-10) & (12,-10)}?(1)*\dir{<};
      (-4,4)*{};(-12,4)*{} **\crv{(-4,10) & (-12,10)}?(1)*\dir{>}?(0)*\dir{>};
  \endxy
  \;\; + \;\; r_{i} \;\;
 \xy 0;/r.17pc/:
    (-4,-4)*{};(-4,4)*{} **\dir{-} ?(.5)*\dir{>};
    (4,-4)*{};(4,4)*{} **\dir{-} ?(.5)*\dir{>};
     (4,4)*{};(-18,4)*{} **\crv{(4,16) & (-18,16)} ?(1)*\dir{>};
     (-4,-4)*{};(18,-4)*{} **\crv{(-4,-16) & (18,-16)} ?(1)*\dir{<};
     (18,-4);(18,12) **\dir{-};(12,-4);(12,12) **\dir{-};
     (-18,4);(-18,-12) **\dir{-};(-12,4);(-12,-12) **\dir{-};
     (8,1)*{ n};
      (4,-4)*{};(12,-4)*{} **\crv{(4,-10) & (12,-10)}?(1)*\dir{<};
      (-4,4)*{};(-12,4)*{} **\crv{(-4,10) & (-12,10)}?(1)*\dir{>};
  \endxy
\end{equation}
\begin{equation} \label{eq_kappa_dot2}
 \kappa\;\;\;
 \xy 0;/r.17pc/:
    (4,-4)*{};(-4,4)*{} **\crv{(4,-1) & (-4,1)}?(1)*\dir{}+(0,0)*{\bullet};
    (-4,-4)*{};(4,4)*{} **\crv{(-4,-1) & (4,1)};
     (-4,4)*{};(18,4)*{} **\crv{(-4,16) & (18,16)} ?(1)*\dir{>};
     (4,-4)*{};(-18,-4)*{} **\crv{(4,-16) & (-18,-16)} ?(1)*\dir{<}?(0)*\dir{<};
     (-18,-4);(-18,12) **\dir{-};(-12,-4);(-12,12) **\dir{-};
     (18,4);(18,-12) **\dir{-};(12,4);(12,-12) **\dir{-};
     (8,1)*{ n};
     (-4,-4)*{};(-12,-4)*{} **\crv{(-4,-10) & (-12,-10)}?(1)*\dir{<}?(0)*\dir{<};
      (4,4)*{};(12,4)*{} **\crv{(4,10) & (12,10)}?(1)*\dir{>}?(0)*\dir{>};
  \endxy \quad = \quad \kappa \;\;\;
   \xy 0;/r.17pc/:
    (4,-4)*{};(-4,4)*{} **\crv{(4,-1) & (-4,1)}?(1)*\dir{>}?(0)*\dir{}+(0,0)*{\bullet};
    (-4,-4)*{};(4,4)*{} **\crv{(-4,-1) & (4,1)};
     (-4,4)*{};(18,4)*{} **\crv{(-4,16) & (18,16)} ?(1)*\dir{>};
     (4,-4)*{};(-18,-4)*{} **\crv{(4,-16) & (-18,-16)} ?(1)*\dir{<};
     (-18,-4);(-18,12) **\dir{-};(-12,-4);(-12,12) **\dir{-};
     (18,4);(18,-12) **\dir{-};(12,4);(12,-12) **\dir{-};
     (8,1)*{ n};
     (-4,-4)*{};(-12,-4)*{} **\crv{(-4,-10) & (-12,-10)}?(1)*\dir{<}?(0)*\dir{<};
      (4,4)*{};(12,4)*{} **\crv{(4,10) & (12,10)}?(1)*\dir{>}?(0)*\dir{>};
  \endxy
 \;\; + \;\; r_{i}  \kappa
 \;\;\;
   \xy 0;/r.17pc/:
    (4,-4)*{};(4,4)*{} **\dir{-}?(.5)*\dir{>};
    (-4,-4)*{};(-4,4)*{} **\dir{-} ?(.5)*\dir{>};
     (-4,4)*{};(18,4)*{} **\crv{(-4,16) & (18,16)} ?(1)*\dir{>};
     (4,-4)*{};(-18,-4)*{} **\crv{(4,-16) & (-18,-16)} ?(1)*\dir{<};
     (-18,-4);(-18,12) **\dir{-};
     (-12,-4);(-12,12) **\dir{-};
     (18,4);(18,-12) **\dir{-};
     (12,4);(12,-12) **\dir{-};
     (8,1)*{ n};
     (-4,-4)*{};(-12,-4)*{} **\crv{(-4,-10) & (-12,-10)}?(1)*\dir{<};
      (4,4)*{};(12,4)*{} **\crv{(4,10) & (12,10)}?(1)*\dir{>};
  \endxy
\end{equation}
The first two terms of \eqref{eq_kappa_dot} are equal to the first two terms of \eqref{eq_kappa_dot2} respectively, using \eqref{eq:cyclic-kappa}.  Thus
\begin{equation}
\xy 0;/r.17pc/:
    (4,-10)*{};(4,10)*{} **\dir{-}?(.5)*\dir{<};
    (-4,-10)*{};(-4,10)*{} **\dir{-} ?(.5)*\dir{<};
     (10,4)*{ n};
  \endxy
 \quad = \quad
 \kappa \;\;\;
   \xy 0;/r.17pc/:
    (4,-10)*{};(4,10)*{} **\dir{-}?(.5)*\dir{<};
    (-4,-10)*{};(-4,10)*{} **\dir{-} ?(.5)*\dir{<};
     (10,4)*{ n};
  \endxy
\end{equation}
so that $\kappa=1$.
\end{proof}

The sideways crossings were defined up to a scalar by Lemma~\ref{lem:homs}. We now fix these scalars as follows.
\begin{equation}
  \xy
  (0,0)*{\xybox{
    (-4,-4)*{};(4,4)*{} **\crv{(-4,-1) & (4,1)}?(1)*\dir{>} ;
    (4,-4)*{};(-4,4)*{} **\crv{(4,-1) & (-4,1)}?(0)*\dir{<};
     (8,2)*{ n};
     (-12,0)*{};(12,0)*{};
     }};
  \endxy
:=
 \xy 0;/r.17pc/:
  (0,0)*{\xybox{
    (4,-4)*{};(-4,4)*{} **\crv{(4,-1) & (-4,1)}?(1)*\dir{>};
    (-4,-4)*{};(4,4)*{} **\crv{(-4,-1) & (4,1)};
     (-4,4);(-4,12) **\dir{-};
     (-12,-4);(-12,12) **\dir{-};
     (4,-4);(4,-12) **\dir{-};(12,4);(12,-12) **\dir{-};
     (16,1)*{n};
     (-10,0)*{};(10,0)*{};
     (-4,-4)*{};(-12,-4)*{} **\crv{(-4,-10) & (-12,-10)}?(1)*\dir{<}?(0)*\dir{<};
      (4,4)*{};(12,4)*{} **\crv{(4,10) & (12,10)}?(1)*\dir{>}?(0)*\dir{>};
     }};
  \endxy
  \qquad  \qquad
  \xy 
  (0,0)*{\xybox{
    (-4,-4)*{};(4,4)*{} **\crv{(-4,-1) & (4,1)}?(0)*\dir{<} ;
    (4,-4)*{};(-4,4)*{} **\crv{(4,-1) & (-4,1)}?(1)*\dir{>};
    (5.1,-3)*{\scs i};
     (-6.5,-3)*{\scs j};
     (9,2)*{ n};
     (-12,0)*{};(12,0)*{};
     }};
  \endxy
\quad := \quad
 \xy 0;/r.17pc/:
  (0,0)*{\xybox{
    (-4,-4)*{};(4,4)*{} **\crv{(-4,-1) & (4,1)}?(1)*\dir{>};
    (4,-4)*{};(-4,4)*{} **\crv{(4,-1) & (-4,1)};
     (4,4);(4,12) **\dir{-};
     (12,-4);(12,12) **\dir{-};
     (-4,-4);(-4,-12) **\dir{-};(-12,4);(-12,-12) **\dir{-};
     (16,-6)*{n};
     (10,0)*{};(-10,0)*{};
     (4,-4)*{};(12,-4)*{} **\crv{(4,-10) & (12,-10)}?(1)*\dir{<}?(0)*\dir{<};
      (-4,4)*{};(-12,4)*{} **\crv{(-4,10) & (-12,10)}?(1)*\dir{>}?(0)*\dir{>};
     }};
  \endxy
\end{equation}

%
\section{Proof of the extended ${\mf{sl}}_2$ relations}\label{sec:proofsl2}
%

By the remarks in section~\ref{sec:sl2convs}, if $I = \{i\}$ then by rescaling degrees it suffices to work with Cartan datum associated to $\mf{sl}_2$.

%
\subsection{Symmetric functions and dotted bubbles} \label{sec:symm}
%

This section will formally define 'fake bubbles' in a $Q$-strong 2-representation. Let ${\rm Sym}$ denote the graded ring of symmetric functions in countably infinite variables. As a graded vector space ${\rm Sym} = \bigoplus_{\ell=0}^{\infty} {\rm Sym}^{\ell}$ where ${\rm Sym}^{\ell}$ denote the homogeneous symmetric functions of degree $\ell$. The ring ${\rm Sym}$ can be identified with the polynomial ring $\Z[e_1,e_2,\dots]$ and the polynomial ring $\Z[h_1,h_2,\dots]$ where $e_j$ is the $j$th elementary symmetric function, and $h_j$ is the $j$th complete symmetric function. These symmetric functions are related by the equation
\begin{equation} \label{eq_eh_rel}
  \sum_{\ell_1+\ell_2=m} (-1)^{\ell_2}e_{\ell_1}h_{\ell_2} = \delta_{m,0},
\end{equation}
where $e_{0}=h_{0}=1$ and $e_r=h_r=0$ for $r<0$. Given a partition $\lambda = (\lambda_1,\lambda_2, \dots, \lambda_m)$ let $e_{\lambda} = e_{\lambda_1}e_{\lambda_2} \dots e_{\lambda_m}$.

From equation \eqref{eq_eh_rel} one can recursively rewrite the complete symmetric functions in terms of the elementary symmetric functions
\begin{equation} \label{eq_complete_alpha}
  h_r = \sum_{\lambda: |\lambda|=r} \alpha_{\lambda} e_{\lambda}
\end{equation}
for some coefficients $\alpha_{\lambda}\in \Z$.

Define a grading preserving map
\begin{eqnarray} \label{eq_symm_iso}
 \phi^n \maps {\rm Sym} &
 \longrightarrow& \End(\1_n) \\
  e_{\lambda}=e_{\lambda_1}\dots e_{\lambda_m} &\mapsto&
  \left\{
  \begin{array}{ccc}
    \xy 0;/r.18pc/:
 (-12,0)*{\cbub{n-1+\lambda_1}};
  (8,0)*{\cbub{n-1+\lambda_2}};
  (36,0)*{\cbub{n-1+\lambda_m}};
  (20,0)*{\cdots};
 (-8,8)*{n};
 \endxy & \quad & \text{if $n > 0$} \\ & &\\
    \xy 0;/r.18pc/:
 (-12,0)*{\ccbub{-n-1+\lambda_1}};
  (8,0)*{\ccbub{-n-1+\lambda_2}};
  (36,0)*{\ccbub{-n-1+\lambda_m}};
  (20,0)*{\cdots};
 (-8,8)*{n};
 \endxy & \quad & \text{if $n<0$.}
  \end{array}
  \right.
\end{eqnarray}
and write
\begin{equation} \label{eq_def_eln}
  e_{\lambda,n} := \phi^n(e_{\lambda}),
\end{equation}
with $\deg(e_{\lambda,n})= 2(\lambda_1 + \lambda_2 + \dots + \lambda_m)$. When $\lambda$ is the empty partition we write $e_{\lambda,n}=e_{\emptyset}=1$.

It was shown in \cite[Proposition 8.2]{Lau1} that this map is an isomorphism in the 2-category $\Ucat(\mathfrak{sl}_2)$. In this context, the map $\phi^n$ provides a natural interpretation to fake bubbles introduced in section~\ref{sec:highersl2}.  Indeed, comparing the homogeneous terms \ref{eq_homo_inf_grass} with equation \eqref{eq_eh_rel} of the infinite Grassmannian equation show that the degree $r$ fake bubble in region $n$ corresponds to $\phi^n((-1)^{r}h_r)$.  In particular, in the 2-category $\Ucat(\mathfrak{sl}_2)$ one has
  \begin{equation} \label{eq_fake_nleqz}
  \vcenter{\xy 0;/r.18pc/:
    (2,-11)*{\cbub{n-1+r}{}};
  (12,-2)*{n};
 \endxy} \;\; = \;\;
\left\{
 \begin{array}{ccl}
   \xsum{\lambda: |\lambda|=r}{}
 \alpha_{\lambda} \;\; \vcenter{\xy 0;/r.18pc/:
    (2,-3)*{\ccbub{-n-1+\lambda_1}{}};
    (13,-2)*{\cdots};
    (28,-3)*{\ccbub{-n-1+\lambda_m}{}};
  (12,8)*{n};
 \endxy}  & \quad & \text{if $0 \leq r < -n+1$}
 \\ \\
   0 & \quad & \text{$r<0$.}
 \end{array}
 \right.
 \end{equation}
for $n\leq 0$, and
  \begin{equation} \label{eq_fake_ngeqz}
  \vcenter{\xy 0;/r.18pc/:
    (2,-11)*{\ccbub{-n-1+r}{}};
  (12,-2)*{n};
 \endxy} \;\; = \;\;
\left\{
 \begin{array}{ccl}
   \xsum{\lambda: |\lambda|=r}{}
 \alpha_{\lambda}  \;\; \vcenter{\xy 0;/r.18pc/:
    (2,-3)*{\cbub{n-1+\lambda_1}{}};
    (13,-2)*{\cdots};
    (28,-3)*{\cbub{n-1+\lambda_m}{}};
  (12,8)*{n};
 \endxy} & \quad & \text{if $0 \leq r < n+1$}
 \\ \\
   0 & \quad & \text{$r<0$.}
 \end{array}
 \right.
 \end{equation}
for $n\geq 0$, where for $n=0$ we define
\begin{equation}
\vcenter{\xy 0;/r.18pc/:
    (2,-11)*{\cbub{-1}{}};
  (12,-2)*{0};
 \endxy} \;\; =  \Id_{\1_{0}}, \qquad \quad
 \vcenter{\xy 0;/r.18pc/:
    (2,-11)*{\ccbub{-1}{}};
  (12,-2)*{0};
 \endxy} \;\; =  \Id_{\1_{0}}.
 \end{equation}  See \cite[Section 3.6]{Lau3} for more details.

%
\subsection{A general form of the ${\mf{sl}}_2$ commutator relation}
%

All string diagrams in this subsection are interpreted in $\cal{K}$, for some 2-category $\cal{K}$ admitting a strong 2-representation of $\mathfrak{sl}_2$.  We also take $n \geq 0$ throughout this section. Results for $n \leq 0$ are proven similarly.

By Corollary~\ref{cor:1} the map
$$\zeta\;\;:=\;\;
\xy 0;/r.15pc/:
    (-4,-4)*{};(4,4)*{} **\crv{(-4,-1) & (4,1)}?(0)*\dir{<} ;
    (4,-4)*{};(-4,4)*{} **\crv{(4,-1) & (-4,1)}?(1)*\dir{>};
  \endxy \;\; \bigoplus_{k=0}^{n-1}
\vcenter{\xy 0;/r.15pc/:
 (-4,2)*{}="t1"; (4,2)*{}="t2";
 "t2";"t1" **\crv{(4,-5) & (-4,-5)}; ?(1)*\dir{>} ?(.8)*\dir{}+(0,-.1)*{\bullet}+(-3,-1)*{\scs k};;
 \endxy}: \F \E \1_n \bigoplus_{k=0}^{n-1} \1_n \la n-1-2k \ra \rightarrow \E \F \1_n
$$
is invertible. We describe its inverse $\zeta^{-1}$ diagrammatically as follows:
\begin{equation}
  \vcenter{\xy 0;/r.15pc/:
    (-4,-8)*{};(4,8)*{} **\crv{(-4,-1) & (4,1)}?(.95)*\dir{>} ?(.1)*\dir{>};
    (4,-8)*{};(-4,8)*{} **\crv{(4,-1) & (-4,1)}?(.05)*\dir{<} ?(.9)*\dir{<};
    (0,0)*{\chern{\zeta(n)}}; (-12,2)*{n};
  \endxy} \;\; \bigoplus_{k=0}^{n-1}
\left( \quad\;\;
 \vcenter{\xy 0;/r.15pc/:
     (-4,-2)*{}="t1";  (4,-2)*{}="t2";
     "t2";"t1" **\crv{(4,5) & (-4,5)};  ?(.9)*\dir{<} ?(.05)*\dir{<};
      (0,6)*{\chern{\zeta(n-1-k)}}; (10,14)*{n};\endxy } \;\; \;\;\right)
 : \E \F \1_{n} \rightarrow\F \E \1_{n} \bigoplus_{k=0}^{n-1} \1_{n} \la n-1-2k \ra.
\end{equation}

Condition \ref{co:EF} of Definition~\ref{defU_cat} only requires the {\it{existence}} of isomorphisms between the two 1-morphisms on either side. However, the space of 2-morphisms between a pair of 1-morphisms in $\cal{K}$ could contain maps that cannot be expressed using 2-morphisms from the strong 2-representation of $\mf{sl}_2$, i.e. using dots, crossings, caps and cups. In the next proposition we show that this is not the case for the 2-morphisms giving the isomorphism $\zeta^{-1}$.

\begin{prop} \label{prop_form-of-inv}
The isomorphism $\zeta^{-1}$ has the form
\begin{equation} \label{eq_phi-inverses-U}
  \vcenter{\xy 0;/r.15pc/:
    (-4,-8)*{};(4,8)*{} **\crv{(-4,-1) & (4,1)}?(.95)*\dir{>} ?(.1)*\dir{>};
    (4,-8)*{};(-4,8)*{} **\crv{(4,-1) & (-4,1)}?(.05)*\dir{<} ?(.9)*\dir{<};
    (0,0)*{\chern{\zeta(n)}}; (-12,2)*{n};
  \endxy} \;\;=\;\; \beta_n
  \;
  \vcenter{\xy 0;/r.15pc/:
    (-4,-8)*{};(4,8)*{} **\crv{(-4,-1) & (4,1)}?(.95)*\dir{>} ?(.1)*\dir{>};
    (4,-8)*{};(-4,8)*{} **\crv{(4,-1) & (-4,1)}?(.05)*\dir{<} ?(.9)*\dir{<};
    (12,2)*{n};
  \endxy\;\;}
\qquad  \qquad
\vcenter{\xy 0;/r.15pc/:
     (-4,-2)*{}="t1";  (4,-2)*{}="t2";
     "t2";"t1" **\crv{(4,5) & (-4,5)};  ?(.9)*\dir{<} ?(.05)*\dir{<};
     (0,6)*{\chern{\zeta(\ell)}};
      (12,10)*{n};\endxy}
\;\; = \;\;
\sum_{|\lambda|+j=\ell} \;\; \alpha_{\lambda}^{\ell}(n)e_{\lambda,n}\;
\vcenter{\xy 0;/r.15pc/:
     (-4,-2)*{}="t1";  (4,-2)*{}="t2";
     "t2";"t1" **\crv{(4,5) & (-4,5)};  ?(.9)*\dir{<} ?(.05)*\dir{<}
     ?(.75)*\dir{}+(0,0)*{\bullet}+(-1,3)*{\scs j};
    (10,10)*{n};\endxy}
\end{equation}
for some coefficients $\beta_n \in \Bbbk^{\times}$ and $\alpha_{\lambda}^{\ell}(n) \in \Bbbk$.
\end{prop}
\begin{proof}
The first claim in the Proposition follows immediately from Lemma~\ref{lem:homs}.  For the second claim  take adjoints in Lemma~\ref{lem:main} equation \eqref{eq:main1} so that
\begin{equation} \label{eq_phi-inverses}
\vcenter{\xy 0;/r.15pc/:
(-4,-2)*{}="t1";  (4,-2)*{}="t2";
     "t2";"t1" **\crv{(4,5) & (-4,5)};  ?(.9)*\dir{<} ?(.05)*\dir{<};
      (0,6)*{\chern{\zeta(\ell)}}; (12,10)*{n};\endxy}
\;\; = \;\;
\sum_{i=0}^{\ell} \;\; \;
\vcenter{\xy 0;/r.15pc/:
     (-4,-2)*{}="t1";  (4,-2)*{}="t2";
     "t2";"t1" **\crv{(4,5) & (-4,5)};  ?(.9)*\dir{<} ?(.05)*\dir{<}
     ?(.75)*\dir{}+(0,0)*{\bullet}+(-5,1)*{\scs \ell-i};
      (0,10)*{\chern{f_i(\ell)}}; (12,10)*{n};\endxy}
\;\; :\E\F \1_n \rightarrow \1_n \la 2 \ell+1-n \ra
\end{equation}
for some 2-endomorphisms $f_i(\ell) \in \End_{\Ucat}(\1_{n},\1_{n}\la 2i \ra)$.

The component of $\zeta^{-1} \zeta$ mapping the summand $\1_{n}\la 2\ell'+1-n\ra$ to the summand $\1_{n}\la 2\ell+1-n\ra$ is given by the composite
\begin{equation}
\xy
 (-40,0)*+{\1_n\la 2\ell'+1-n\ra}="1";
 (0,0)*+{\E\F\1_n}="2";
 (40,0)*+{\1_n\la 2\ell+1-n \ra .}="3";
 {\ar^{\vcenter{\xy 0;/r.15pc/:
     (4,2)*{}="t1";  (-4,2)*{}="t2";
     "t2";"t1" **\crv{(-4,-5) & (4,-5)};  ?(.9)*\dir{<} ?(.05)*\dir{<}
     ?(.3)*\dir{}+(0,0)*{\bullet}+(-5,-3)*{\scs \quad n-1-\ell'};
 (10,-8)*{n};\endxy}} "1";"2"};
 {\ar^-{\sum_{i} \;
\vcenter{\xy 0;/r.15pc/:
     (-4,-2)*{}="t1";  (4,-2)*{}="t2";
     "t2";"t1" **\crv{(4,5) & (-4,5)};  ?(.9)*\dir{<} ?(.05)*\dir{<}
     ?(.75)*\dir{}+(0,0)*{\bullet}+(-5,1)*{\scs \ell-i};
      (0,10)*{\chern{f_i(\ell)}}; (12,10)*{n};\endxy}} "2";"3"};
\endxy
\end{equation}
The condition $\zeta^{-1} \zeta = \Id$ implies that this composite must satisfy the equation
\begin{equation}
 \sum_{i=0}^{\ell}
\vcenter{\xy 0;/r.15pc/:
     (2,-6)*{\cbub{n-1+\ell-\ell'-i}{}};
      (0,10)*{\chern{f_{i}(\ell)}};
      (12,2)*{n};\endxy}
\;\; = \;\; \delta_{\ell,\ell'}
\end{equation}
for $0 \leq \ell', \ell \leq n-1$.

By varying $\ell'$ for fixed $\ell$, it is possible to rewrite each of the 2-endomorphisms $f_i(\ell)$ as products of dotted bubbles.  For example, by setting $\ell'=\ell$ it follows that the degree zero 2-endomorphism $f_{0}(\ell)$ is multiplication by the scalar $1$.  Continuing by induction, decreasing $\ell'$ shows that all the $f_i(\ell)$ can be rewritten as a linear combination
\begin{equation}
\vcenter{\xy 0;/r.15pc/:
      (0,0)*{\chern{f_{i}(\ell)}};
      (12,6)*{n};\endxy} \;\; = \;\;
 \sum_{\lambda: |\lambda|=i}\;\; \alpha_{\lambda}^{\ell}(n) e_{\lambda,n}
\end{equation}
of products $e_{\lambda,n}$ of dotted bubbles for some  $\alpha_{\lambda}^{\ell}(n) \in \Bbbk$, completing the proof.
\end{proof}

For notational simplicity we write $\alpha_{\lambda}^{\ell}(n)$ for all $\ell$ and $\lambda$ and $n$, but set $\alpha_{\lambda}^{\ell}(n)=0$ unless $|\lambda|\leq \ell \leq n-1$.

%
\subsection{Relations resulting from the ${\mf{sl}}_2$ commutator relation}
%

The plan in this section is to show that $\alpha_\lambda^{(\ell)}(n)$ is equal to $\alpha_\lambda$ (as defined in equation \ref{eq_complete_alpha}) and moreover that $\beta_n = -r_{i}^2$. Observe that $\zeta^{-1}$ in Proposition~\ref{prop_form-of-inv} is the inverse of $\zeta$  if and only if the following relations hold in $\cal{K}$:

\begin{center}
\begin{tabular}{|l|c|}
\hline \multicolumn{2}{|c|}{\bf Relations for $n \geq 0$} \\ \hline & \\
{\bf (A1)} &
\begin{tabular}{c}
  $\textcolor[rgb]{0.00,0.00,1.00}{
  \1_{n}\la 2\ell'+1-n\ra \to \1_{n}\la 2\ell+1-n\ra}$ \\
  $   \delta_{\ell,\ell'} = \xsum{\lambda: |\lambda| \leq \ell}{}\; \alpha_{\lambda}^{\ell}(n)\; e_{\lambda,n}
\xy 0;/r.16pc/:
 (0,0)*{\cbub{n-1+\ell-\ell'-|\lambda|}};
  (4,8)*{n};
 \endxy
$ \\
\end{tabular}
   \\  & \\
   \hline
   & \\
{\bf (A2)} &
\begin{tabular}{c}
 $\textcolor[rgb]{0.00,0.00,1.00}{\F\E\1_{n} \to \F\E\1_{n}}$ \\
 \\
 $ \vcenter{\xy 0;/r.16pc/:
  (-8,0)*{};(-6,-8)*{\scs };(6,-8)*{\scs };
  (8,0)*{};
  (-4,10)*{}="t1";
  (4,10)*{}="t2";
  (-4,-10)*{}="b1";
  (4,-10)*{}="b2";
  "t1";"b1" **\dir{-} ?(.5)*\dir{>};
  "t2";"b2" **\dir{-} ?(.5)*\dir{<};
  (10,2)*{n};
  (-10,2)*{n};
  \endxy}
\;\; = \;\;
 \beta_{n}\;\;
   \vcenter{\xy 0;/r.16pc/:
    (-4,-6)*{};(4,4)*{} **\crv{(-4,-1) & (4,1)}?(1)*\dir{<};?(0)*\dir{<};
    (4,-6)*{};(-4,4)*{} **\crv{(4,-1) & (-4,1)}?(1)*\dir{>};
    (-4,4)*{};(4,14)*{} **\crv{(-4,7) & (4,9)}?(1)*\dir{>};
    (4,4)*{};(-4,14)*{} **\crv{(4,7) & (-4,9)};
  (8,8)*{n};(-6,-3)*{\scs };  (6,-3)*{\scs };
 \endxy}$  \\
\end{tabular}
   \\  & \\
   \hline
   & \\
{\bf (A3)} &
\begin{tabular}{c}
  $\textcolor[rgb]{0.00,0.00,1.00}{\1_{n}\la 2\ell+1-n\ra \to \F\E\1_{n}}$\\
  \\
 $ \beta_{n} \;\; \vcenter{\xy 0;/r.16pc/:
    (-4,6)*{};(4,-4)*{} **\crv{(-4,1) & (4,-1)}?(1)*\dir{>};
    (4,6)*{};(-4,-4)*{} **\crv{(4,1) & (-4,-1)}?(0)*\dir{<};
    (-4,-4)*{};(4,-4)*{} **\crv{(-4,-8) & (4,-8)}?(.5)*\dir{}+(0,0)*{\bullet}+(-1,-3)*{\scs n-1-\ell};
  (12,-2)*{n};(-6,-3)*{\scs };  (6,-3)*{\scs };
 \endxy} \;\; = \;\; 0$  \\
\end{tabular}
     \\  & \\
   \hline
   & \\
{\bf (A4)} &
\begin{tabular}{c}
   $\textcolor[rgb]{0.00,0.00,1.00}{\F\E\1_{n} \to \1_{n}\la 2\ell+1-n\ra}$\\
   \\
  $
    \sum_{\lambda} \;\alpha_{\lambda}^{\ell}(n) \;e_{\lambda,n} \vcenter{\xy 0;/r.16pc/:
    (-4,-6)*{};(4,4)*{} **\crv{(-4,-1) & (4,1)}?(1)*\dir{};?(0)*\dir{<};
    (4,-6)*{};(-4,4)*{} **\crv{(4,-1) & (-4,1)}?(1)*\dir{>};
    (-4,4)*{};(4,4)*{} **\crv{(-4,8) & (4,8)}?(.5)*\dir{}+(0,0)*{\bullet}+(-1,3)*{\scs n-1-\ell};
  (12,2)*{n};(-6,-3)*{\scs };  (6,-3)*{\scs };
 \endxy} \;\; = \;\; 0$ \\
\end{tabular}
    \\  & \\
   \hline
   & \\
{\bf (A5)} &
\begin{tabular}{c}
  $\textcolor[rgb]{0.00,0.00,1.00}{\E\F\1_{n} \to \E\F\1_{n}}$ \\
  \\
  $ \vcenter{\xy 0;/r.16pc/:
  (-8,0)*{};
  (8,0)*{};
  (-4,10)*{}="t1";
  (4,10)*{}="t2";
  (-4,-10)*{}="b1";
  (4,-10)*{}="b2";(-6,-8)*{\scs };(6,-8)*{\scs };
  "t1";"b1" **\dir{-} ?(.5)*\dir{<};
  "t2";"b2" **\dir{-} ?(.5)*\dir{>};
  (10,2)*{n};
  (-10,2)*{n};
  \endxy}
\;\; = \;\;
 \beta_{n}\;\;
 \vcenter{   \xy 0;/r.16pc/:
    (-4,-6)*{};(4,4)*{} **\crv{(-4,-1) & (4,1)}?(1)*\dir{>};
    (4,-6)*{};(-4,4)*{} **\crv{(4,-1) & (-4,1)}?(1)*\dir{<};?(0)*\dir{<};
    (-4,4)*{};(4,14)*{} **\crv{(-4,7) & (4,9)};
    (4,4)*{};(-4,14)*{} **\crv{(4,7) & (-4,9)}?(1)*\dir{>};
  (8,8)*{n};(-6,-3)*{\scs };
     (6.5,-3)*{\scs };
 \endxy}
  \;\; + \;\;
   \xsum{ \xy  (0,8)*{}; (0,3)*{\scs f_1+f_2+|\lambda|}; (0,0)*{\scs =n-1};\endxy}{}\;
    \alpha_{\lambda}^{|\lambda|+f_2}(n) e_{\lambda,n}
    \vcenter{\xy 0;/r.16pc/:
    (-10,10)*{n};
    (-8,0)*{};
  (8,0)*{};
  (-4,-10)*{}="b1";
  (4,-10)*{}="b2";
  "b2";"b1" **\crv{(5,-3) & (-5,-3)}; ?(.05)*\dir{<} ?(.93)*\dir{<}
  ?(.8)*\dir{}+(0,-.1)*{\bullet}+(-3,2)*{\scs f_2};
  (-4,10)*{}="t1";
  (4,10)*{}="t2";
  "t2";"t1" **\crv{(5,3) & (-5,3)}; ?(.15)*\dir{>} ?(.95)*\dir{>}
  ?(.4)*\dir{}+(0,-.2)*{\bullet}+(3,-2)*{\scs \; f_1};
  \endxy}$ \\
\end{tabular}
   \\ & \\
  \hline
\end{tabular}
\end{center}

All the curls in the summation (A4) are zero because of (A3) (note that you can rotate the curl in (A3) by using adjunction). So relation (A4) follows immediately from (A3) and does not impose any conditions on the coefficients $\alpha_\l^{\ell}(n)$.

\begin{prop} \label{prop:alpha}
Setting the coefficients $\alpha_{\lambda}^{\ell}(n)$ from Proposition~\ref{prop_form-of-inv} equal to the coefficients $\alpha_{\lambda}$ from equation~\eqref{eq_complete_alpha} gives a unique solution to the equations in (A1).
\end{prop}

\begin{proof}
 First we note that equations (A1) can be written in terms of
$$h_s^\ell := \begin{array}{ccc}
(-1)^s \xsum{\lambda : |\lambda|=s}{}\alpha_{\lambda}^{\ell}(n)e_{\lambda,n} & \quad & \text{for $n \in \Z$},
\end{array}$$
where $0 \le s \le \ell$. More specifically, (A1) can be written as
\begin{equation}\label{eq:A1}
\delta_{b,0} = \sum_{\lambda: |\lambda| \leq b} \alpha_{\lambda}^{\ell}(n)e_{\lambda,n} e_{b-|\lambda|,n} \Longleftrightarrow \delta_{b,0} = \sum_{s \le b} (-1)^s h_s^\ell e_{b-s,n}
\end{equation}
where $b = \ell-\ell'$ while (A5) can be written as
\begin{equation}\label{eq:A5}
{\xy 0;/r.16pc/:
  (-8,0)*{};
  (8,0)*{};
  (-4,10)*{}="t1";
  (4,10)*{}="t2";
  (-4,-10)*{}="b1";
  (4,-10)*{}="b2";(-6,-8)*{\scs };(6,-8)*{\scs };
  "t1";"b1" **\dir{-} ?(.5)*\dir{<};
  "t2";"b2" **\dir{-} ?(.5)*\dir{>};
  (10,2)*{n};
  (-10,2)*{n};
  \endxy}
\;\; = \;\;
 \beta_{n}\;\;
 \vcenter{   \xy 0;/r.16pc/:
    (-4,-6)*{};(4,4)*{} **\crv{(-4,-1) & (4,1)}?(1)*\dir{>};
    (4,-6)*{};(-4,4)*{} **\crv{(4,-1) & (-4,1)}?(1)*\dir{<};?(0)*\dir{<};
    (-4,4)*{};(4,14)*{} **\crv{(-4,7) & (4,9)};
    (4,4)*{};(-4,14)*{} **\crv{(4,7) & (-4,9)}?(1)*\dir{>};
  (8,8)*{n};(-6,-3)*{\scs };
     (6.5,-3)*{\scs };
 \endxy}
  \;\; + \;\;
   \xsum{ \xy  (0,8)*{}; (0,3)*{\scs f_1+f_2 < n};\endxy}{}\;
    (-1)^{n-1-f_1-f_2} h_{n-1-f_1-f_2}^{n-1-f_1}
    \vcenter{\xy 0;/r.16pc/:
    (-10,10)*{n};
    (-8,0)*{};
  (8,0)*{};
  (-4,-10)*{}="b1";
  (4,-10)*{}="b2";
  "b2";"b1" **\crv{(5,-3) & (-5,-3)}; ?(.05)*\dir{<} ?(.93)*\dir{<}
  ?(.8)*\dir{}+(0,-.1)*{\bullet}+(-3,2)*{\scs f_2};
  (-4,10)*{}="t1";
  (4,10)*{}="t2";
  "t2";"t1" **\crv{(5,3) & (-5,3)}; ?(.15)*\dir{>} ?(.95)*\dir{>}
  ?(.4)*\dir{}+(0,-.2)*{\bullet}+(3,-2)*{\scs \; f_1};
  \endxy}
\end{equation}

Now, from (\ref{eq:A1}) we can solve uniquely for $h_b^\ell$ in terms of $h_s^\ell$ for $s < b$. So we just need to choose $\alpha_\l^{\ell}(n)$ so that (\ref{eq:A1}) is satisfied. However we can rewrite (\ref{eq:A1}) as
\begin{equation*}
  \sum_{s+r=b}(-1)^s h^{\ell}_s e_r = \delta_{b,0}
\end{equation*}
where $e_r = e_{r,n}(n)$. This means that choosing $\alpha_\l^{\ell}(n) = \alpha_\l$ gives the unique solution to (A1).
\end{proof}

\begin{cor}
Taking the coefficients $\alpha_\l^{\ell}(n)$ equal to $\alpha_\l$ for all $n$ gives a solution of (A2)--(A5).
 \end{cor}

\begin{proof}
The only equation left to verify is (A5).  Since this equation can also be expressed in terms of the $h_s^{\ell}$ in the proof of Proposition~\ref{prop:alpha}, the same argument shows that these coefficients are uniquely determined.
\end{proof}

Since, by the above proposition, the coefficients $\alpha_{\lambda}^{\ell}(n)$ are independent of $\ell$ and $n$ we will just write $\alpha_{\lambda} = \alpha_{\lambda}^{|\lambda|}(n) = \alpha_{\lambda}^{\ell}(n)$ for all $n$ and $0 \leq \ell \leq n-1$.

It is clear that relations (A1)--(A5) allow you to simplify clockwise oriented curls. Less obvious is that they also allow you to simplify counter-clockwise curls. To see this add a cap with $n$ dots to the bottom of equation (A5). Then simplifying using the NilHecke relations and equation (A3). Finally, use equation (A1) noting that $\alpha_0^0(n)=1$. This way you arrive at a formula for simplifying any type of curl.

We now summarize this fact together with relations (A1)--(A5) in a more compact form utilizing fake bubbles in $\cal{K}$, defined as in \eqref{eq_fake_nleqz} and \eqref{eq_fake_ngeqz}. We write down just the case $n \ge 0$ as the case $n < 0$ is similar.

\begin{itemize}
  \item The biadjointness condition, the cyclic condition , and the NilHecke algebra axioms hold.

 \item  Negative degree bubbles are still zero, but a dotted bubble of degree zero is multiplication by 1 for $n \neq -1$, and equal to multiplication by $c_{-1}$ when $n=-1$.

 \item If $n > 0$ the following relations hold: \newline
\begin{equation}\label{eq:first}
  \xy 0;/r.17pc/:
  (14,8)*{n};
  (-3,-10)*{};(3,5)*{} **\crv{(-3,-2) & (2,1)}?(1)*\dir{>};?(.15)*\dir{>};
    (3,-5)*{};(-3,10)*{} **\crv{(2,-1) & (-3,2)}?(.85)*\dir{>} ?(.1)*\dir{>};
  (3,5)*{}="t1";  (9,5)*{}="t2";
  (3,-5)*{}="t1'";  (9,-5)*{}="t2'";
   "t1";"t2" **\crv{(4,8) & (9, 8)};
   "t1'";"t2'" **\crv{(4,-8) & (9, -8)};
   "t2'";"t2" **\crv{(10,0)} ;
 \endxy\;\; = \;\; 0
\qquad \qquad
  \xy 0;/r.17pc/:
  (-14,8)*{n};
  (3,-10)*{};(-3,5)*{} **\crv{(3,-2) & (-2,1)}?(1)*\dir{>};?(.15)*\dir{>};
    (-3,-5)*{};(3,10)*{} **\crv{(-2,-1) & (3,2)}?(.85)*\dir{>} ?(.1)*\dir{>};
  (-3,5)*{}="t1";  (-9,5)*{}="t2";
  (-3,-5)*{}="t1'";  (-9,-5)*{}="t2'";
   "t1";"t2" **\crv{(-4,8) & (-9, 8)};
   "t1'";"t2'" **\crv{(-4,-8) & (-9, -8)};
   "t2'";"t2" **\crv{(-10,0)} ;
 \endxy \;\; = \;\; -\frac{1}{r_{i}\beta_n}
 \sum_{g_1+g_2=n}^{}
  \xy 0;/r.17pc/:
  (-12,8)*{n};
  (0,0)*{\bbe{}};(2,-8)*{\scs};
  (-12,-2)*{\ccbub{-n-1+g_2}};
  (0,6)*{\bullet}+(3,-1)*{\scs g_1};
 \endxy
\end{equation}
\begin{equation}
 \vcenter{\xy 0;/r.17pc/:
  (-8,0)*{};
  (8,0)*{};
  (-4,10)*{}="t1";
  (4,10)*{}="t2";
  (-4,-10)*{}="b1";
  (4,-10)*{}="b2";(-6,-8)*{\scs };(6,-8)*{\scs };
  "t1";"b1" **\dir{-} ?(.5)*\dir{<};
  "t2";"b2" **\dir{-} ?(.5)*\dir{>};
  (10,2)*{n};
  \endxy}
\;\; = \;\;
 \beta_n\;\;
 \vcenter{   \xy 0;/r.17pc/:
    (-4,-4)*{};(4,4)*{} **\crv{(-4,-1) & (4,1)}?(1)*\dir{>};
    (4,-4)*{};(-4,4)*{} **\crv{(4,-1) & (-4,1)}?(1)*\dir{<};?(0)*\dir{<};
    (-4,4)*{};(4,12)*{} **\crv{(-4,7) & (4,9)};
    (4,4)*{};(-4,12)*{} **\crv{(4,7) & (-4,9)}?(1)*\dir{>};
  (8,8)*{n};(-6,-3)*{\scs };
     (6.5,-3)*{\scs };
 \endxy}
  \;\; +\;
   \sum_{ \xy  (0,3)*{\scs f_1+f_2+f_3}; (0,0)*{\scs =n-1};\endxy}
    \vcenter{\xy 0;/r.17pc/:
    (-10,10)*{n};
    (-8,0)*{};
  (8,0)*{};
  (-4,-15)*{}="b1";
  (4,-15)*{}="b2";
  "b2";"b1" **\crv{(5,-8) & (-5,-8)}; ?(.05)*\dir{<} ?(.93)*\dir{<}
  ?(.8)*\dir{}+(0,-.1)*{\bullet}+(-3,2)*{\scs f_3};
  (-4,15)*{}="t1";
  (4,15)*{}="t2";
  "t2";"t1" **\crv{(5,8) & (-5,8)}; ?(.15)*\dir{>} ?(.95)*\dir{>}
  ?(.4)*\dir{}+(0,-.2)*{\bullet}+(3,-2)*{\scs \; f_1};
  (0,0)*{\ccbub{\scs \quad -n-1+f_2}};
  \endxy}
   \qquad \quad
 \vcenter{\xy 0;/r.17pc/:
  (-8,0)*{};(-6,-8)*{\scs };(6,-8)*{\scs };
  (8,0)*{};
  (-4,10)*{}="t1";
  (4,10)*{}="t2";
  (-4,-10)*{}="b1";
  (4,-10)*{}="b2";
  "t1";"b1" **\dir{-} ?(.5)*\dir{>};
  "t2";"b2" **\dir{-} ?(.5)*\dir{<};
  (10,2)*{n};
  \endxy}
\;\; = \;\;
 \beta_n\;\;
   \vcenter{\xy 0;/r.17pc/:
    (-4,-4)*{};(4,4)*{} **\crv{(-4,-1) & (4,1)}?(1)*\dir{<};?(0)*\dir{<};
    (4,-4)*{};(-4,4)*{} **\crv{(4,-1) & (-4,1)}?(1)*\dir{>};
    (-4,4)*{};(4,12)*{} **\crv{(-4,7) & (4,9)}?(1)*\dir{>};
    (4,4)*{};(-4,12)*{} **\crv{(4,7) & (-4,9)};
  (8,8)*{n};(-6,-3)*{\scs };  (6,-3)*{\scs };
 \endxy} \label{eq_EFdecomp_ngz}
\end{equation}
where all bubbles appearing above are fake bubbles.

\item If $n=0$ then
\begin{equation} \label{eq_ident_decompT}
 \vcenter{\xy 0;/r.17pc/:
  (-8,0)*{};
  (8,0)*{};
  (-4,10)*{}="t1";
  (4,10)*{}="t2";
  (-4,-10)*{}="b1";
  (4,-10)*{}="b2";(-6,-8)*{\scs };(6,-8)*{\scs };
  "t1";"b1" **\dir{-} ?(.5)*\dir{<};
  "t2";"b2" **\dir{-} ?(.5)*\dir{>};
  (10,2)*{n};
  \endxy}
\;\; = \;\;
 \beta_0\;\;
 \vcenter{   \xy 0;/r.17pc/:
    (-4,-4)*{};(4,4)*{} **\crv{(-4,-1) & (4,1)}?(1)*\dir{>};
    (4,-4)*{};(-4,4)*{} **\crv{(4,-1) & (-4,1)}?(1)*\dir{<};?(0)*\dir{<};
    (-4,4)*{};(4,12)*{} **\crv{(-4,7) & (4,9)};
    (4,4)*{};(-4,12)*{} **\crv{(4,7) & (-4,9)}?(1)*\dir{>};
  (8,8)*{n};(-6,-3)*{\scs };
     (6.5,-3)*{\scs };
 \endxy}
   \qquad \quad
 \vcenter{\xy 0;/r.17pc/:
  (-8,0)*{};(-6,-8)*{\scs };(6,-8)*{\scs };
  (8,0)*{};
  (-4,10)*{}="t1";
  (4,10)*{}="t2";
  (-4,-10)*{}="b1";
  (4,-10)*{}="b2";
  "t1";"b1" **\dir{-} ?(.5)*\dir{>};
  "t2";"b2" **\dir{-} ?(.5)*\dir{<};
  (10,2)*{n};
  \endxy}
\;\; = \;\;
 \beta_0\;\;
   \vcenter{\xy 0;/r.17pc/:
    (-4,-4)*{};(4,4)*{} **\crv{(-4,-1) & (4,1)}?(1)*\dir{<};?(0)*\dir{<};
    (4,-4)*{};(-4,4)*{} **\crv{(4,-1) & (-4,1)}?(1)*\dir{>};
    (-4,4)*{};(4,12)*{} **\crv{(-4,7) & (4,9)}?(1)*\dir{>};
    (4,4)*{};(-4,12)*{} **\crv{(4,7) & (-4,9)};
  (8,8)*{n};(-6,-3)*{\scs };  (6,-3)*{\scs };
 \endxy}
\end{equation}

Also, we have that
\begin{equation}
  \deg\left(  \text{$  \xy 0;/r.17pc/:
  (14,8)*{n}; (-9,0)*{};(16,0)*{};
  (-3,-10)*{};(3,5)*{} **\crv{(-3,-2) & (2,1)}?(1)*\dir{>};?(.15)*\dir{>};
    (3,-5)*{};(-3,10)*{} **\crv{(2,-1) & (-3,2)}?(.85)*\dir{>} ?(.1)*\dir{>};
  (3,5)*{}="t1";  (9,5)*{}="t2";
  (3,-5)*{}="t1'";  (9,-5)*{}="t2'";
   "t1";"t2" **\crv{(4,8) & (9, 8)};
   "t1'";"t2'" **\crv{(4,-8) & (9, -8)};
   "t2'";"t2" **\crv{(10,0)} ;
 \endxy $} \right) = 0, \qquad \qquad
   \deg\left(\text{$   \xy 0;/r.17pc/:
  (-14,8)*{n}; (9,0)*{};(-16,0)*{};
  (3,-10)*{};(-3,5)*{} **\crv{(3,-2) & (-2,1)}?(1)*\dir{>};?(.15)*\dir{>};
    (-3,-5)*{};(3,10)*{} **\crv{(-2,-1) & (3,2)}?(.85)*\dir{>} ?(.1)*\dir{>};
  (-3,5)*{}="t1";  (-9,5)*{}="t2";
  (-3,-5)*{}="t1'";  (-9,-5)*{}="t2'";
   "t1";"t2" **\crv{(-4,8) & (-9, 8)};
   "t1'";"t2'" **\crv{(-4,-8) & (-9, -8)};
   "t2'";"t2" **\crv{(-10,0)} ;
 \endxy$}\right) = 0.
\end{equation}
Since the space of degree zero endomorphisms $\E \1_n $ is one-dimensional in $\cal{K}$, we must have that both of the above morphisms are multiples of the identity 2-morphism $\Id_{\E \1_n}$.  We keep track of these multiples as follows:
\begin{equation} \label{eq_def_czeropm}
    \text{$  \xy 0;/r.17pc/:
  (14,8)*{n};
  (-3,-10)*{};(3,5)*{} **\crv{(-3,-2) & (2,1)}?(1)*\dir{>};?(.15)*\dir{>};
    (3,-5)*{};(-3,10)*{} **\crv{(2,-1) & (-3,2)}?(.85)*\dir{>} ?(.1)*\dir{>};
  (3,5)*{}="t1";  (9,5)*{}="t2";
  (3,-5)*{}="t1'";  (9,-5)*{}="t2'";
   "t1";"t2" **\crv{(4,8) & (9, 8)};
   "t1'";"t2'" **\crv{(4,-8) & (9, -8)};
   "t2'";"t2" **\crv{(10,0)} ;
 \endxy $}  = - c_0^+\;
   \xy
  (5,4)*{n};
  (0,0)*{\bbe{}};(-2,-8)*{\scs };
 \endxy, \qquad \qquad
   \text{$   \xy 0;/r.17pc/:
  (-14,8)*{n};
  (3,-10)*{};(-3,5)*{} **\crv{(3,-2) & (-2,1)}?(1)*\dir{>};?(.15)*\dir{>};
    (-3,-5)*{};(3,10)*{} **\crv{(-2,-1) & (3,2)}?(.85)*\dir{>} ?(.1)*\dir{>};
  (-3,5)*{}="t1";  (-9,5)*{}="t2";
  (-3,-5)*{}="t1'";  (-9,-5)*{}="t2'";
   "t1";"t2" **\crv{(-4,8) & (-9, 8)};
   "t1'";"t2'" **\crv{(-4,-8) & (-9, -8)};
   "t2'";"t2" **\crv{(-10,0)} ;
 \endxy$} =   c_0^- \xy
  (-5,4)*{n};
  (0,0)*{\bbe{}};(-2,-8)*{\scs };
 \endxy
\end{equation}
\end{itemize}

\begin{lem} \label{lem_coeff} \hfill
\begin{itemize}
  \item The coefficient $c_{-1}$ from equation \eqref{eq_defcmone} satisfies $c_{-1}=1$.
  \item For all values of $n$ such that $\1_n$ is non-zero, we have $\beta_n=-r_{i}^{-2}$.
  \item The coefficients introduced in equation \eqref{eq_def_czeropm} satisfy $c_0^+=c_0^-=r_{i}$.
\end{itemize}
\end{lem}

\begin{proof}
The first equality follows from the equalities
\begin{equation}
0
 \;\; \refequal{\eqref{eq_nil_rels}}\;\;
  \vcenter{\xy 0;/r.18pc/:
    (4,-10)*{};(0,0)*{} **\crv{(4,-5) & (0,-5)};
    (0,0)*{};(4,10)*{} **\crv{(0,5) & (4,5)}?(1)*\dir{>};
  (-6,8)*{+1};
  (1,-1)*{\sccbub{}};
 \endxy}
 \;\; = \;\;
\vcenter{\xy 0;/r.18pc/:
    (-4,-10)*{};(0,0)*{} **\crv{(-4,-5) & (0,-5)};
    (0,0)*{};(-4,10)*{} **\crv{(0,5) & (-4,5)}?(1)*\dir{>};
  (8,8)*{-1};(-5,-3)*{\scs {\bf }};
  (0,-1)*{\sccbub{}};
 \endxy}
  \;\; \refequal{\eqref{eq_EFdecomp_ngz}} \;\;
  \frac{1}{\beta_1}
   \xy 0;/r.18pc/:
  (14,8)*{-1};
  (0,0)*{\bbe{}};(-2,-8)*{\scs };
  (12,-2)*{\nccbub};
 \endxy
 \;\; - \;\;\frac{1}{\beta_1}
    \xy 0;/r.18pc/:
  (8,8)*{-1};
  (0,0)*{\bbe{}};(-2,-8)*{\scs };
  (-12,-2)*{\ccbub{-2}};
 \endxy
\end{equation}
where the second equality follows from  equation~\ref{eq_EFdecomp_ngz}.

Using the definition of fake bubbles this implies
\begin{equation}
  0=\frac{1}{\beta_1} \left( \;\; c_{-1}\;\;
   \xy 0;/r.18pc/:
  (8,8)*{-1};
  (0,0)*{\bbe{}};(-2,-8)*{\scs };
 \endxy
 \;\; - \;\;
    \xy 0;/r.18pc/:
  (8,8)*{-1};
  (0,0)*{\bbe{}};(-2,-8)*{\scs };
 \endxy \;\;\right)
\end{equation}
so that $c_{-1}=1$.

For $n\geq 0$ it follows that
\begin{equation} \label{eq_coeffreduction}
\frac{1}{\beta_n}
   \xy 0;/r.18pc/:
  (-14,8)*{n+2};
  (0,0)*{\bbe{}};(-2,-8)*{\scs };
  (-12,-2)*{\cbub{(n+2)-1}};
 \endxy
\;\; \refequal{\eqref{eq_EFdecomp_ngz}}\;\; \vcenter{\xy 0;/r.18pc/:
    (4,-10)*{};(0,0)*{} **\crv{(4,-5) & (0,-5)};
    (0,0)*{};(4,10)*{} **\crv{(0,5) & (4,5)}?(1)*\dir{>};
  (-8,8)*{n+2};(5,-3)*{\scs {\bf }};
  (0,-1)*{\scbub}; (-3,-4)*{\bullet}+(-3.5,-2)*{\scs n+1};
 \endxy}
 \;\; = \;\;
  \vcenter{\xy 0;/r.18pc/:
    (-4,-10)*{};(0,0)*{} **\crv{(-4,-5) & (0,-5)};
    (0,0)*{};(-4,10)*{} **\crv{(0,5) & (-4,5)}?(1)*\dir{>};
  (6,8)*{n};
  (-1,-1)*{\scbub}; (-4,-4)*{\bullet}+(-5.5,-2)*{\scs n+1};
 \endxy}
 \;\; \refequal{\eqref{eq_ind_dotslide}}\;\;
 r_{i}\sum_{\ell_1+\ell_2=n}
 \xy 0;/r.18pc/:
  (12,8)*{n};
   (3,-4)*{};(-3,4)*{} **\crv{(3,-1) & (-3,1)};
    (-3,-4)*{};(3,4)*{} **\crv{(-3,-1) & (3,1)};
    (-3,4);(-3,10) **\dir{-} ?(1)*\dir{>};
    (-3,-10);(-3,-4) **\dir{-};
  (9,4)*{}="t1";
  (3,4)*{}="t2";
  (9,-4)*{}="t1'";
  (3,-4)*{}="t2'";
   "t1";"t2" **\crv{(9,9) & (3, 9)};
   "t1'";"t2'" **\crv{(9,-9) & (3, -9)};
  "t1'";"t1" **\dir{-} ?(.5)*\dir{<};
  (-3,5)*{\bullet}+(-4,1)*{\scs \ell_2};
  (3,5)*{\bullet}+(-2,3)*{\scs \ell_1};
 \endxy \;\;\refequal{\eqref{eq:first}}\;\;
  r_{i}
 \xy 0;/r.18pc/:
  (12,8)*{n};
   (3,-4)*{};(-3,4)*{} **\crv{(3,-1) & (-3,1)};
    (-3,-4)*{};(3,4)*{} **\crv{(-3,-1) & (3,1)};
    (-3,4);(-3,10) **\dir{-} ?(1)*\dir{>};
    (-3,-10);(-3,-4) **\dir{-};
  (9,4)*{}="t1";
  (3,4)*{}="t2";
  (9,-4)*{}="t1'";
  (3,-4)*{}="t2'";
   "t1";"t2" **\crv{(9,9) & (3, 9)};
   "t1'";"t2'" **\crv{(9,-9) & (3, -9)};
  "t1'";"t1" **\dir{-} ?(.5)*\dir{<};
  (3,5)*{\bullet}+(-2,3)*{\scs n};
 \endxy
\end{equation}
When $n=0$ the above together with equation \eqref{eq_def_czeropm}  implies $c_0^+=-\frac{1}{\beta_0r_{i}}$. For $n>0$, simplifying using the NilHecke dot slide formula implies $\beta_n=-r_{i}^{-2}$. A similar calculation for $n \leq 0$ using a clockwise oriented bubble with $-n+1$ dots implies $c_0^-=-\frac{1}{\beta_0r_{i}}$, and then $\beta_n=-r_{i}^{-2}$ for $n<0$.

Capping off the top of \eqref{eq_ident_decompT} for $n=0$ and simplifying the resulting diagram implies that $\beta_0=-r_{i}^2$, and that $c_0^+=c_0^-=r_{i}$.
\end{proof}

\begin{thm}
Any strong 2-representation of $\mathfrak{sl}_2$ on a 2-category $\cal{K}$ extends to a 2-representation $ \Ucat_Q(\mf{g}) \to \cal{K}$.
\end{thm}
\begin{proof}
This proof is immediate from Lemma~\ref{lem_coeff} and relations (\ref{eq:first}) through (\ref{eq_ident_decompT}).
\end{proof}

%
\section{Proof of the non-${\mf{sl}}_2$ relations}\label{sec:proofnonsl2}
%

Fix a given $Q$-strong 2-representation of $\mf{g}$. In this section we prove the relations in $\Ucat_Q(\mf{g})$ involving strands with different labels (namely, the mixed relations and $Q$-cyclicity involving differently labeled strands). We refer to these as non-${\mf{sl}}_2$ relations.

\subsection{Preliminary results}

We begin with a couple of results on spaces of Homs.

\begin{lem}\label{lem:homEiEj}
For $i \ne j \in I$, we have
\begin{eqnarray}\label{eq:4}
\Hom^k(\E_i \E_j \1_\l, \E_j \E_i \1_\l) \cong
\left\{\begin{array}{ll}
0 & \mbox{if $k < -(\alpha_i, \alpha_j)$} \\
\k & \mbox{if $k = - (\alpha_i, \alpha_j)$}
\end{array}
\right.
\end{eqnarray}
while
\begin{eqnarray}\label{eq:5}
\Hom^k(\E_i \E_j \1_\l, \E_i \E_j \1_\l) \cong
\left\{\begin{array}{ll}
0 & \mbox{if $k < 0$,} \\
\k & \mbox{if $k = 0$.}
\end{array}
\right.
\end{eqnarray}
The same results hold if we replace all the $\E$s by $\F$s.
\end{lem}
\begin{proof}
If the Dynkin diagram of $\g$ is simply laced then this result follows from Lemma 4.5 of \cite{CK3}. But since we are dealing with arbitrary symmetrizable Kac-Moody algebras we reproduce the proof here.

Suppose $\la j, \l \ra \le 0$ (the case $\la j, \l \ra \ge 0$ is similar). The proof is then by induction on $\la j, \l \ra$. The base case is vacuous since $\1_\l = 0$ if $\la j, \l \ra \ll 0$. Recall that $d_i := (\alpha_i, \alpha_i)/2$ for any $i \in I$. Note that under these conventions we have
$$(\E_i \1_\l)_L \cong \1_\l \F_i \la -(\l, \alpha_i)-d_i \ra \hspace{.2cm} \text{ and } \hspace{.2cm}
(\1_\l \F_i)_L \cong \E_i \1_\l \la (\l, \alpha_i)+d_i \ra.$$
Now:
\begin{eqnarray*}
& & \Hom^k(\E_i \E_j \1_\l, \E_j \E_i \1_\l) \\
&\cong& \Hom^k((\E_j \1_{\l+\alpha_i})_L \E_i \E_j \1_\l, \E_i \1_\l) \\
&\cong& \Hom^k(\F_j \la - (\l+\alpha_i, \alpha_j) - d_j \ra \E_i \E_j \1_\l, \E_i \1_\l) \\
&\cong& \Hom^k(\E_i \F_j \E_j \1_\l, \E_i \la (\l+\alpha_i, \alpha_j) + d_j \ra \1_\l) \\
&\cong& \Hom^k(\oplus_{[- \la \l,j \ra]_j} \E_i \1_\l, \E_i \la (\l+\alpha_i, \alpha_j)+d_j \ra \1_\l) \oplus \Hom^k(\E_i \E_j \F_j \1_\l, \E_i \la (\l+\alpha_i, \alpha_j)+d_j \ra \1_\l) \\
&\cong& \bigoplus_{s=0}^{- \la j,\l \ra-d_j} \Hom^k(\E_i \1_\l, \E_i \la (\alpha_i, \alpha_j) - 2sd_j \ra \1_\l) \\
& & \oplus \Hom^k(\E_i \E_j \1_{\l-\alpha_j}, \E_i (\1_{\l-\alpha_j} \F_j)_L  \la (\l+\alpha_i, \alpha_j) + d_j\ra)
\end{eqnarray*}
where to get the last equality we use that $\la j,\l \ra \cdot d_j = (\alpha_j, \l)$. Now, $\Hom^k(\E_i \1_\l, \E_i \1_\l)$ is zero if $k < 0$ and equals $\k$ if $k=0$. So the summation above is zero if $k + (\alpha_i,\alpha_j) < 0$ and $\k$ if $k + (\alpha_i,\alpha_j) = 0$. Meanwhile the second term above equals
\begin{equation}\label{eq:6}
\Hom^k(\E_i \E_j \1_{\l-\alpha_j}, \E_i \E_j \la 2 (\l, \alpha_j) + (\alpha_i,\alpha_j) \ra \1_{\l-\alpha_j}).
\end{equation}
We now use induction with respect to \eqref{eq:5} to show that  $\Hom^t(\E_i \E_j \1_{\l-\alpha_j}, \E_i \E_j \1_{\l-\alpha_j})$ is zero if $t < 0$ and $\k$ if $t=0$. Thus (\ref{eq:6}) vanishes since $(\alpha_i, \alpha_j) < 0$, and this completes the proof of (\ref{eq:4}). The proof of equation (\ref{eq:5}) is similar.
\end{proof}

\begin{cor}\label{cor:homEiFj}
For $i \ne j \in I$, we have
\begin{eqnarray}\label{eq:7}
\Hom^k(\E_i \F_j \1_\l, \F_j \E_i \1_\l) \cong
\left\{\begin{array}{ll}0 & \mbox{if $k < 0$} \\
\k & \mbox{if $k = 0$} \end{array} \right.
\end{eqnarray}
The same results hold if we interchange all the $\E$s with $\F$s and all $\F$s with $\E$s.
\end{cor}
\begin{proof}
By adjunction we have:
\begin{eqnarray*}
& & \Hom^k(\E_i \F_j \1_\l, \F_j \E_i \1_\l) \\
&\cong& \Hom^k((\1_{\l+\alpha_i - \alpha_j} \F_j)_L \E_i \1_{\l-\alpha_j}, \E_i (\1_{\l-\alpha_j} \F_j)_L \1_{\l-\alpha_j}) \\
&\cong& \Hom^k(\E_j \E_i \la (\l+\alpha_i - \alpha_j, \alpha_j) + d_j \ra \1_{\l-\alpha_j}, \E_i \E_j \la (\l - \alpha_j, \alpha_j) + d_j \ra \1_{\l-\alpha_j}) \\
&\cong& \Hom^k(\E_j \E_i \1_{\l-\alpha_j}, \E_i \E_j \la - (\alpha_i, \alpha_j) \ra \1_{\l-\alpha_j})
\end{eqnarray*}
where, as before, $d_j = (\alpha_j,\alpha_j)/2$. By lemma \ref{lem:homEiEj}, this is isomorphic to zero if $k < 0$, and $\k$ if $k=0$.
\end{proof}

%
\subsection{Mixed relations $\E_i \F_j \cong \F_j \E_i$}
%

Recall the definitions of sideways crossings from (\ref{eq_crossl-gen}) and (\ref{eq_crossr-gen}). Note that these 2-morphisms are non-zero since they are related by adjunction to the unique maps in $\Hom^{-(\alpha_i,\alpha_j)}(\E_i \E_j \1_\l, \E_j \E_i \1_\l)$.

Composing these maps gives rise to degree zero endomorphisms of $\E_i\F_j \1_\l$ and of $\F_i\E_j \1_\l$ for $i \neq j$. Since $\E_i \F_j \1_\l \cong \F_j \E_i \1_\l$ these Hom spaces are 1-dimensional in degree zero by corollary \ref{cor:homEiFj}. Thus we must have
\begin{align}
 \vcenter{   \xy 0;/r.18pc/:
    (-4,-4)*{};(4,4)*{} **\crv{(-4,-1) & (4,1)}?(1)*\dir{>};
    (4,-4)*{};(-4,4)*{} **\crv{(4,-1) & (-4,1)}?(1)*\dir{<};?(0)*\dir{<};
    (-4,4)*{};(4,12)*{} **\crv{(-4,7) & (4,9)};
    (4,4)*{};(-4,12)*{} **\crv{(4,7) & (-4,9)}?(1)*\dir{>};
  (8,8)*{\lambda};(-6,-3)*{\scs i};
     (6,-3)*{\scs j};
 \endxy}
 \;\;&:= \;\;
 \text{$\xy   0;/r.14pc/:
    (-16,8)*{}="1";
    (-4,8)*{}="2";
    (8,8)*{}="3";
    (-16,-8);"1" **\dir{-};
    "1";"2" **\crv{(-16,16) & (-4,16)} ?(0)*\dir{<} ;
    "2";"3" **\crv{(-4,0) & (8,0)}?(1)*\dir{<};
    "3"; (8,16) **\dir{-};
    (-16,-8)*{}="1'";
    (-4,-8)*{}="2'";
    (8,-8)*{}="3'";
    "1'";"2'" **\crv{(-16,-16) & (-4,-16)} ?(0)*\dir{>} ;
    "2'";"3'" **\crv{(-4,0) & (8,0)};
    "3'"; (8,-16) **\dir{-};
    (14,0)*{\lambda};
    (0,16);(-6,0) **\crv{(0,8) & (-6,10)} ?(0)*\dir{<} ;;
    (0,-16);(-6,0) **\crv{(0,-8) & (-6,-10)};
    (-2.5,-14)*{\scs i};
     (10,-14)*{\scs j};
    \endxy  $}
 \;\; = \;\;\gamma_{ij}(\lambda)\;\;
\xy 0;/r.18pc/:
  (3,9);(3,-9) **\dir{-}?(.55)*\dir{>}+(2.3,0)*{};
  (-3,9);(-3,-9) **\dir{-}?(.5)*\dir{<}+(2.3,0)*{};
  (8,2)*{\lambda};(-5,-6)*{\scs i};     (5.1,-6)*{\scs j};
 \endxy \label{eq_downup_ij-genL}
\\
    \vcenter{\xy 0;/r.18pc/:
    (-4,-4)*{};(4,4)*{} **\crv{(-4,-1) & (4,1)}?(1)*\dir{<};?(0)*\dir{<};
    (4,-4)*{};(-4,4)*{} **\crv{(4,-1) & (-4,1)}?(1)*\dir{>};
    (-4,4)*{};(4,12)*{} **\crv{(-4,7) & (4,9)}?(1)*\dir{>};
    (4,4)*{};(-4,12)*{} **\crv{(4,7) & (-4,9)};
  (8,8)*{\lambda};(-6,-3)*{\scs i};
     (6,-3)*{\scs j};
 \endxy}
 \;\; &:=\;\;
 \text{$\xy   0;/r.14pc/:
    (16,8)*{}="1";
    (4,8)*{}="2";
    (-8,8)*{}="3";
    (16,-8);"1" **\dir{-};
    "1";"2" **\crv{(16,16) & (4,16)} ?(0)*\dir{<} ;
    "2";"3" **\crv{(4,0) & (-8,0)}?(1)*\dir{<};
    "3"; (-8,16) **\dir{-};
    (16,-8)*{}="1'";
    (4,-8)*{}="2'";
    (-8,-8)*{}="3'";
    "1'";"2'" **\crv{(16,-16) & (4,-16)} ?(0)*\dir{>} ;
    "2'";"3'" **\crv{(4,0) & (-8,0)};
    "3'"; (-8,-16) **\dir{-};
    (22,0)*{ \lambda};
    (0,16);(6,0) **\crv{(0,8) & (6,10)} ?(0)*\dir{<} ;;
    (0,-16);(6,0) **\crv{(0,-8) & (6,-10)};
     (-10,-14)*{\scs i};
     (2.5,-14)*{\scs j};
    \endxy $}
 \;\;=\;\; \beta_{ij}(\lambda)\;\;
\xy 0;/r.18pc/:
  (3,9);(3,-9) **\dir{-}?(.5)*\dir{<}+(2.3,0)*{};
  (-3,9);(-3,-9) **\dir{-}?(.55)*\dir{>}+(2.3,0)*{};
  (8,2)*{\lambda};(-5,-6)*{\scs i};     (5.1,-6)*{\scs j};
 \endxy \label{eq_downup_ij-genR}
\end{align}
for some $\gamma_{ij}(\l), \beta_{ij}(\l) \in \k$.

\begin{prop}
For all $i,j \in I$ with $i \neq j$ the coefficients from \eqref{eq_downup_ij-genL} and \eqref{eq_downup_ij-genR} satisfy
\[
 \beta_{ij}(\lambda) = \gamma_{ji}(\lambda) = t_{ij}.
\]

\end{prop}

\begin{proof}
Equations \eqref{eq_downup_ij-genL} and \eqref{eq_downup_ij-genR} imply
\begin{equation}
\gamma_{ji}(\lambda)\;  \xy 0;/r.18pc/:
  (0,0)*{\xybox{
    (-4,-4)*{};(4,4)*{} **\crv{(-4,-1) & (4,1)}?(0)*\dir{<} ;
    (4,-4)*{};(-4,4)*{} **\crv{(4,-1) & (-4,1)}?(1)*\dir{>};
    (5.1,-3)*{\scs j};
     (-6.5,-3)*{\scs i};
     (9,2)*{ \lambda};
     (-12,0)*{};(12,0)*{};
     }};
  \endxy \;\; = \;\;
   \vcenter{   \xy 0;/r.18pc/:
    (-4,-12)*{};(4,-4)*{} **\crv{(-4,-9) & (4,-7)}?(0)*\dir{<};
    (4,-12)*{};(-4,-4)*{} **\crv{(4,-9) & (-4,-7)};
    (-4,-4)*{};(4,4)*{} **\crv{(-4,-1) & (4,1)}?(1)*\dir{>};
    (4,-4)*{};(-4,4)*{} **\crv{(4,-1) & (-4,1)}?(1)*\dir{<};?(0)*\dir{<};
    (-4,4)*{};(4,12)*{} **\crv{(-4,7) & (4,9)};
    (4,4)*{};(-4,12)*{} **\crv{(4,7) & (-4,9)}?(1)*\dir{>};
  (8,8)*{\lambda};(-6,-11)*{\scs i};
     (6,-11)*{\scs j};
 \endxy}
 \;\; =\;\; \beta_{ij}(\lambda) \;
   \xy 0;/r.18pc/:
  (0,0)*{\xybox{
    (-4,-4)*{};(4,4)*{} **\crv{(-4,-1) & (4,1)}?(0)*\dir{<} ;
    (4,-4)*{};(-4,4)*{} **\crv{(4,-1) & (-4,1)}?(1)*\dir{>};
    (5.1,-3)*{\scs j};
     (-6.5,-3)*{\scs i};
     (9,2)*{ \lambda};
     (-12,0)*{};(12,0)*{};
     }};
  \endxy
\end{equation}
so that $\gamma_{ji}(\lambda)=\beta_{ij}(\lambda)$ since the 2-morphism on the right and left is non-zero.

When $-\la i,\lambda+\alpha_j\ra +1 \leq 0$ closing off the left side of \eqref{eq_downup_ij-genR} with $\la i,\lambda+\alpha_j\ra -1$ dots and simplifying the left-hand side using the KLR relation~\eqref{eq_r2_ij-gen} and \eqref{eq_dot_slide_ij-gen} implies $\beta_{ij}(\lambda)=t_{ij}$ for all $i,j \in I$ and $\lambda$ satisfying $1+d_{ij} \leq \la i, \lambda \ra$.  When $\la i,\l\ra \leq 1$, one can perform an analogous computation with \eqref{eq_downup_ij-genL}: switch the labels $i$ and $j$, close off the right-hand side with $-\la i,\l\ra+1$ dots, and simplify to obtain $\gamma_{ji}(\lambda) = t_{ij}$.

It remains to show that $\beta_{ij}(\lambda)=t_{ij}$ when $1 <\la i, \lambda\ra \leq d_{ij}$.  Consider the case when $\la i ,\lambda \ra = 2$ so that $-\la i ,\lambda-\alpha_i \ra =0$.
Then switch the labels of $i$ and $j$ in \eqref{eq_downup_ij-genL}, use that $\gamma_{ji}(\lambda) = \beta_{ij}(\lambda)$, and glue the result into the diagram
\[
  \xy 0;/r.17pc/:
  (-6,12)*{\lambda-\alpha_i};
  (3,-12)*{};(-3,5)*{} **\crv{(3,-2) & (-2,1)}?(.15)*\dir{>} ?(.85)*\dir{>};
    (-3,-5)*{};(3,15)*{} **\crv{(-2,-1) & (3,2)}?(.9)*\dir{>} ?(.1)*\dir{>};
  (-3,5)*{}="t1";  (-12,5)*{}="t2";
  (-3,-5)*{}="t1'";  (-12,-5)*{}="t2'";
   "t1";"t2" **\crv{(-4,8) & (-12, 8)};
   "t1'";"t2'" **\crv{(-4,-8) & (-12, -8)};
  (-8,5)*{}="tl";
  (-24,5)*{}="tr";
  (-8,-5)*{}="bl";
  (-24,-5)*{}="br";
  "tl";"tr" **\dir{.};
  "tl";"bl" **\dir{.};
  "bl";"br" **\dir{.};
  "tr";"br" **\dir{.};
  (5,-11)*{\scs i}; (-13,-11)*{\scs j};
  (-15,-12)*{}; (-20,-5)*{} **\crv{(-15,-8) & (-20,-7)}?(.5)*\dir{>};
 (-20,5)*{}; (-15,15)*{} **\crv{(-20,7) & (-15,8)}?(.5)*\dir{>};  \endxy
\]
to conclude that
\[
r_i\beta_{ij} \;\;
  \xy 0;/r.17pc/:
  (-9,12)*{\lambda-\alpha_i};
  (3,-11)*{\scs i}; (-15,-11)*{\scs j};
 (-18,-12)*{}; (-18,15)*{} **\dir{-} ?(.5)*\dir{>};
 (0,-12)*{}; (0,15)*{} **\dir{-} ?(.5)*\dir{>};
 \endxy
\quad = \quad  \beta_{ij}\;\;
  \xy 0;/r.17pc/:
  (-7,12)*{\lambda-\alpha_i};
  (3,-12)*{};(-3,5)*{} **\crv{(3,-2) & (-2,1)}?(.15)*\dir{>} ?(.85)*\dir{>};
    (-3,-5)*{};(3,15)*{} **\crv{(-2,-1) & (3,2)}?(.9)*\dir{>} ?(.1)*\dir{>};
  (-3,5)*{}="t1";  (-12,0)*{}="t2";
  (-3,-5)*{}="t1'";  (-12,0)*{}="t2'";
   "t1";"t2" **\crv{(-4,8) & (-12, 8)};
   "t1'";"t2'" **\crv{(-4,-8) & (-12, -8)};
  (5,-11)*{\scs i}; (-15,-11)*{\scs j};
 (-18,-12)*{}; (-18,15)*{} **\dir{-} ?(.5)*\dir{>};
 \endxy
\quad =\quad
  \xy 0;/r.17pc/:
    (-3,5)*{}="t1";  (-14,5)*{}="t2";
  (-3,-5)*{}="t1'";  (-14,-5)*{}="t2'";
  "t2";"t2'" **\dir{-};
   "t1";"t2" **\crv{(-2,12) & (-14, 10)};
   "t1'";"t2'" **\crv{(-2,-12) & (-14, -10)};
  (-20,12)*{\lambda-\alpha_i+\alpha_j};
  (3,-15)*{};"t1" **\crv{(3,-2) & (-2,1)}?(.15)*\dir{>} ?(.85)*\dir{>};
    "t1'";(3,15)*{} **\crv{(-2,-1) & (3,2)}?(.9)*\dir{>} ?(.1)*\dir{>};
  (5,-14)*{\scs i}; (-6,-14)*{\scs j};
  (-3,-15)*{}; (-9,0)*{} **\crv{(-3,-6) & (-9,-6)}?(1)*\dir{>};
 (-9,0)*{}; (-3,15)*{} **\crv{(-9,6) & (-3,6)};
  \endxy
\]
Simplifying using \eqref{eq_r3_hard-gen} and the fact that $d_{ij}-1 = -\la i, \lambda-\alpha_i+\alpha_j\ra -1$ completes the proof for $\la i, \lambda \ra =2$.

For the remaining cases with $2< \la i,\lambda \ra \leq d_{ij}$, switch the labels of $i$ and $j$ in \eqref{eq_downup_ij-genL}, use that $\gamma_{ji}(\lambda) = \beta_{ij}(\lambda)$, and glue the result into the diagram
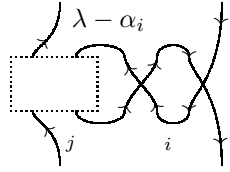
\[
  \xy 0;/r.17pc/:
  (-6,12)*{\lambda-\alpha_i};
    (3,-5)*{};(-3,5)*{} **\crv{(3,-2) & (-2,1)}?(.15)*\dir{>} ?(.85)*\dir{>};
    (-3,-5)*{};(3,5)*{} **\crv{(-2,-1) & (3,2)}?(.9)*\dir{>} ?(.1)*\dir{>};
  (9,5)*{}="t0"; (-3,5)*{}="t1";  (-12,5)*{}="t2";
  (9,-5)*{}="t0'";(-3,-5)*{}="t1'";  (-12,-5)*{}="t2'";
        "t1";"t2" **\crv{(-3,8) & (-12, 8)};
        "t1'";"t2'" **\crv{(-3,-8) & (-12, -8)};
        "t0";(3,5)*{} **\crv{(9,8) & (3, 8)};
        "t0'";(3,-5)*{} **\crv{(9,-8) & (3, -8)};
        (9,-5)*{};(15,15)*{} **\crv{(9,-2) & (15,1)}?(0)*\dir{<} ?(.9)*\dir{<};
        (9,5)*{};(15,-15)*{} **\crv{(9,2) & (15,-1)}?(0)*\dir{>} ?(.9)*\dir{>};
  (-8,5)*{}="tl";
  (-24,5)*{}="tr";
  (-8,-5)*{}="bl";
  (-24,-5)*{}="br";
  "tl";"tr" **\dir{.};
  "tl";"bl" **\dir{.};
  "bl";"br" **\dir{.};
  "tr";"br" **\dir{.};
  (5,-11)*{\scs i}; (-13,-11)*{\scs j};
  (-15,-15)*{}; (-20,-5)*{} **\crv{(-15,-8) & (-20,-7)}?(.5)*\dir{>};
 (-20,5)*{}; (-15,15)*{} **\crv{(-20,7) & (-15,8)}?(.5)*\dir{>};
  \endxy
\]
Using the ``step function" \cite[Equation 3.25]{BKL} (or direct computation) to simplify the  double curl, the right-hand side can be simplified to $-r_i^2\beta_{ij}\Id_{\cal{E}_j\cal{F}_i}$.
The left-hand side simplifies to give the desired result after using equation \eqref{eq_r3_hard-gen} and observing that all terms vanish except one using the assumption that $2< \la i,\lambda \ra \leq d_{ij}$.
\end{proof}

\begin{rem} One can obtain bubble sliding relations by generalizing the arguments in the proof above.
Closing off the left side of \eqref{eq_downup_ij-genR} with $\la i,\lambda+\alpha_j\ra -1+m$ dots, where $m \geq \max(-\la i,\lambda+\alpha_j\ra +1,0)$, then simplifying the left-hand side using the KLR relation~\eqref{eq_r2_ij-gen} implies
\begin{equation} \label{eq_bubslide_lr}
\beta_{ij}\;\xy 0;/r.17pc/:
  (0,9);(0,-9) **\dir{-}?(.5)*\dir{<}+(2.3,0)*{};
  (6,6)*{\lambda};(2,-8)*{\scs j};
  (-12,0)*{\icbub{(\la i,\lambda+\alpha_j\ra -1)+m\qquad }{i}}
 \endxy =
 t_{ij} \;
 \xy 0;/r.17pc/:
  (0,9);(0,-9) **\dir{-}?(.5)*\dir{<}+(2.3,0)*{};
  (18,6)*{\lambda};(-2,-8)*{\scs j};
  (12,0)*{\icbub{(\la i,\lambda\ra -1)+m}{i}}
 \endxy
 + t_{ji}
  \xy 0;/r.17pc/:
  (0,9);(0,-9) **\dir{-}?(.5)*\dir{<}+(2.3,0)*{};
  (0,6)*{\bullet}+(-3.5,1)*{\scs d_{ji}};
  (19,6)*{\lambda};(-2,-8)*{\scs j};
    (12,0)*{\icbub{
        \xy (0,9)*{};(0,3)*{\scs (\la i,\lambda\ra -1)};
        (0,-1)*{\scs  +m-d_{ij} }; \endxy
            }{i}}
 \endxy
+
 \xsum{p,q}{} s_{ji}^{pq}\;
   \xy 0;/r.17pc/:
  (0,9);(0,-9) **\dir{-}?(.5)*\dir{<}+(2.3,0)*{};
    (0,6)*{\bullet}+(-3.5,1)*{\scs p};
  (19,6)*{\lambda};(-2,-8)*{\scs j};
  (12,0)*{\icbub{
        \xy (0,9)*{};(0,3)*{\scs (\la i,\lambda\ra -1)};
        (0,-1)*{\scs  +m+q-d_{ij} }; \endxy
            }{i}}
 \endxy
\end{equation}
Interchanging the labels of the strands in \eqref{eq_downup_ij-genL} and capping off the right strand with $-\la i,\lambda-\alpha_i\ra-1+m=-\la i,\lambda\ra +1+m$ dots, where $m \geq \max(\la i,\lambda\ra -1,0)$, implies
\begin{equation} \label{eq_bubslide_rl}
\gamma_{ij}\;\xy 0;/r.17pc/:
  (0,9);(0,-9) **\dir{-}?(.5)*\dir{<}+(2.3,0)*{};
  (19,8)*{\lambda-\alpha_i};(-2,-8)*{\scs j};
    (12,-2)*{\iccbub{
        \xy (0,9)*{};(0,3)*{\scs -\la i,\lambda-\alpha_i\ra };
        (0,-1)*{\scs -1 +m }; \endxy
            }{i}}
 \endxy =
 t_{ij} \;
 \xy 0;/r.17pc/:
  (0,9);(0,-9) **\dir{-}?(.5)*\dir{<}+(2.3,0)*{};
  (9,6)*{\lambda-\alpha_i};(-2,-8)*{\scs j};
    (-12,0)*{\iccbub{
        \xy (0,9)*{};(0,3)*{\scs (-\la i,\lambda'\ra -1)};
        (0,-1)*{\scs  +m}; \endxy
            }{i}};
 \endxy
 + t_{ji}
  \xy 0;/r.17pc/:
  (0,9);(0,-9) **\dir{-}?(.5)*\dir{<}+(2.3,0)*{};
  (0,6)*{\bullet}+(-3.5,1)*{\scs d_{ji}};
  (10,6)*{\lambda-\alpha_i};
  (2,-8)*{\scs j};
    (-12,0)*{\iccbub{
        \xy (0,9)*{};(0,3)*{\scs (-\la i,\lambda'\ra -1)};
        (0,-1)*{\scs  +m-d_{ij} }; \endxy
            }{i}};
 \endxy
+
 \xsum{p,q}{} s_{ij}^{qp}\;
   \xy 0;/r.17pc/:
  (0,9);(0,-9) **\dir{-}?(.5)*\dir{<}+(2.3,0)*{};
    (0,6)*{\bullet}+(-3.5,1)*{\scs p};
 (10,6)*{\lambda-\alpha_i};
  (2,-8)*{\scs j};
  (-12,0)*{\iccbub{
        \xy (0,9)*{};(0,3)*{\scs (-\la i,\lambda'\ra -1)};
        (0,-1)*{\scs  +m+q-d_{ij} }; \endxy
            }{i}};
 \endxy
\end{equation}
with $\lambda'=\lambda-\alpha_i+\alpha_j$.
\end{rem}

%
\subsection{$Q$-Cyclicity}
%

By uniqueness of the space of 2-morphisms in the appropriate degree we have
\begin{equation} \label{eq_cyclic_cross-gen}
 \xy 0;/r.17pc/:
  (0,0)*{\xybox{
    (4,-4)*{};(-4,4)*{} **\crv{(4,-1) & (-4,1)}?(1)*\dir{>};
    (-4,-4)*{};(4,4)*{} **\crv{(-4,-1) & (4,1)};
     (-4,4)*{};(18,4)*{} **\crv{(-4,16) & (18,16)} ?(1)*\dir{>};
     (4,-4)*{};(-18,-4)*{} **\crv{(4,-16) & (-18,-16)} ?(1)*\dir{<}?(0)*\dir{<};
     (-18,-4);(-18,12) **\dir{-};(-12,-4);(-12,12) **\dir{-};
     (18,4);(18,-12) **\dir{-};(12,4);(12,-12) **\dir{-};
     (8,1)*{ \lambda};
     (-10,0)*{};(10,0)*{};
     (-4,-4)*{};(-12,-4)*{} **\crv{(-4,-10) & (-12,-10)}?(1)*\dir{<}?(0)*\dir{<};
      (4,4)*{};(12,4)*{} **\crv{(4,10) & (12,10)}?(1)*\dir{>}?(0)*\dir{>};
      (-20,11)*{\scs j};(-10,11)*{\scs i};
      (20,-11)*{\scs j};(10,-11)*{\scs i};
     }};
  \endxy
\quad =  \quad b_{ij}(\lambda)
\xy 0;/r.17pc/:
  (0,0)*{\xybox{
    (-4,-4)*{};(4,4)*{} **\crv{(-4,-1) & (4,1)}?(1)*\dir{>};
    (4,-4)*{};(-4,4)*{} **\crv{(4,-1) & (-4,1)};
     (4,4)*{};(-18,4)*{} **\crv{(4,16) & (-18,16)} ?(1)*\dir{>};
     (-4,-4)*{};(18,-4)*{} **\crv{(-4,-16) & (18,-16)} ?(1)*\dir{<}?(0)*\dir{<};
     (18,-4);(18,12) **\dir{-};(12,-4);(12,12) **\dir{-};
     (-18,4);(-18,-12) **\dir{-};(-12,4);(-12,-12) **\dir{-};
     (8,1)*{ \lambda};
     (-10,0)*{};(10,0)*{};
      (4,-4)*{};(12,-4)*{} **\crv{(4,-10) & (12,-10)}?(1)*\dir{<}?(0)*\dir{<};
      (-4,4)*{};(-12,4)*{} **\crv{(-4,10) & (-12,10)}?(1)*\dir{>}?(0)*\dir{>};
      (20,11)*{\scs i};(10,11)*{\scs j};
      (-20,-11)*{\scs i};(-10,-11)*{\scs j};
     }};
  \endxy
\end{equation}
where $b_{ij}(\lambda) \in \Bbbk^{\times}$. Take the the left hand side of \eqref{eq_cyclic_cross-gen}, switch $i$ and $j$, and glue it under the original equation \eqref{eq_cyclic_cross-gen} to obtain
\begin{equation}
   \xy 0;/r.17pc/:
  (0,0)*{\xybox{
    (4,-4)*{};(-4,4)*{} **\crv{(4,-1) & (-4,1)};
    (-4,-4)*{};(4,4)*{} **\crv{(-4,-1) & (4,1)};
    (4,4)*{};(-4,12)*{} **\crv{(4,7) & (-4,9)}?(1)*\dir{>};
    (-4,4)*{};(4,12)*{} **\crv{(-4,7) & (4,9)}?(1)*\dir{>};
    (4,-4)*{};(-18,-4)*{} **\crv{(4,-16) & (-18,-16)} ?(1)*\dir{<}?(0)*\dir{<};
     (-18,-4);(-18,18) **\dir{-};
     (-12,-4);(-12,18) **\dir{-};
     (12,12);(12,-14) **\dir{-}?(1)*\dir{>};
     (8,8)*{ \lambda};
     (-10,0)*{};(10,0)*{};
     (-4,12)*{};(18,12)*{} **\crv{(-4,24) & (18,24)}?(0)*\dir{>};
     (18,12);(18,-14) **\dir{-}?(1)*\dir{>};
     (-4,-4)*{};(-12,-4)*{} **\crv{(-4,-10) & (-12,-10)}?(1)*\dir{<}?(0)*\dir{<};
     (4,12)*{};(12,12)*{} **\crv{(4,18) & (12,18)};
      (-20,11)*{\scs j};(-10,11)*{\scs i};
     (10,-11)*{\scs j}; (20,-11)*{\scs i};
     }};
  \endxy
\;\; = \;\;
  b_{ij}(\lambda)
 \xy 0;/r.17pc/:
  (1.35,-9)*{\xybox{
    (4,-4)*{};(-4,4)*{} **\crv{(4,-1) & (-4,1)}?(1)*\dir{>};
    (-4,-4)*{};(4,4)*{} **\crv{(-4,-1) & (4,1)};
     (-4,4)*{};(18,4)*{} **\crv{(-4,16) & (18,16)} ?(1)*\dir{>};
     (4,-4)*{};(-18,-4)*{} **\crv{(4,-16) & (-18,-16)} ?(1)*\dir{<}?(0)*\dir{<};
     (-18,-4);(-18,18) **\dir{-};
     (-12,-4);(-12,18) **\dir{-};
     (18,4);(18,-12) **\dir{-};(12,4);(12,-12) **\dir{-};
     (8,1)*{ \lambda};
     (-10,0)*{};(10,0)*{};
     (-4,-4)*{};(-12,-4)*{} **\crv{(-4,-10) & (-12,-10)}?(1)*\dir{<}?(0)*\dir{<};
      (4,4)*{};(12,4)*{} **\crv{(4,10) & (12,10)}?(1)*\dir{>}?(0)*\dir{>};
      (20,-11)*{\scs i};(10,-11)*{\scs j};
     }};
     (0,20)*{\xybox{
    (-4,-4)*{};(4,4)*{} **\crv{(-4,-1) & (4,1)}?(1)*\dir{>};
    (4,-4)*{};(-4,4)*{} **\crv{(4,-1) & (-4,1)};
     (4,4)*{};(-18,4)*{} **\crv{(4,16) & (-18,16)} ?(1)*\dir{>};
     (-4,-4)*{};(18,-4)*{} **\crv{(-4,-16) & (18,-16)} ?(1)*\dir{<}?(0)*\dir{<};
     (18,-4);(18,12) **\dir{-};
     (12,-4);(12,12) **\dir{-};
     (-18,4);(-18,-12) **\dir{-};(-12,4);(-12,-12) **\dir{-};
     (8,1)*{ \lambda};
     (-10,0)*{};(10,0)*{};
      (4,-4)*{};(12,-4)*{} **\crv{(4,-10) & (12,-10)}?(1)*\dir{<}?(0)*\dir{<};
      (-4,4)*{};(-12,4)*{} **\crv{(-4,10) & (-12,10)}?(1)*\dir{>}?(0)*\dir{>};
      (20,11)*{\scs i};(10,11)*{\scs j};
      (-20,-11)*{\scs i};(-10,-11)*{\scs j};
     }};
  \endxy \label{eq_dummy}
\end{equation}
If $\la i, \lambda \ra \leq -1$ then cap off the right most strands in both diagrams with $m=-\la i, \lambda\ra -1$ dots, simplify using cyclicity for dots and using \eqref{eq_downup_ij-genL} to get
\begin{equation}
   \xy 0;/r.17pc/:
  (0,0)*{\xybox{
    (4,-4)*{};(-4,4)*{} **\crv{(4,-1) & (-4,1)};
    (-4,-4)*{};(4,4)*{} **\crv{(-4,-1) & (4,1)};
    (4,4)*{};(-4,12)*{} **\crv{(4,7) & (-4,9)}?(1)*\dir{>};
    (-4,4)*{};(4,12)*{} **\crv{(-4,7) & (4,9)}?(1)*\dir{>};
    (4,-4)*{};(-18,-4)*{} **\crv{(4,-16) & (-18,-16)} ?(1)*\dir{<}?(0)*\dir{<};
     (-18,-4);(-18,14) **\dir{-};(-12,-4);(-12,12) **\dir{-};
     (12,12);(12,-14) **\dir{-}?(1)*\dir{>};
     (24,8)*{ \lambda+\alpha_j};
     (-10,0)*{};(10,0)*{};
     (-4,-4)*{};(-12,-4)*{} **\crv{(-4,-10) & (-12,-10)}?(1)*\dir{<}?(0)*\dir{<};
     (4,12)*{};(12,12)*{} **\crv{(4,18) & (12,18)};
     (-4,12)*{};(-12,12)*{} **\crv{(-4,18) & (-12,18)};
      (-20,11)*{\scs j};(-10,11)*{\scs i};
     (10,-11)*{\scs j}; (-12,6)*{\bullet}+(-3,1)*{\scs m};
     }};
  \endxy
\;\; = \;\;
b_{ij}(\lambda) \; t_{ji} \;
 \xy 0;/r.17pc/:
  (0,12);(0,-12) **\dir{-}?(.5)*\dir{>}+(2.3,0)*{};
  (14,6)*{\lambda+\alpha_j};(-2,-8)*{\scs j};
    (-12,0)*{\iccbub{m }{i}};
 \endxy
\end{equation}
Since $m=-\la i, \lambda\ra -1$ and degree zero bubbles are multiplication by $1$, simplifying the left-hand side using the KLR relation \eqref{eq_dot_slide_ij-gen} implies $b_{ij}(\lambda)=t_{ij}t_{ji}^{-1}$ for $\la i, \lambda \ra \leq -1$.

If $\la i, \lambda \ra \geq -1$ then glue a downward crossing onto the top of \eqref{eq_cyclic_cross-gen} and close off the left most strand with $m = \la i, \lambda+\alpha_i-\alpha_j\ra -1$ dots. Note that since $\la i,\alpha_j\ra\leq 0$ we have $m \ge 0$. Then arguing as above shows that $b_{ij}(\lambda)=t_{ij}t_{ji}^{-1}$. Thus, in all cases we have $b_{ij}(\lambda)=t_{ij}t_{ji}^{-1}$ for all $i,j \in I$ with $i \neq j$.

%
\section{Applications} \label{sec_apps}
%

%
\subsection{Cohomology of iterated flag varieties}
%

Extending results from \cite{CR,BFK}, an action of the 2-category $\Ucat(\mf{sl}_2)$ was constructed on categories of modules over cohomology rings~\cite{Lau1} (as well as equivariant cohomology rings~\cite{Lau2}) of partial flag varieties. This action was generalized to an action of $\Ucat(\mf{sl}_n)$ in \cite{KL3}, where it was used to prove nondegeneracy of the 2-category $\Ucat(\mf{sl}_n)$.

For these 2-representations on partial flag varieties, construction of an action of the KLR algebras was fairly straightforward. However, verification of the extended ${\mf{sl}}_2$ relations was rather complicated and involved comparing by hand certain bimodule maps. Our result \ref{thm:main} provides an immediate simplification, since one only needs to check the $[\E,\F]$ commutation relation for ${\mf{sl}}_2$ and then all the extended ${\mf{sl}}_2$ relations follow formally.

%
\subsection{Cyclotomic quotients}
%

For $\Lambda = \sum_{i \in I} \lambda_i \Lambda_i \in X^+$ consider the two-sided ideal $\cal{J}^{\Lambda}$ of $R=R_Q$ generated by elements
\[
\xy 0;/r.17pc/:
  (-3,9);(-3,-9) **\dir{-}?(0)*\dir{<}+(2.3,0)*{};
  (-3,2)*{\bullet};(-7.5,5)*{\scs \lambda_{i_1}};
  (-5,-6)*{\scs i_1};
 \endxy
\xy 0;/r.17pc/:
  (-3,9);(-3,-9) **\dir{-}?(0)*\dir{<}+(2.3,0)*{};
  (-5,-6)*{\scs i_2};
 \endxy  \cdots
\xy 0;/r.17pc/:
  (-3,9);(-3,-9) **\dir{-}?(0)*\dir{<}+(2.3,0)*{};
  (-5,-6)*{\scs i_m};
 \endxy
\]
over all $(i_1, i_2, \dots, i_m) \in I^m$ and $m \geq 0$. Define the cyclotomic quotient of the KLR algebra as the quotient
\begin{equation} \label{eq_def_RLambda}
R^{\Lambda} := R /\cal{J}^{\Lambda}.
\end{equation}

The cyclotomic quotient conjecture from \cite{KL,KL2} states that the category $\Proj(R^{\Lambda})$ of finitely generated graded projective $R^{\Lambda}$-modules categorifies the irreducible highest weight representation of $U_q(\mf{g})$ with dominant integral weight $\Lambda$.  This implies there is an isomorphism of $U_q(\mf{g})$-modules
\[
 V(\Lambda)_{\Z} \cong [\Proj(R^{\Lambda})],
\]
where $V(\Lambda)_{\Z}$ is the integral version of the irreducible highest weight representation for $\U(\mf{g})$ associated to $\Lambda$, and $[\Proj(R^{\Lambda})]$ is the Grothendieck ring of $\Proj(R^{\Lambda})$.

Parts of this conjecture have been proven in~\cite{BS,KR,BK2,LV,KR2,KP}. Recently Kang and Kashiwara~\cite{KasK} prove this conjecture by showing that $i$-induction and $i$-restriction functors induce a functorial action of $U_q(\mf{g})$ on the categories $\Proj(R^{\Lambda})$. They also define natural transformations between these functors giving an action of the KLR algebra~\cite[Section 6]{KasK}.

The category $\Proj(R^{\Lambda})$ decomposes as a direct sum of categories $\Proj( R^{\Lambda}(\alpha))$ for $\alpha=(i_1, i_2, \dots, i_m) \in I^m$ lifting the weight spaces of the irreducible representation $V(\Lambda)$.

\begin{conj}\label{conj:1}
For any $\alpha$ corresponding to a nonzero weight space of $V(\Lambda)$, the center of the ring $R^{\Lambda}(\alpha)$ is zero dimensional in negative degrees and one dimensional in degree zero.
\end{conj}

\begin{cor} \label{cor_cyclotomic}
Let $R^{\Lambda}$ denote the cyclotomic quotient of the KLR algebra associated to a symmetrizable Kac-Moody algebra $\mf{g}$ for any choice of scalars $Q$. Then, assuming Conjecture \ref{conj:1}, the 2-categories $\Proj(R^{\Lambda})$ form a 2-representation of $\Ucat_Q(\mf{g})$.
\end{cor}

\begin{rem}
A statement similar to Corollary~\ref{cor_cyclotomic} appears in \cite[Theorem 1.7 (version 6)]{Web1}. However, the 2-category used there is not quite $\Ucat_Q(\mf{g})$ whenever the coefficient $t_{ij} \neq 1$.
\end{rem}

%
\subsection{Derived categories of coherent sheaves}
%

\subsubsection{Cotangent bundles to Grassmannians}

In \cite{CKL1} the idea of a ``geometric categorical ${\mf{sl}}_2$ action'' was introduced. In \cite{CKL2} we showed that such an action induces a strong (for the standard choice of $Q$) 2-representation of ${\mf{sl}}_2$, as in section \ref{sec:strongcat}. In \cite{CKL1} and \cite{CKL3} a geometric categorical ${\mf{sl}}_2$ action was constructed on derived categories of coherent sheaves on cotangent bundles of Grassmannians $T^\star \mathbb{G}(k,N)$.

If we put all these results together, we obtain a strong 2-representation of  ${\mf{sl}}_2$ where:
\begin{itemize}
\item the objects $\l \in \Z$ are $DCoh(T^\star \mathbb{G}(k,N))$ where $N \in \Z^{\ge 0}$ is fixed, $0 \le k \le N$, $\l = N-2k$ and $DCoh(X)$ denotes the bounded derived category of coherent sheaves on $X$.
\item the 1-morphisms
\begin{eqnarray*}
\E^{(r)} &:& DCoh(T^\star \mathbb{G}(k+r,N)) \rightarrow DCoh(T^\star \mathbb{G}(k,N)) \text{ and } \\
\F^{(r)} &:& DCoh(T^\star \mathbb{G}(k,N)) \rightarrow DCoh(T^\star \mathbb{G}(k+r,N))
\end{eqnarray*}
are induced by certain correspondences
$$T^\star \mathbb{G}(k+r,N) \leftarrow W^r(k,N) \rightarrow T^\star \mathbb{G}(k,N)$$
by pulling back, tensoring with a line bundle and pushing forward (Fourier-Mukai transforms).
\item the action of the KLR algebra (a.k.a. the NilHecke algebra in this case) on the $\E$s is given by the formal construction in \cite{CKL2}.
\end{itemize}

Applying theorem \ref{thm:main} we find that:

\begin{thm}\label{thm:action} The geometric categorical ${\mf{sl}}_2$ action on $\bigoplus_{k=0}^N DCoh(T^\star \mathbb{G}(k,N))$ defined in \cite{CKL1} and \cite{CKL3} induces a 2-representation $\UcatD({\mf{sl}}_2)$ on $\mathcal{K}$ where:
\begin{itemize}
\item the objects of ${\mathcal{K}}$ are the derived categories of coherent sheaves $DCoh(T^\star \mathbb{G}(k,N))$,
\item the 1-morphisms of ${\mathcal{K}}$ are the Fourier-Mukai kernels in $DCoh(T^\star \mathbb{G}(k+r,N) \times T^\star \mathbb{G}(k,N))$,
\item the 2-morphisms of ${\mathcal{K}}$ are morphisms between these Fourier-Mukai kernels.
\end{itemize}
\end{thm}

\subsubsection{Nakajima quiver varieties}

In \cite{CKL4} one extends the notion from \cite{CKL2} to define a ``geometric categorical $\g$ action'' where $\g$ is an arbitrary simply laced Kac-Moody Lie algebra. One also constructs such a geometric categorical $\g$ action on the derived categories of coherent sheaves on Nakajima quiver varieties (this generalizes the action mentioned above on cotangent bundles of Grassmannians).

Although the analogue of the result from \cite{CKL2} has not been established we believe it should be true, at least when the Dynkin diagram of $\g$ is a tree. More precisely,

\begin{conj} If the Dynkin diagram of $\g$ is simply laced and contains no loops, then a geometric categorical $\g$ action induces a strong 2-representation of $\g$.
\end{conj}

The difficult part in proving the conjecture above is showing that the KLR algebra acts on $\E$s. Once we know this the conjecture follows.  Then theorem \ref{thm:main} implies that there is a 2-representation of $\UcatD(\g)$ on derived categories of coherent sheaves on Nakajima quiver varieties of type $\g$ (just like theorem \ref{thm:action} when $\g = {\mf{sl}}_2$).

%
\bibliographystyle{plain}
\bibliography{bib_unique}

%

%
\end{document}